\documentclass[11pt,a4paper]{amsart}

\usepackage[margin=2cm]{geometry}

\let\oldtocsection=\tocsection
\let\oldtocsubsection=\tocsubsection
\let\oldtocsubsubsection=\tocsubsubsection
\makeatletter
\def\subsection{\@startsection{subsection}{3}%
  \z@{.5\linespacing\@plus.7\linespacing}{.5\linespacing}%
  {\bf}}
\def\subsubsection{\@startsection{subsubsection}{3}%
  \z@{.5\linespacing\@plus.7\linespacing}{.5\linespacing}%
  {\it}}
\makeatother

\renewcommand{\tocsection}[2]{\hspace{0em}\oldtocsection{#1}{#2}\textbf}
\renewcommand{\tocsubsection}[2]{\hspace{1em}\oldtocsubsection{#1}{#2}}
\renewcommand{\tocsubsubsection}[2]{\hspace{2em}\oldtocsubsubsection{#1}{#2}}
\renewcommand{\paragraph}[1]{\noindent\underline{#1}}
\newcommand{\lambdamax}{\lambda_{\mathrm{max}}}
\newcommand{\calF}{\mathcal{F}}

\makeatletter
\def\@setthanks{\def\thanks##1{\par##1}\thankses}
\makeatother

\usepackage[T1]{fontenc}
\usepackage[utf8]{inputenc}
\usepackage[english]{babel}
\usepackage{mdframed}
\usepackage{color}
\usepackage{wrapfig}
\usepackage{placeins}
\usepackage{subfigure}
\usepackage{enumerate}
\usepackage{tabu}
\usepackage{fullpage}
\usepackage[squaren]{SIunits}
\usepackage{graphicx}
\usepackage{natbib}
\usepackage{tikz}
\usepackage{pgfplots}
\pgfplotsset{compat=1.17}
\usepackage{epstopdf}
\usepackage{epsfig}
\usepackage{pifont}

\usepackage[colorlinks=true,bookmarks=false,linkcolor=blue,urlcolor=blue,citecolor=blue,breaklinks=true]{hyperref}
\usepackage{breakcites}
\usepackage{tikz}
\usepackage{tikz-qtree}
\usepackage{eurosym}
\usepackage{amsmath}
\numberwithin{equation}{section}
\usepackage{amssymb}
\usepackage{amsthm}
\usepackage{mathrsfs}
\usepackage{blkarray}
\usepackage{dsfont}% use $\mathds{1}$
\usepackage{comment}
\usepackage{relsize}

\usepackage[nameinlink]{cleveref}
\usepackage{rotating}
\newcommand{\minimize}{\textrm{minimize}}

\newcommand{\N}{\mathbb{N}}

\newcommand{\R}{\mathbb{R}}

\newcommand{\M}{\mathcal{M}}
\newcommand{\calM}{\mathcal{M}}

\newcommand{\calO}{\mathcal{O}}

\definecolor{purple}{rgb}{0.74, 0.2, 0.64}

\newcommand{\calD}{\mathcal{D}}

\newcommand{\D}{\mathrm{d}}

\newcommand{\ddt}{\frac{\mathrm{d}}{\mathrm{d}t}}

\newcommand{\calE}{\mathcal{E}}
\newcommand{\calL}{\mathcal{L}}

\newcommand{\p}{^{\perp}}

\newcommand{\transpose}{^\top}
\newcommand{\T}{^\top}

\newcommand{\St}{\mathrm{St}}

\newcommand{\grad}{\mathrm{grad}}
\newcommand{\Proj}{\mathrm{Proj}}

\newcommand{\Hess}{\mathrm{Hess}}

\newcommand{\Nrm}{\mathrm{N}}

\newcommand{\Sym}{\mathrm{Sym}}
\newcommand{\sym}{\operatorname{sym}}
\newcommand{\Skew}{\mathrm{Skew}}
\renewcommand{\skew}{\operatorname{skew}}

\newcommand{\Rm}{\mathbb{R}^m}
\newcommand{\Rn}{\mathbb{R}^n}

\newcommand{\Rmn}{\mathbb{R}^{m\times n}}

\newcommand{\Rpp}{\mathbb{R}^{p\times p}}

\newcommand{\rank}{\operatorname{rank}}

\newcommand{\inner}[2]{\left\langle{#1},{#2}\right\rangle}
\newcommand{\Id}{\mathrm{Id}}
\newcommand{\I}{\mathrm{I}}
\newcommand{\sigmamin}{\sigma_\mathrm{min}}

\newcommand{\sigmabar}{\underline{\sigma}}

\newcommand{\lambdamin}{\lambda_\mathrm{min}}

\newcommand{\Rnp}{{\mathbb{R}^{n \times p}}}
\newcommand{\Rnpstar}{{\mathbb{R}_{*}^{n \times p}}}

\usepackage{bm}

\usepackage{algorithm}
\usepackage{algpseudocode}
\usepackage{todonotes}

\newcommand{\inv}{^{-1}}

\newcommand{\aref}[1]{\hyperref[#1]{A\ref{#1}}}
\newcommand{\norm}[1]{\left\|#1\right\|}

\newcommand{\Fnorm}[1]{\left\|{#1}\right\|_\mathcal{F}}

\usepackage{apxproof}
\newtheorem{theorem}			     {Theorem}	[section]
\newtheorem{corollary}	  [theorem]	 {Corollary}	
\newtheorem{lemma}	      [theorem]  {Lemma}

\newtheorem{assumption} {A\ignorespaces}%[section]
\newtheorem{remark}			     {Remark}	[section]
\newtheoremrep{example}			     {Example}
\newtheoremrep{proposition} [theorem] {Proposition}

\crefname{assumption}{}{}
\crefformat{assumption}{#2A#1#3}
\usepackage{listings}
\usepackage{textcomp}
\definecolor{listinggray}{gray}{0.9}
\definecolor{lbcolor}{rgb}{0.9,0.9,0.9}
\usepackage{mathtools}
\mathtoolsset{showonlyrefs=false}
\usepackage{tcolorbox}

\newcommand{\Trm}{\mathrm{T}}
\newcommand{\<}{\ensuremath \leq}
\renewcommand{\>}{\ensuremath \geq}

\newcommand{\gE}{g^{\mathcal{E}}}

\newcommand{\normalE}{\mathrm{N}^\mathcal{E}}

\newcommand{\DD}{\mathrm{D}}
\newcommand{\Ip}{\mathrm{I_p}}
\newcommand{\StX}{\mathrm{St}_{X\T X}}
\newcommand{\rmN}{\mathrm{N}}
\newcommand{\rmT}{\mathrm{T}}
\newcommand{\Mx}{\mathcal{M}_x}
\newcommand{\Tx}{\mathrm{T}_x}
\newcommand{\Nx}{\mathrm{N}_x}
\newcommand{\Nxg}{\mathrm{N}_x^g}

\newcommand{\saferegion}{\mathcal{R}}
\newcommand{\cset}{\mathcal{F}}
\newcommand{\gradgx}{\mathrm{grad}_{\mathcal{M}_x}^g}

\renewcommand{\c}{c}
\newcommand{\normal}{d_N}
\newcommand{\tangent}{d_T}

\crefname{figure}{Figure}{Figs.}
\crefname{equation}{}{}

\newcommand{\Nxk}{\mathrm{N}_{x_k}}
\newcommand{\Txk}{\mathrm{T}_{x_k}}
\newcommand{\Mxk}{\mathcal{M}_{x_k}}

\newcommand{\flow}{f_{\mathrm{low}}}

\date{\today}

\newcommand{\revision}{\textcolor{blue}}

\begin{document}
\vspace*{-0.5em}

\title{The Riemannian Landing Method: \\
from projected gradient flows to SQP}
%geometry of the landing method}
%for equality constrained optimization}
\author{Florentin Goyens$^{1*}$, Florian Feppon$^{2}$}
\thanks{\hspace*{0.3em}\textsuperscript{$*$} Corresponding author. Email:
    \texttt{goyensflorentin@gmail.com}. \\[2ex]
\textbf{Acknowledgements.} This work was supported by the Fonds de la Recherche
Scientifique-FNRS under Grant no T.0001.23. F. Feppon was supported by the
Flanders Research Foundation (FWO) under Grant G001824N.
\hspace*{\parindent}
}

\maketitle

%%%%%%%%%%%%%%%%%%%%%%%%%%%%%%%%%%%%%%%%%%%%%%%%%%%%%%%%%%%%%%%%%%%%%%%%%%%%%%%%%%%%%%%%%%%%%%%%
\vspace{-2em}
\begin{center}
    \emph{\textsuperscript{1} ICTEAM Institute, UCLouvain, Louvain-la-Neuve, Belgium}     \\
\emph{\textsuperscript{2} NUMA Unit, Department of Computer Science, KU Leuven,
    Belgium.}
\end{center}
%%%%%%%%%%%%%%%%%%%%%%%%%%%%%%%%%%%%%%%%%%%%%%%%%%%%%%%%%%%%%%%%%%%%%%%%%%%%%%%%%%%%%%%%%%%%%%%%

\begin{abstract}
%In recent years, an algorithm for optimization under constraints---sometimes called
%the landing algorithm---has received increased attention. In this work, we uncover
%several hidden connections between the landing and existing method for constrained
%optimization. We also define the landing in its full generality, for any choice of
%metric in the ambient space. We provide the first adaptive step size selection
%procedure for the landing, with a line search filter landing. The landing has mostly
%been considered for optimization under orthogonality constraints. We describe several
%choices of ambient metric for orthogonality constraints and relate it with existing
%version of the landing. We also numerically charactirize the relevance of the landing
%against feasible Riemannian optimization for orthogonality constraints. \nl

Landing methods have recently emerged in Riemannian matrix optimization as
efficient schemes for handling nonlinear equality constraints without resorting
to costly retractions. These methods decompose the search direction into tangent
and normal components, enabling asymptotic feasibility while maintaining
inexpensive updates. In this work, we provide a unifying geometric framework
which reveals that, under suitable choices of
Riemannian metric, the landing algorithm encompasses
several classical optimization methods such as projected and
null-space gradient flows, Sequential Quadratic Programming (SQP), and a certain
form of the augmented Lagrangian method. In particular, we
show that a quadratically convergent landing method essentially reproduces
the quadratically convergent SQP method. These connections also
allow us to propose a globally convergent landing method using  
adaptive step sizes. 
The backtracking line search satisfies an Armijo condition
on a merit function, and does not require prior knowledge of Lipschitz constants. 

Our second key contribution is to analyze landing methods through  
a geometric parameterization
of the metric in terms of fields of oblique projectors and associated metric
restrictions. This viewpoint disentangles the roles of orthogonality, tangent
and normal metrics, and elucidates how to design the metric to obtain explicit
tangent and normal updates. For matrix optimization, this framework not only recovers
recent constructions in the literature for problems with
orthogonality constraints, but also provides systematic guidelines for designing
new metrics that admit closed-form search directions.
\end{abstract} 

 \par
% 45M05 Gradient methods
% 90C30 Nonlinear programming 
% 53C21 Methods of Riemannian geometry
% 65K10 Numerical optimization and variational techniques
% 49M29 Methods for constrained problems
\medskip
\bigskip
\hrule
\tableofcontents
\vspace{-0.5cm}
\hrule
\medskip
\bigskip

\section{Introduction}\label{sec:intro}

The minimization of a function subject to nonlinear constraints is a central
problem across engineering, scientific computing, and machine learning. In
recent years, the concept of \emph{landing method} has emerged in the
Riemannian matrix optimization community after the works of
\cite{ablin2022fast} and \cite{gao_optimization_2022}.  The essential
idea of the landing algorithm is to decompose the search direction 
\begin{equation}
\label{eqn:atgi3}
d = d_T + d_N,
\end{equation}
into a \emph{tangent} component $d_T$, which is responsible for decreasing the
objective function without worsening the constraint violation at first order, and a
\emph{normal} component $d_N$, whose role is to progressively drive iterates
toward the constraint manifold by reducing infeasibility.

In the context of matrix optimization, this approach enables one to enforce
orthogonality constraints asymptotically, thus avoiding the computationally
expensive \emph{retraction} operations--- i.e., nonlinear projections onto the
feasible set---required by standard feasible Riemannian algorithms
\citep{absil2008optimization,absil_projection-like_2012}. Since evaluating the normal term
$d_N$ is typically cheaper than computing retractions, the
method is appealing and makes geometric optimization potentially viable for
large-scale problems. Applications include imposing the orthogonality 
of weight matrices
during the training of Large Language Models (LLMs) or other resource-intensive
machine learning tasks \citep{hu2025ostquant}.

\medskip

Consider the equality-constrained optimization problem 
\begin{equation}\label{eq:P}\tag{P}
    \left\{\begin{aligned}
        \min_{x\in \mathcal{E}} & \quad f(x)\\
    \text{s.t.} & \quad c(x)=0,
    \end{aligned}\right.
\end{equation}
where
$f\colon\mathcal{E}\to \R$ and $c\colon\mathcal{E}\to \mathcal{F}$ are 
continuously differentiable—possibly nonconvex—mappings between
finite-dimensional vector spaces 
$\mathcal{E}$ and $\mathcal{F}$ with $\dim(\mathcal{F})<\dim(\mathcal{E})$. 
We define the Riemannian landing method (RLM) as 
\begin{subequations}\label{eq:landing}
\begin{align}
x_{k+1} &= x_k + \alpha_k (d_T(x_k)+d_N(x_k)),\\
d_T(x) &= -\grad_{\mathcal{M}_x}^{g} f(x),  \\
d_N(x) &= -\nabla_g \psi(x)\label{eqn:3dfjy}.
\end{align}
\end{subequations}
where $\alpha_k>0$ is an adaptive step size,
$g$ is a Riemannian metric on $\calE$, and 
$$\psi(x)=	\frac12	\norm{c(x)}^{2}_{\mathcal{F}}$$
is the infeasibility measure.
The tangent direction $-\grad_{\mathcal{M}_x}^{g} f(x)$ corresponds
%respectively
to the negative \emph{Riemannian constrained gradient} of $f$ on the level curve 
\begin{equation}
\label{eqn:od6i9}
\M_{x}:=\{ y\in\calE\,\colon\, \c(y)=\c(x) \},
\end{equation}
while the normal direction $-\nabla_g \psi(x)$ is
the negative \emph{unconstrained Riemannian gradient} of the infeasibility measure
$\psi(x)$. 
%The   role  of the tangent term is to
%decrease the objective function $f$ without increasing the
% violation of the
%constraints, while the purpose of the normal term is to 
%`land' the iterates
%$x_k$ back toward feasibility by descending along
% the constraint violation
%$c(x)$. 
The gradients are computed with respect to a
(user-defined) metric $g$ that endows $\mathcal{E}$ with a Riemannian
structure. These concepts are defined formally in Section~\ref{sec:geometry}. 
When $\mathcal{E}=\R^{n}$ and $\mathcal{F}=\R^{m}$ are equipped with
the standard Euclidean inner product, 
denoting 
$\DD c(x)=(\partial_j c_i(x))_{\substack{1\<i\<m \\ 1\<j\<m }}$ the
Jacobian matrix of $c(x)$,
one obtains $d_N(x) = - \DD c(x) \,
c(x)$ and 
\begin{equation}
\label{eqn:7jkb4}
d_T(x) = - \bigl(\I_n - \DD c(x)^\top
(\DD c(x)\DD c(x)^\top)^{-1}\DD c(x)\bigr)\nabla f(x),
\end{equation}
the orthogonal projection of $\nabla f(x)$ onto the tangent space of
$\mathcal{M}_x$.
\medskip 

\subsection*{Contributions}
\textbf{(i) Unification}
We propose the Riemannian landing method~\cref{eq:landing}, a unifying framework in which a general metric $g$ defines both the tangent and normal terms, whereas existing landing methods typically use a Riemannian metric only for the tangent term.
Interestingly, the Riemannian landing method is also related to numerous 
other optimization schemes. 
Section~\ref{sec:14l4s} gives a
detailed bibliographical account of the
historical development of methods exploiting the decomposition of the search
direction into tangent and normal components, including recent landing methods.
Related algorithms include: 
\begin{enumerate}
	\item \textbf{projected gradient flows with `pseudoinverse' normal step}: 
Section~\ref{sec:14l4s} explains that,
 \emph{under suitable
choices of the metric $g$}, the Riemannian 
landing method~\eqref{eq:landing} is
equivalent to existing algorithms that have emerged under different names in
other areas of optimization, such as the \emph{Corrective Gradient Projection}
method \citep{frost_algorithm_1972,gangi_corrective_1976}, \emph{projected gradient flows}
\citep{tanabe_geometric_1980,yamashita_differential_1980,schropp_dynamical_2000},
and \emph{null-space gradient flows}
\citep{feppon_null_2020,feppon_density-based_2024}.
These algorithms 
usually use  a pseudoinverse or `Newton'-like term to define 
the normal
step, an idea also considered for landing iterations in \citep{schechtman_orthogonal_2023}. 
We show in Proposition~\ref{prop:pseudoinverse_E} that the pseudoinverse
formulation~\eqref{eq:pseudoinverse_landing} is equivalent
to~\eqref{eq:landing}.
\item \textbf{SQP and ALM}: we show
in \cref{subsec:bxujs} that the Riemannian landing scheme~\cref{eq:landing} includes, as particular
cases, the Sequential Quadratic Programming (SQP) algorithm
\citep{boggs_sequential_1995} in its basic form (without trust regions), as well
as a simple augmented Lagrangian method \citep{tapia_diagonalized_1977} with a
particular update rule for the multipliers.
\end{enumerate}
 Although the introduction of
landing methods can thus be viewed as a resurgence of
classical methods, the questions of how to choose the metric
 $g$ in \cref{eq:landing} in order to derive practical schemes 
 for matrix optimization, and of how to obtain global convergence guarantees, 
 have been explored only  relatively recently, 
and open new perspectives on existing algorithms.

\medskip 

\textbf{(ii) Metric design}
Our second key contribution is to make explicit how the tangent and normal
components of~\eqref{eq:landing} depend on the metric $g$. Section~\ref{sec:euclidean_landing} introduces
the following viewpoint to design the metric~$g$: 
\emph{first}, associate a
normal space \(\Nx\Mx\) to every tangent space \(\Tx\Mx\) of the constraint
manifold \(\mathcal{M}_x\). \emph{Then}, define the restriction
of the metric to the tangent and normal spaces (see \cref{eqn:oitse}). Mathematically,
this requires specifying for each \( x \in \Mx\) an \emph{oblique} projector onto the tangent space at \(x\). 
We express the
directions \(d_T(x)\) and \(d_N(x)\) explicitly in
terms of the oblique projector and these metric restrictions
(\cref{prop:kk5n4}). 
This characterization of the metric plays a key role 
in establishing a link between existing algorithms 
((i)~Unification), and also has direct applications for the 
design of new landing methods ((iii)~Algorithmic consequences). 

\medskip

\textbf{(iii) Algorithmic consequences} Using our metric
construction, we show how to design the metric to make~\cref{eq:landing} quadratically convergent,
 or to facilitate the computation of
the tangent and normal components.
The connection with the SQP framework also allows us to propose 
a globalization procedure with adaptive step sizes. These contributions are
summarized in the next three items.
\begin{itemize}
	\item \textbf{`All roads lead to
SQP'} 
 With the aim of identifying conditions under which the scheme
 \cref{eq:landing} can be quadratically convergent, 
we show that when setting the tangent term \(d_T(x)\) to a
Riemannian Newton step, the
Riemannian landing scheme achieves quadratic convergence (for unit step sizes
$\alpha_k=1$)
\emph{only} for a carefully designed 
normal step \(d_N(x)\) such  that 
\cref{eq:landing} essentially reproduces SQP 
iterates~(\cref{prop:94jvx}).
In
reference to the leitmotiv ``all roads lead to Newton'' 
proclaimed in
\citep{absil_all_2009}, this observation leads us to  state here that \emph{`all roads lead to
SQP'} 
for quadratically converging infeasible methods.
	\item \textbf{Design of practical updates}
	For matrix optimization with orthogonality constraints,  
computing the tangent
direction in the Euclidean metric requires to solve a linear system involving
$\DD c(x)\DD c(x)^\top$,
which 
amounts to solve a  Sylvester 
equation.
We show how to identify several natural choices for 
the projector mapping and metric,
which lead to 
explicit formulas
for $d_T(x)$ and $d_N(x)$, thereby enabling
efficient evaluations (based only on matrix products).  
%This viewpoint clarifies why the metrics considered in recent
%publications, e.g., \citep{ablin2022fast,goyens_geometric_2026}, lead to explicit
%formulas for $d_T(x)$ and $d_N(x)$.
\item \textbf{Global convergence}
We \emph{propose an adaptive procedure
for selecting the step size} $\alpha_k$, based on an Armijo condition for the decrease of a
merit function, that guarantees global convergence of the algorithm
\cref{eq:landing} toward a feasible point satisfying first-order optimality
conditions. The procedure and its analysis rely heavily on standard techniques 
for the globalization of SQP~\citep{nocedal2006numerical,curtis2024worst}, and
does not require the knowledge of global Lipschitz constants of the objective
function
$f$ or of the constraint $c(x)$.
Prior works proving the convergence of 
a landing iteration assumed constant or scheduled step sizes, where 
 global convergence is obtained 
for a  fixed step size smaller than some unknown constant
\citep{schechtman_orthogonal_2023,zhang2024retractionfree,song2025distributed}. 
Recently, \cite{shi_adaptive_2025} proposed 
an adaptive step size method also based on line search, 
but still requiring the knowledge of global
Lipschitz constants.
Our proposed method, which does not require the knowledge of these constants, 
fills thus an important gap, since 
landing schemes with constant step sizes can be 
very sensitive to step size tuning in practice. 
\end{itemize}

\subsection*{Contents}
The article is organized as follows. \Cref{sec:14l4s} provides  a detailed
account of historical developments of constrained optimization methods closely
related to the landing approach.
\Cref{sec:geometry} reviews the background in differential geometry and
introduces our notations. 
\Cref{sec:RLM} introduces the parameterization of the metric
 by a choice of oblique projector, together
with the tangent and normal metric restrictions, and uses this
to link the landing
method with existing projected gradient flows approaches. 
\Cref{sec:comparison} 
establishes new connections between the Riemannian landing method and 
existing algorithms. 
%We first show that 
%\cref{eq:landing} encompasses a
%version of the augmented Lagrangian method with least-squares multiplier
%update, before 
%proving that the Sequential Quadratic Programming (SQP)
%method is a particular instance of the Riemannian landing method. 
%Then,  
%we deepen known connections between SQP and Riemannian Newton method, formally
%establishing that 
%the landing method \cref{eq:landing} is locally quadratically convergent if it 
%essentially reproduces quadratically convergent SQP iterates, 
In~\Cref{sec:globalization}, we establish global convergence 
of the Riemannian landing method with adaptive step sizes
(Algorithm~\ref{algo:LSLanding}). 
%The framework adapts 
%well-known ideas for the gloablization of SQP methods based on a merit function.
Finally, \Cref{sec:ym3u0} presents several metric designs leading to explicit
gradient formulas for matrix optimization
under orthogonality constraints.

\section{Historical developments of landing-related methods}
\label{sec:14l4s}

To the best of our knowledge, the first trace of a geometric algorithm for 
equality constraints is the \emph{Gradient Projection Method}
formalized by \cite{rosen_gradient_1960,rosen_gradient_1961}. For nonlinear
constraints, this method projects at every step the gradient of the objective
function onto the tangent space to the constraint set---the null space of the
linearized constraints---using \cref{eqn:7jkb4}, or onto the tangent cone in the
case of inequality constraints. A similar
algorithm was independently devised by \cite{booker_multiple-constraint_1971}
for geophysics signal processing with equality constraints. Shortly thereafter,
the convergence rate of the method, considering geodesic retractions and
equality constraints, was formally analyzed by \cite{luenberger_gradient_1972}.

However, finite steps on curved manifolds induce numerical infeasibility, which raised
the need for an explicit correction term. This led to the development of the
\emph{Corrective Gradient Projection (CGP)} method
\citep{frost_algorithm_1972,gangi_corrective_1976}, which augments the tangent
direction with a normal component correcting deviations from the
 \revision{constraints---introducing a decomposion as in~\eqref{eq:landing}.}
This normal component, 
\begin{equation} \label{eqn:3g1yn} d_N(x)=-\DD c(x)\T
(\DD c(x)\DD c(x)\T)^{-1} c(x), \end{equation} 
is the `pseudoinverse' solution for finding a zero of the linearization of the constraints.

\medskip

In parallel with the development of CGP, \cite{tanabe1974algorithm} proposed
to interpret Rosen’s gradient projection method as the discretization of the
dynamical system $\dot x=d_T(x)$. Later, \cite{yamashita_differential_1980}
introduced a corrective version of this projected gradient flow,
\begin{equation}
\label{eqn:gho94}
\begin{aligned}
\dot x &= d_T(x) + d_N(x)\\
&= -\left(\I_n-\DD c(x)\T \! \left(\DD c(x)\DD c(x)\T\right)^{-1}\!\DD 
c(x)\right) \!\nabla f(x)
 - \DD c(x)\T \! \left(\DD c(x)\DD c(x)\T\right)^{-1}\!c(x),
\end{aligned}
\end{equation}
rediscovering the pseudoinverse step of 
\cite{frost_algorithm_1972}. This flow was later considered by
\cite{evtushenko_stable_1994} to develop a barrier method for linear
programming. Subsequently, \cite{schropp_dynamical_2000} proved that any bounded
solution of \cref{eqn:gho94} converges to a steady-state equilibrium, and 
stable equilibria coincide with local minimizers of \cref{eq:P}.  
\cite{schropp_dynamical_2000} suggested that inequality
constraints can be handled using \cref{eqn:gho94} 
by introducing slack variables to convert them into equality constraints; however, 
this produces an exponential growth in false equilibrium points as the number of
inactive constraints increases 
\citep{jongen_complexity_2003}.
 To overcome this,
\cite{jongen_constrained_2004} and \cite{shikhman_constrained_2009} proposed variations of
\cref{eqn:gho94} (without the corrective term), yielding modified gradient flows 
that implicitly incorporates the effects of slack variables.

\medskip

A distinct differential-equation-based approach for problems with equality and
inequality constraints emerged in shape and topology optimization under the name
\emph{null space gradient flows} or \emph{Null Space Optimizer}
\citep{feppon_null_2020}. This method extends the projected flow
\cref{eqn:gho94} to general constrained problems and is compatible with the
infinite-dimensional setting of shape optimization based on Hadamard's boundary
variation method \citep{henrot_shape_2018,allaire_chapter_2021}. This
compatibility is possible thanks to  the
central role of the metric in converting differential into gradients and the
interpretation of shape updates as retractions on an abstract manifold. Its
tangent direction is the projection of the objective gradient onto the tangent
cone of active or violated constraints, akin to the gradient projection of
\cite{rosen_gradient_1961}, here computed by solving a quadratic program. The absence of
hyperparameters makes the method particularly effective for topology and shape
optimization \citep{feppon_topology_2020-1}. An alternative construction of
tangent and normal terms for equality-constrained shape optimization was
proposed in \cite{wegert_hilbertian_2023}. A variant
of the Null Space Optimizer for dealing with 
constraints with sparse Jacobian matrix was 
later introduced in \cite{feppon_density-based_2024}.
Let us mention that, before these
lines of work, related gradient projection ideas had appeared in topology
optimization: \cite{yulin_level_2004} projected shape gradients on tangent
cones, while \cite{barbarosie_gradient-type_2020} proposed a method close to the
Null Space Optimizer for equality constraints, together with a heuristic
active-set strategy to handle inequalities.

\medskip

More recently, the \emph{landing} method was
introduced in \citep{ablin2022fast,gao_optimization_2022} as a retraction-free
method for
matrix optimization with orthogonality constraints. 
This contrasted with 
the so far predominant Riemannian optimization methods,
such as Riemannian gradient descent or Riemannian Newton methods on
the constraint manifold, which leverage intrinsic differential geometry and
retractions~\citep{absil2008optimization,absil_projection-like_2012}.
   
Although 
 the landing method introduced in \citep{ablin2022fast} turns out to be very
 close in spirit with 
 the projected gradient flow \cref{eqn:gho94} of
 \cite{yamashita_differential_1980}, 
 there is a key difference: 
 the normal term in \cref{eq:landing} 
 is not a pseudoinverse step, but the
unconstrained gradient of a constraint violation. However,
we show below that both coincide upon a suitable metric choice
(\cref{prop:pseudoinverse_E}).

The landing algorithm was extended to general
equality constraints and stochastic settings by
\cite{schechtman_orthogonal_2023}, who renamed it the \emph{Orthogonal Directions
Constrained Gradient Method}. The authors prove optimal complexity
guarantees for the discretized algorithm, namely $O(\epsilon^{-2})$ iterations
to reach an $\epsilon$-KKT point. To our knowledge, this is the first
convergence proof for a discrete landing algorithm; earlier
guarantees concerned the continuous system \cref{eqn:gho94}, assuming that
the chosen discretization approximate the ODE trajectories sufficiently well. This analysis was
further strengthened in \citep{vary_optimization_2024} for the stochastic setting,
using a smooth merit function to prove global convergence even in the presence
of non-tangent random errors.

More recently, \cite{zhang2024retractionfree} applied the landing method to
accelerate low-rank adaptation in fine-tuning large language models, proving
convergence to the Stiefel manifold using a sufficiently small constant
step size. \cite{goyens_geometric_2026} derived landing flows on the Stiefel
manifold for a family of metrics extending the $\beta$-metric of
\cite{huper_lagrangian_2021,mataigne_efficient_2025}.
\cite{song2025distributed} adapted the landing algorithm for distributed stochastic
optimization on the Stiefel manifold. Lately, \cite{shi_adaptive_2025} proposed
an extension called \emph{Adaptive Directional Decomposition Methods} to
incorporate inequality constraints. The tangent term is obtained by solving a
quadratic program, similarly to the null-space gradient-flow strategy of
\cite{feppon_null_2020} for equality and inequality constrained problems.
They show global convergence guarantees using an adaptive step size rule, which 
relies on the knowledge of Lipschitz constants. 
\bigskip

\noindent
\textbf{Links with null-space projections and SQP}

\medskip 
\noindent
The use of null-space projections to compute descent directions on the
constraint manifold is certainly universal. It is therefore not surprising that
formulas analogous to \cref{eqn:7jkb4,eqn:3g1yn} have appeared in numerous
alternative optimization frameworks.
% including robotics, where null-space
%control is used to exploit actuation redundancy and allow robots to perform
%multiple tasks simultaneously \citep{dietrich_overview_2015}.

Most prominently, the decomposition \cref{eqn:atgi3} plays a central role in
so-called \emph{null-space methods}, where it is employed to 
%isolate the objective
%descent and 
construct reduced-Hessian approximations within the tangent space 
\citep{nocedal_projected_1985,yuan2001null,nie_null_2004,
biros_parallel_2005,nocedal2006numerical,berahas_stochastic_2024}. 
These ideas lead to expressions
directly comparable to the tangent and normal components in \cref{eqn:gho94}.
Such decompositions into null-space and range-space steps have become standard
in modern SQP implementations, most notably in the SNOPT software, where they
are used both to reduce the dimension of the quadratic programming subproblem
and to exploit reduced Hessians
\citep{gill_snopt_2002,gill_sequential_2012,gill_users_2015,
gharaei_integrated_2023,fang_fully_2024}. They have also been exploited in
interior-point methods \citep{liu_null-space_2010,nocedal_interior_2014}.
Finally, we note that the least-squares multiplier $\lambda(x)=-(\DD
c(x)\DD c(x)^\top)^{-1}\DD c(x)\nabla f(x)$ is likewise employed in
certain variants of the augmented Lagrangian method
\citep{tapia_diagonalized_1977,conn_globally_1991}, making these variants
equivalent to landing-type algorithms (see \cref{prop:6nqa8}).

\smallskip

Finally, connections between SQP and differential-equation-based methods have
also been pointed out repeatedly. For instance, \cite{schropp_dynamical_2000}
showed that standard SQP schemes can be interpreted as variable step-size Euler
discretizations of preconditioned versions of the flow \cref{eqn:gho94}.
Precise relationships between the Riemannian Newton method and
 the quadratically convergent 
SQP method were established in \cite{miller_newton_2005} and
\cite{absil_all_2009}, 
and later exploited in
\cite{mishra2016riemannian} to design metric preconditioners accelerating the
convergence of feasible Riemannian optimization methods on matrix manifolds.
Recently, an analysis of the rate of convergence of the SQP algorithm with a
constant quadratic term was derived by \cite{bai2018analysis} through the
interpretation as a Riemannian
optimization method. 

\section{Riemannian geometry and notation}
\label{sec:geometry}

We consider the equality-constrained optimization problem \cref{eq:P} on the
Euclidean spaces $\calE$ and $\cset$. We denote  by $\inner{\cdot}{\cdot}_\calE$
and $\inner{\cdot}{\cdot}_\cset$ their respective inner products, and by
$\norm{\cdot}_\calE$ and $\norm{\cdot}_\cset$ their associated norms. The
feasible set is denoted by $    \calM := \left\lbrace x\in \calE\colon \c(x)
=0\right \rbrace$. 
For intuition, one may think of these spaces as
$\mathcal{E}\simeq \R^{n}$ and $\mathcal{F}\simeq \R^{m}$.
The more general Euclidean-space formalism is utilized in
\cref{sec:ym3u0}, where $\mathcal{E}$ and $\mathcal{F}$ are matrix subspaces
arising in optimization problems with orthogonality constraints.
If $V\subset
\mathcal{E}$ is a subspace of $\mathcal{E}$, we denote by
$V^{\perp,\mathcal{E}}\subset \mathcal{E}$ its
 orthogonal complement with respect
to the Euclidean metric of $\mathcal{E}$.

\medskip 

Throughout the paper, we denote by $\DD c$ the differential of the constraint
function $c$. The linear map $\DD c(x)\,:\,\mathcal{E}\to \mathcal{F}$ is such
that 
\[
c(x+h)=c(x)+\DD c(x)h+o(h) \text{ with
}\frac{\norm{o(h)}_{\mathcal{F}}}{\norm{h}_{\mathcal{E}}}\rightarrow 0 \text{ as
}\norm{h}_{\mathcal{E}}\rightarrow 0.
\] 
We denote by $\mathcal{D}$ the set of points $x\in \mathcal{E}$ at which the
differential of the constraints is full rank:
\begin{align}
\mathcal{D} = \{ x \in \calE \colon \rank(\DD \c(x)) = m \},
\label{eq:D}
\end{align}
which is also called the Linear Independence Constraint Qualification
(LICQ) condition.
We recall that the set $\mathcal{M}_x$  defined in \cref{eqn:od6i9} is a smooth
manifold when it is included in $\calD$. Every set $\Mx$ is a level curve of the
constraint function $\c$, called a \emph{layer
manifold}~\citep{goyens2024computing}.
The tangent
space of $\Mx$ at $x\in \calD$ is the null space of the differential of the
constraint:
\begin{equation}
 \Tx\Mx = \ker(\DD c(x)).
\end{equation}
 
\bigskip 
\noindent
\textbf{Riemannian metric and normal space}

\smallskip 
Let $g$ be a Riemannian metric on $\mathcal{E}$, that is a smooth family of
inner products $x\mapsto g_x(\cdot,\cdot)$ on $\mathcal{E}$. With a small abuse
of notation, we often drop the subscript $x$ when referring to the metric. The
metric is represented by a smoothly varying family of operators
$G(x):\calE\to\calE$, symmetric positive definite with respect to the Euclidean
product $\langle \cdot,\cdot  \rangle_{\mathcal{E}}$, such that
\begin{equation}
\label{eqn:4x5li}
g(\xi,\zeta) = \inner{G(x)\,\xi}{\zeta}_\calE,
\qquad \text{ for all }\xi, \zeta\in\calE.
\end{equation}
The corresponding norm is denoted by $\|u\|_{g}=\sqrt{g(u,u)}$ for any $u\in
\mathcal{E}$. We sometimes write $g^{\mathcal{E}}=\langle \cdot,\cdot
\rangle_{\mathcal{E}}$ for the Euclidean metric.

\medskip

The metric $g$ allows to define several important objects for optimization.
First, it determines the \emph{normal space} $\Nrm^{g}_x \Mx$ to the manifold $\M_x$
at $x\in \calD$:
\begin{align}
    \Nrm^g_x \M_x:= \left\lbrace v\in \calE \colon g(v,\xi) = 0 \text{ for all }
    \xi\in \Trm_x \Mx\right \rbrace.
\end{align}

Throughout the paper, we denote by $\Proj_{x,g} \colon \calE \to \Trm_x \Mx$ the
orthogonal projection operator onto $\Trm_x \Mx$ with respect to the
inner product $g$. This operator is the unique linear operator satisfying 
\begin{align}\label{eq:projection}
    g(\xi, v-\Proj_{x,g}(v))=0\quad \text{for all }\xi \in \Trm_x \Mx\text{
   and }v\in \calE,
\end{align}
and it is associated with the decomposition of $\calE$ into $\Trm_x \Mx \oplus \Nrm^{g}_x
\mathcal{M}_x$
~\cite[(3.37)]{absil2008optimization}.
  If the metric $g$ is not the Euclidean inner product, we say that $\Proj_{x,g}$ 
  is an \emph{oblique} projector (although it is a $g$-orthogonal projector).
 In \cref{sec:geometry}, we adopt a point of view where the tangent projector is
 defined before the metric $g$; in that case, we denote it simply by $\Proj_x$.
 The orthogonal projector associated to the Euclidean metric $g^{\mathcal{E}}$ will be
denoted by $\Proj_{x,\mathcal{E}}$.

\medskip 

 The metric enables to define gradients
 and adjoints, whose definitions are recalled in the next paragraphs. 

\bigskip 

\noindent
\textbf{Riemannian gradients on $\calE$ and on $\Mx$}

\smallskip 
Let $\DD f(x)\colon \mathcal{E}\to \R$ denote the differential of $f$ at $x\in
\mathcal{D}$. Through the celebrated Riesz representation theorem, there are
several ways to identify the linear form $\DD f(x)$ to a vector playing the role
of a gradient. First, for any $x\in \mathcal{D}$, the \emph{unconstrained
Euclidean gradient} of $f$ is the unique element denoted by 
$\nabla_{\mathcal{E}}f(x) \in \mathcal{E}$ such that
\[
\inner{\xi}{ \nabla_\calE f(x)}_\calE = 
\mathrm{D}f(x)[\xi] \text{ for all }\xi
    \in \mathcal{E}.
\]
The unconstrained \emph{Riemannian} gradient of $f$ with respect to the metric
$g$ is written $\nabla_g f$ and is defined, for every $x\in \calD$, as the
unique element $\nabla_g f(x)\in\rmT_x \calD \simeq \calE$ that satisfies 
	\begin{align}\label{eq:unconstrained_gradient}
        g(\nabla_g f(x),\xi) = \mathrm{D}f(x)[\xi]  \quad \text{ for all }\xi
        \in \rmT_x
    \calD\simeq \mathcal{E}.
	\end{align}
The Riemannian metric induces in turn a \emph{constrained} Riemannian gradient
on the manifold $\mathcal{M}_x$, denoted by $\gradgx f$. Given $x\in \calD$,
$\gradgx f(x)$ is the unique vector in $\rmT_x\Mx$ satisfying 
\begin{align}\label{eq:riemannian_gradient}
 g(\gradgx f(x),\xi)=\mathrm{D}f(x)[\xi]  , \quad \text{ for all }\xi\in
\Trm_x \Mx .
\end{align}
The unconstrained Euclidean and Riemannian gradients are related by the
identity 
\begin{align}\label{eq:nablaEG}
  %  \nabla_\calE f(x) = G(x)\nabla_g f(x).  \qquad \text{and} \qquad 
    \nabla_g f(x) = G(x) \inv\nabla_\calE f(x).
\end{align}
Moreover,  the constrained Riemannian gradient is the orthogonal projection of
the unconstrained Riemannian gradient on the tangent space $T_x\mathcal{M}_x$:
 \begin{equation}
 \label{eqn:vdyhp}
\gradgx f(x) = \Proj_{x,g}(\nabla_g f(x)).
 \end{equation}

%        \begin{mdframed}[backgroundcolor=black!10]
%            \FF{Maybe to remove}
%		The constrained Euclidean gradient of $f$ is the unique element of $ \Tx\Mx $ such that $\mathrm{D}f(x)[\xi] =  \inner{\xi}{ \gradEx f(x)}_\calE$ for all $\xi \in  \Tx\Mx $. Naturally, we have 
%				\begin{align}\label{eq:gradE}
%\gradEx f(x) = \Proj_{x,\calE}(\nabla_\calE f(x)),					
%				\end{align}
%and from~\eqref{eq:nablaEG}, we find
%%		\begin{align}
%%				 \inner{\xi}{ \gradEx f(x)}_\calE =	\mathrm{D}f(x)[\xi] =  g(\xi, \gradgx f(x)) = \inner{\xi}{G(x)\gradgx f(x)}_\calE 
%%		\end{align}
%		\begin{align}\label{eq:gradproj_g2e}
%			\gradEx f(x)= \Proj_{x,\calE}\left(G(x)\nabla_g f(x)\right),
%%			\gradEx f(x)= \Proj_{x,\calE}\left(G(x)\gradgx f(x)\right),
%		\end{align}
%and
%		\begin{align}\label{eq:gradproj_e2g}
%			\gradgx = \Proj_{x,g}\left(G(x)\inv \nabla_\calE f(x)\right).
%%			\gradgx = \Proj_{x,g}\left(G(x)\inv \gradEx f(x)\right).
%		\end{align}
%\end{mdframed}

        %\FF{Do we really need to define $\gradEx$?}
 
 \bigskip 

 \noindent\textbf{Metric adjoints}

 \smallskip 

%When $\calE=\R^n$ and $\cset=\R^m$, we write $\mathrm{J}_c(x)\in \Rmn$ for its matrix representation.
  Given a linear operator $A:\calE\to\cset$, 
  the adjoint with respect to the metric $g$ is defined as the
unique linear operator $A^{*,g}:\cset\to\calE$ satisfying
\begin{equation}
\label{eqn:dvpan}
\inner{A\xi}{y}_\cset \;=\; g(\xi,A^{*,g}y),
\qquad \text{ for all }\,\xi\in\calE,\ y\in\cset.
\end{equation}
We 
denote by $A^{*,\mathcal{E}}$ the adjoint operator with respect to the Euclidean
metric, which satisfies:
\begin{equation}
\label{eqn:0ulus}
\inner{A\xi}{y}_\cset \;=\; \inner{\xi}{A^{*,\calE}y}_\calE,
\qquad \text{ for all }\,\xi\in\calE,\ y\in\cset.
\end{equation}
The Riemannian and the Euclidean adjoint operators 
are related through the
identity
\begin{equation}\label{eq:link_adjoints}
A^{*,g} = G(x)\inv A^{*,\calE}.
\end{equation}
In the particular case where 
 $\calE = \Rn$ and $\cset = \Rm$, 
 the differential $\DD c(x)$  is represented by its Jacobian matrix
\begin{align}
	\DD c(x) & = \begin{pmatrix}
		\nabla c_1(x)\T \\		
		\vdots \\ 
		\nabla c_m(x)\T
	\end{pmatrix} \in \Rmn,
\end{align}
and the Euclidean adjoint of 
any matrix $A\in \Rmn$ is given by the usual transpose $A^{*,\mathcal{E}}=A\T $.

\bigskip

\noindent\textbf{Metric pseudoinverse and formulas for $g$-orthogonal projectors}
\smallskip

An important operator in optimization algorithms is  the mapping
 \[
\DD c(x)\,\DD c(x)^{*,g}\colon \cset\to\cset,
\]
which is self-adjoint with respect to the inner product on $\mathcal{F}$ 
and positive definite for $x\in \mathcal{D}$.
The right \emph{$g$-pseudoinverse} of $\DD c(x)$ is defined as the operator 
\begin{align}\label{eq:pseudoinverse_g}
\DD c(x)^{\dagger,g} := \DD c(x)^{*,g}\,\left(\DD c(x)\,\DD
c(x)^{*,g}\right)^{-1}
\colon \cset\to\calE,
\end{align}
which satisfies $\DD c(x)\DD c(x)^{\dagger,g}=\Id_{\mathcal{F}}$.
The $g$-orthogonal projectors onto $\Tx\mathcal{M}_x$ and $\Nx \mathcal{M}_x$ 
read respectively
\begin{align}\label{eq:Projg}
 \Proj_{x,g}
&=
\Id_{\calE} - (\DD c(x))^{\dagger,g}\,\DD c(x)  &&\textrm{ and } &  \Proj_{x,g}\p
=
(\DD c(x))^{\dagger,g}\,\DD c(x) .	
\end{align}

We note that the $g$-normal space is \(\Nxg\Mx=\mathrm{Range}(\DD
c(x)^{*,g})\).

\section{Alternative definitions of the landing algorithm 
and metric construction via normal bundles.}
\label{sec:RLM}

The first purpose of this section is to establish the equivalence between 
the Riemannian landing method~\cref{eq:landing} and a 
different definition (see~\eqref{eq:pseudoinverse_landing}), in which the normal step is replaced by
 a pseudoinverse step of the form~\cref{eqn:3g1yn}, as considered in the
projected gradient flow \cref{eqn:gho94}. 
More precisely, in~\cref{subsec:bxujs}, we show that the
pseudoinverse step $d_N(x)$ depends only on the choice of the normal space and
not on the restriction of the normal metric. Then, we prove that the normal metric can
always be redefined so that the pseudoinverse step becomes the Riemannian gradient
of the constraint violation $\psi(x)=\frac12 \norm{c(x)}_{\cset}^{2}$ as in \cref{eqn:3dfjy}.

To prepare this result, we propose in
\cref{sec:euclidean_landing} a systematic procedure for defining the ambient metric $g$. 
Instead of specifying the metric directly, we first choose a decomposition of the 
ambient space into tangent and normal spaces,
before defining the metric separately on these two subspaces.
This enables to describe
explicitly how the tangent and normal components of the landing flow
depend on the orthogonal projection on $\Tx\Mx$ and on the restrictions
of the metric $g$ to the tangent and normal spaces.

\medskip

Before turning to these results, we observe that the Riemannian landing cannot be made more
general by considering two different metrics for the
tangent and normal terms.
\begin{lemma}
    Consider the landing scheme \cref{eq:landing} with the tangent and normal
    terms calculated in two different metrics $g_1$ and $g_2$:
    \begin{align}
        d_T(x) &= -\grad^{g_1}_{\mathcal{M}_x} f(x),\\
        d_N(x) &= -\nabla_{g_2}\psi(x).
    \end{align}
    There exists a common metric $g$ such that
    $d_T(x)=-\grad^{g}_{\mathcal{M}_x}f(x)$ 
    and  $d_N(x)=-\nabla_g \psi(x)$.
\end{lemma}
\begin{proof}
    Consider the metric $g$ given by 
    \[
    g(\xi,\zeta):=g_1(\Proj_{x,g_2}\xi,\Proj_{x,g_2}\zeta) 
    +g_2(\Proj_{x,g_2}^{\perp}\xi,\Proj_{x,g_2}^{\perp}\zeta).
    \] 
    It is clear that for any $\xi,\zeta\in\Tx\Mx$,
    $g(\xi,\zeta)=g_1(\xi,\zeta)$, which implies
    $\grad^{g_1}_{\Mx}f(x)=\grad^{g}_{\Mx}f(x)$. On the other hand,
    since $\DD \psi(x)[\xi] = \langle \DD c(x)\xi,c(x)  \rangle_{\mathcal{F}}$,
    $\nabla_{g_2}\psi(x)$ is the unique vector in $\mathcal{E}$
    satisfying 
$g_2(\nabla_{g_2}\psi(x),\xi)=\langle \DD c(x)\xi,c(x)  \rangle_{\mathcal{F}}$
for all $\xi\in \mathcal{E}$. Since
$g(\nabla_{g_2}\psi(x),\xi)=g_2(\nabla_{g_2}\psi(x),\xi)$ for all $\xi \in \calE$, we have 
$\nabla_{g_2}\psi(x)=\nabla_g
\psi(x)$. 
\end{proof}
We note that the initial paper defining the landing algorithm
\citep{ablin2022fast} on the Stiefel manifold considered 
a Riemannian metric for the tangent term
and the Euclidean metric for the normal 
term~\citep{gao_optimization_2022}. The previous
lemma shows that the metric in the normal space can always be redefined to infer
both terms from a common metric.

\subsection{Construction of the ambient metric $g$ from a 
normal bundle}
\label{sec:euclidean_landing}

In the Riemannian landing iteration, defined in \cref{eq:landing}, the metric $g$ is
given \emph{a priori}, which enables to compute the tangent and the normal
terms from the formulas 
\begin{align}
    \label{eqn:uaqaq}
    d_T(x) &= -\Proj_{x,g}(G(x)^{-1}\nabla_{\mathcal{E}}f(x)),\\
    \label{eqn:f7rzw}
    d_N(x) &= -G(x)^{-1} \DD c(x)^{*,\mathcal{E}}c(x).
\end{align}
The first formula follows 
from 
\cref{eq:nablaEG,eqn:vdyhp,eq:link_adjoints}, while 
the second 
can be inferred from the identities
\cref{eq:unconstrained_gradient,eqn:dvpan}. Indeed, for all $\xi \in \calE$, 
we have
\[
 g(\nabla_g\psi(x),\xi)=   \DD \psi(x) [\xi] = \langle \DD c(x)\xi, c(x)
 \rangle_{\mathcal{F}}=g(\DD
    c(x)^{*,g}c(x),\xi)
   =g(G(x)^{-1}\DD c(x)^{*,\mathcal{E}}c(x),\xi).
\] 
An issue with formulas \cref{eqn:uaqaq,eqn:f7rzw}, in the context of Riemannian
matrix optimization, is that closed-form expressions for the projection
mapping $\Proj_{x,g}$ or the inverse metric $G(x)^{-1}$ may not be
available.

\medskip 

Alternatively, the metric $g$ can be constructed by \emph{first} specifying
a smooth field of (oblique) linear projectors 
\[
    x\mapsto \Proj_x,
\]
where, for each $x\in \mathcal{D}$, the operator $\Proj_x\colon 
\calE\to\calE	$ satisfies
 \[
 \Proj_x^2 = \Proj_x, \quad
\mathrm{Range}(\Proj_x)=\Tx\Mx 
\quad\text{ and }  {\Proj_{x}}_{\big|\Tx\Mx}=\Id.\]
This viewpoint is equivalent to attaching to every $x\in \Mx$ 
a normal space characterized by 
\[
    \Nx\Mx := \ker(\Proj_x).
\] 
Then, consider two operator fields 
\[
G_T(x)\colon \mathcal{E} \to \mathcal{E},
\qquad
G_N(x)\colon \mathcal{E}\to \mathcal{E},
\] 
which are symmetric with respect to the 
Euclidean inner product, 
and which are used to define the restriction of the metric to the tangent and
normal spaces respectively. They are
required to be positive definite on $\Tx\Mx$ 
and $\Nx\Mx$, i.e.,  
there exist $g_T,g_N>0$ such that 
\begin{equation}
\label{eqn:k6ox1}
\begin{aligned}
 \langle \xi,G_T(x)\xi
\rangle_{\mathcal{E}}\>g_T\norm{\xi}_{\mathcal{E}}^{2}, \qquad \text{ for all }
\xi\in\Tx\Mx,\\
\langle \xi,G_N(x)\xi
\rangle_{\mathcal{E}}\>g_N\norm{\xi}_{\mathcal{E}}^{2}, \qquad \text{ for all }
\xi\in\Nx\Mx.\\
\end{aligned}
\end{equation}
With these ingredients, the construction of the metric $g$ can be summarized in
three steps:
\begin{enumerate}[(i)]
\item choose a projector $\Proj_x$ defining the decomposition
      $\mathcal{E}=\Tx\Mx\oplus \Nx\Mx$; 
\item specify the metric restrictions $G_T(x)$ on the tangent space, and
      $G_N(x)$ on the normal space; 
\item combine both restrictions to define the ambient metric on the whole
      space.
\end{enumerate}

The resulting ambient metric is defined accordingly by
\begin{equation}
\label{eqn:oitse}
 g(\xi,\zeta)
 =\langle \Proj_x \xi,G_T(x)\Proj_x \zeta  \rangle_{\mathcal{E}}
 +\langle \Proj_x^{\perp}\xi,G_N(x)\Proj_x^{\perp}\zeta  \rangle_{\mathcal{E}},
\end{equation}
where $\Proj_x^{\perp}:=\Id_{\mathcal{E}}-\Proj_x$ denotes the (oblique)
projector on $\Nx\Mx$.
Since 
\begin{equation}\label{eqn:Gsmth}
    g(\xi,\zeta) =\langle  \xi, \Proj_{x}^{*,\mathcal{E}}  G_T(x)\Proj_x \zeta  \rangle_{\mathcal{E}}
+\langle \xi, (\Proj_{x}\p)^{*,\mathcal{E}}G_N(x)\Proj_x\p\zeta  \rangle_{\mathcal{E}},
\end{equation}
we have $g(\xi,\zeta)=\langle \xi,G(x)\zeta
\rangle_{\mathcal{E}}$ 
where $G(x)\colon\mathcal{E}\to \mathcal{E}$ is 
\begin{equation}\label{eq:block-diag-metric}
	G(x) = \Proj_{x}^{*,\mathcal{E}}  G_T(x)  \Proj_{x} + 
    (\Proj_{x}\p)^{*,\mathcal{E}}  G_N(x)  \Proj_{x}\p.
\end{equation}

As a result of this construction, $\Tx\Mx$ and $\Nx\Mx$ are $g$-orthogonal
 subspaces, and 
  $\Proj_x=\Proj_{x,g}$ and $\Proj_x\p=\Proj_{x,g}^{\perp}$ 
are
the $g$-orthogonal projectors
associated with the decomposition $\mathcal{E}=\Tx\Mx\oplus \Nx\Mx$.

\medskip 

Finally,~\eqref{eqn:Gsmth} reads
\begin{equation}
     g(\xi,\zeta) = \inner{\xi}{\widetilde{G_T}\zeta}_{\calE} + \inner{\xi}{\widetilde{G_N}\zeta}_\calE,
\end{equation}
where $\widetilde{G_T}$ and $\widetilde{G_N}$ are the the operators
\begin{equation}
\label{eqn:4ir6e}
    \begin{aligned}
        \widetilde{G_T}& :  =\Proj_{x}^{*,\mathcal{E}}  G_T(x)  \Proj_{x} 
        & \,:\, \Tx\Mx \to (\Nx\Mx)^{\perp,\mathcal{E}},\\
        \widetilde{G_N}& :  =(\Proj_{x}\p)^{*,\mathcal{E}}  G_N(x) 
        \Proj_{x}\p 
        & \,:\, \Nx\Mx \to (\Tx\Mx)^{\perp,\mathcal{E}},\\
    \end{aligned}
\end{equation}
which are invertible due to \cref{eqn:k6ox1}. Denoting their inverses 
by $\widetilde{G_T}(x)^{-1}$
and $\widetilde{G_N}(x)^{-1}$, the inverse of
$G(x)$ reads 
\begin{equation}\label{eq:block-diag-metricinverse}
    G(x)^{-1} = \Proj_{x}  \widetilde{G_T}(x)^{-1}  \Proj_{x}^{*,\mathcal{E}} + 
    \Proj_{x}\p  \widetilde{G_N}(x)^{-1}  (\Proj_{x}\p)^{*,\mathcal{E}}. 
\end{equation}
Note that, since $\mathrm{Range}(\Proj_x^{*,\calE})=(\ker \Proj_x)^{\perp,\calE}$,
 the linear projectors \( \Proj_x^{*,\mathcal E}\) and \((\Proj_x^\perp)^{*,\mathcal E}\) 
 have ranges \((\Nx\Mx)^{\perp,\mathcal E}\) and \((\Tx\Mx)^{\perp,\mathcal E}\) respectively,
  and are associated with the decomposition
   \( \mathcal{E}=(\Nx\Mx)^{\perp,\mathcal E}\oplus (\Tx\Mx)^{\perp,\mathcal E}\).
\begin{remark}
Since only the restrictions of $G_T$ and $G_N$ to respectively the
tangent space $\Tx\Mx$ and the normal space $\Nx\Mx$ matter in the definition of
$G(x)$, the definitions of $G_T(x)$ and $G_N(x)$ as operators on the whole set
$\mathcal{E}$ is motivated by the need to define a differentiable metric field through
\cref{eq:block-diag-metric}. 
\end{remark}
\begin{remark}
This approach, whereby the differential structure of a manifold is specified by
a mapping of projectors, has been considered to analyze the geometric structure of certain
algebraic matrix decompositions, see  
\citep{feppon_extrinsic_2019}. 
\end{remark}
We can now express the tangent and normal steps of 
the Riemannian landing~\eqref{eq:landing} in
terms of $\Proj_x$, $\widetilde{G_T}(x)^{-1}$, and $\widetilde{G_N}(x)^{-1}$.
\begin{proposition}
The tangent and normal steps of the Riemannian landing
~\eqref{eq:landing}
can be written in terms of
the projectors, the tangent metric $G_T(x)$, and the normal metric $G_N(x)$ as  
\begin{align}
    \label{eqn:gd5ra}
    d_T(x) &=
    -\widetilde{G_T}(x)^{-1}\Proj_x^{*,\mathcal{E}}\nabla_{\mathcal{E}}f(x),\\
    \label{eqn:dploh}
        d_N(x) &= - \widetilde{G_N}(x)^{-1}\DD c(x)^{*,\mathcal{E}}c(x).
\end{align} 
\end{proposition}
\begin{proof}
     By  \cref{eq:nablaEG},
     \cref{eqn:vdyhp} and \cref{eq:block-diag-metricinverse}, 
\[
    \begin{aligned}
    d_T(x) & =-\grad^{g}_{\M_x}f(x)
  =-\Proj_{x}(G(x)^{-1}\nabla_{\mathcal{E}}f(x)) \\
 & =-\Proj_x (\Proj_x  \widetilde{G_T}(x)^{-1} \Proj_x^{*,\mathcal{E}}  +
 \Proj_x^{\perp} \widetilde{G_N}(x)^{-1}
 (\Proj_x^{\perp})^{*,\mathcal{E}})\nabla_{\mathcal{E}}f(x)\\
 &= -\Proj_x  \widetilde{G_T}(x)^{-1}\Proj_x^{*,\mathcal{E}}
 \nabla_{\mathcal{E}}f(x)=
 -\widetilde{G_T}(x)^{-1}\Proj_x^{*,\mathcal{E}}
 \nabla_{\mathcal{E}}f(x).
    \end{aligned}
\] 
This proves \cref{eqn:gd5ra}. For \cref{eqn:dploh}, 
we write 
\[
    \begin{aligned}
    d_N(x) & =-\DD c(x)^{*,g}c(x)
  =-G(x)^{-1}\DD c(x)^{*,\mathcal{E}}c(x)
   =-G(x)^{-1}(\Proj_x^\perp)^{*,\mathcal{E}} \DD c(x)^{*,\mathcal{E}} c(x)\\
  &= -\widetilde{G_N}(x)^{-1}        \DD c(x)^{*,\mathcal{E}}c(x).
    \end{aligned}
\] 
\end{proof}
\begin{remark}
    \label{rmk:xxqnr}
Although formula \cref{eqn:gd5ra} suggests that $d_T(x)$ could depend on  the
choice of the normal space through the projector $\Proj_x^{*,\mathcal{E}}$, this is not the
case. This is visible in the definition \cref{eq:riemannian_gradient} of the
constrained Riemannian gradient. The operator
$\widetilde{G_T}(x)^{-1}\Proj_x^{*,\mathcal{E}}$ depends thus solely on the
restriction of the metric $g$ to the tangent space $\Tx\Mx$, but not on the
particular choice of normal space. In particular,  the formula
\[
    \widetilde{G_T}(x)^{-1}\Proj_x^{*,\mathcal{E}}=
    (\Proj_x^{*,\mathcal{E}}G_T(x)\Proj_x)^{-1}\Proj_x^{*,\mathcal{E}} 
\] 
does not depend on the choice of projection projector $\Proj_x$ so long as
$\mathrm{Range}(\Proj_x)=\Tx\Mx$. In fact, one may verify that 
\[
\widetilde{G_T}(x)^{-1}\Proj_x^{*,\mathcal{E}}\xi 
=\sum_{i=1}^{\dim(\Tx\Mx)} \langle \xi,u_i  \rangle_{\mathcal{E}} u_i,
\] 
where $(u_i)_{1\<i\<\dim(\Tx\Mx)}$ is any $g$-orthonormal basis of $\Tx\Mx$.
However, \cref{eqn:gd5ra,eqn:dploh} show that  closed-form expressions are
available for $d_T(x)$ and $d_N(x)$ if this is the case for (i) the projection
operator $\Proj_x$ and (ii) the inverse of the restriction operators
$\widetilde{G_T}(x)$ and $\widetilde{G_N}(x)$.
\end{remark}

\subsection{Link between the Riemannian landing method and projected gradient flows}
\label{subsec:bxujs}
This section defines a landing iteration based on a pseudoinverse normal 
step~\eqref{eq:pseudoinverse_landing}, and shows that it
is equivalent to the Riemannian 
landing scheme introduced in~\eqref{eq:landing}.
Given an arbitrary smooth symmetric positive-definite operator field
\[H(x)\colon \cset \to \cset,\] 
we consider the iterative scheme 
\begin{subequations}\label{eq:pseudoinverse_landing}
\begin{align}
x_{k+1} &= x_k + \alpha_k (d_T(x_k)+d_N(x_k)),\\
    d_T(x) &= -\grad^{g}_{\mathcal{M}_x} f(x)\in\Tx\Mx,  \\
d_N(x) &= -\DD c(x)^{\dagger,g} H(x)c(x)\in \Nxg\Mx. \label{eq:rlanding-gstar-normal}
\end{align}
\end{subequations}
The tangent term $d_T(x)$ is unchanged compared to the Riemannian landing~\eqref{eq:landing}, and 
the new normal 
component $d_N(x)$---labelled a `pseudoinverse' step---is defined for $x\in \mathcal{D}$. 
The iterative algorithm \cref{eq:pseudoinverse_landing} 
may be interpreted as the discretization of the ordinary differential equation
\begin{equation}
\label{eqn:vvhfc}
\begin{aligned}
\dot x  & = -\tangent(x) - \normal(x) \\
&= (\Id_{\mathcal{E}}-
 \DD c(x)^{*,g}\,\left(\DD c(x)\,\DD c(x)^{*,g}\right)^{-1}\DD c(x))\nabla_g f(x) \\
 & \quad 
-\,\DD c(x)^{*,g} \left( \DD c(x) \DD c(x)^{*,g}\right)\inv\,H(x)\,c(x),
\end{aligned}
\end{equation}
When
$H(x)=\Id$ and $g=g^{\mathcal{E}}$ is the Euclidean metric, 
\cref{eqn:vvhfc} coincides with the 
pseudoinverse step of the
projected gradient flow \cref{eqn:gho94}. 

\medskip

Formula \cref{eq:rlanding-gstar-normal} offers a somewhat more explicit control
of how optimization trajectories land on the constraint manifold than 
\cref{eqn:3dfjy} based on the unconstrained Riemannian gradient. It
is a descent direction for the infeasibility measure
$\psi(x)=\norm{c(x)}_{\mathcal{F}}^{2}/2$: 
\begin{equation}
    \begin{aligned}
	\DD \psi (x) [d_N(x)] &= \inner{\DD c(x)[d_N(x)]}{c(x)}_\cset \\
    &
    = -\inner{\DD c(x)\DD c(x)^{*,g} \left( \DD c(x) \DD
    c(x)^{*,g}\right)\inv\,H(x)\,c(x)}{c(x)}_\cset\\
	&= -\inner{ H(x) c(x)}{c(x)}_\cset <0.
    \end{aligned}
\end{equation}
We note that the descent value depends explicitly
on $H(x)$ and is independent of $g$.
More precisely, the normal step $d_N(x)$ is the minimum $g$-norm solution of the
undetermined system
$\DD \c(x)[\normal] = -H(x)\c(x)$, that is,
\begin{align}\label{eq:pseudoinverse_problem}
d_N(x):=\arg\min_{\normal \in \calE} \frac{1}{2}\norm{\normal}_g^{2} \text{ such that } \DD
\c(x)[\normal] = -H(x)\c(x).
\end{align}
This implies that, at the continuous level, the constraint values along the
gradient flow trajectories $x(t)$ solving \cref{eqn:vvhfc} satisfy 
\[
\ddt c(x(t))= -H(x(t))c(x(t)), 
\] 
which entails $c(x(t))=c(x(0))\exp\left( -\int_0^{t}H(x(s))\mathrm{d} s \right)$.
Hence,
the  eigenvalues and the eigenvectors of the symmetric operator $H(x)$  determine
the instantaneous decay rate of the components of the constraint vector
$c(x)\in\mathcal{F}$.

%First, it is 
% designed to be a descent direction for the
%infeasibility measure $\psi(x)=||c(x)||_{\mathcal{F}}^{2}/2$:
%if $c(x)\neq 0$, it holds 

\medskip

When $H(x)=\DD c(x)\,\DD c(x)^{*,g}$, we retrieve $d_N(x)=-\DD
c(x)^{*,g}c(x)=-\nabla_g\psi(x)$, the negative unconstrained Riemannian gradient
of the infeasibility measure $\psi(x)=\norm{c(x)}^{2}_{\mathcal{F}}/2$. The
following propositions show that the converse result is true: \emph{firstly},
\cref{eq:rlanding-gstar-normal} does not depend on the choice of the normal
metric $G_N(x)$; \emph{secondly}, it is actually possible to redefine the normal
metric $G_N(x)$ \emph{without changing the values of $d_T(x)$ and $d_N(x)$} so
that $d_N(x)$ becomes the negative Riemannian unconstrained gradient of $\psi$
in the new metric. 

\begin{proposition}
    \label{prop:kk5n4}
The `pseudoinverse' normal step \cref{eq:rlanding-gstar-normal} 
depends only on the choice of orthogonal normal space $\Nx\Mx$ through 
the normal projector $\Proj_x^{\perp}$. More precisely, it holds
\begin{equation}
        \label{eqn:cuylb}
        d_N(x) = -\Proj_{x}^{\perp} \DD c(x)^{\dagger,\mathcal{E}} H(x) c(x).
\end{equation} 
\end{proposition}
\begin{proof}
   Recall that $\Proj_x^{\perp}=\DD c(x)^{\dagger,g} \DD c(x)$ (eq.
\cref{eq:Projg}). Multiplying to the right by $\DD c(x)^{\dagger,\mathcal{E}}$ and using $\DD
c(x)\DD c(x)^{\dagger,\mathcal{E}}=\Id_{\mathcal{F}}$, we obtain the
identity $\DD c(x)^{\dagger,g}=\Proj_x^{\perp}  \DD c(x)^{\dagger,\mathcal{E}}$
relating the $g$-- and Euclidean right-pseudoinverses.
Formula \cref{eqn:cuylb} follows by substituting into
\cref{eq:rlanding-gstar-normal}.
\end{proof}
%Proposition~\ref{prop:pseudoinverse_E} shows that for a suitable block-diagonal metric $g$, we have 
%	\begin{align}\label{eq:pseudoinverse_euclidean}
%			d_N(x) = - \nabla_g \psi(x) =	-\DD c(x)^{\dagger,\calE} c(x), % =	-\DD c(x)^{*,\calE} \left(\DD c(x)\,\DD c(x)^{*,\calE}\right)^{-1} c(x).
%		\end{align}
%		where $\DD c(x)^{\dagger,\calE}:=\DD c(x)^{*,\calE}\bigl(\DD c(x)\,\DD c(x)^{*,\calE}\bigr)^{-1}$
%is the (Euclidean) pseudoinverse of $\DD c(x)$. 

\begin{proposition}\label{prop:pseudoinverse_E}
The pseudoinverse normal step
\cref{eq:rlanding-gstar-normal} (or \cref{eqn:cuylb}) can be rewritten as 
	\begin{align}\label{eq:pseudoinverse_h}
			d_N(x) =	 - \nabla_g \psi(x), % =	-\DD c(x)^{*,\calE} \left(\DD c(x)\,\DD c(x)^{*,\calE}\right)^{-1} c(x).
		\end{align}
        with 
$\psi(x)=\norm{c(x)}_{\mathcal{F}}^{2}/2$
for the metric $g$ defined through \cref{eqn:4x5li,eq:block-diag-metric} 
by setting 
\[
    G_N(x):= \DD c(x)^{*,\calE}H(x)^{-1} \DD c(x),
\]
which defines a symmetric positive-definite operator on
$	\rmN_{x} \Mx$. \end{proposition}
\begin{proof}
The operator $G_N(x)$ is clearly symmetric for the Euclidean inner product. To
prove positive definiteness of $G_N(x)$ on $\rmN_x^\calE\calM_x$, consider any
nonzero $\xi\in\Nx\Mx$. We find
\[
\langle G_N(x)\xi,\,\xi\rangle_{\calE}
\;=\;\langle \DD c(x)^{*,\calE}H(x)^{-1}\DD c(x)\xi,\,\xi\rangle_{\calE}
\;=\;\langle H(x)^{-1}\DD c(x)\xi,\,\DD c(x)\xi\rangle_{\cset}
> 0,
\]
since $\xi \in \rmN_x^\calE \calM_x$, $\ker \DD c(x) \cap  \Nx =
\{0\}$ and $H(x)^{-1}$ is symmetric definite positive.
Thus, $G_N(x)\succ 0$ on $\Nx\calM_x$ and \cref{eqn:oitse} (leaving $G_T(x)$
unchanged) does define
a metric $g$ on $\mathcal{E}$.

\noindent
Let us recall (equation \cref{eqn:dploh}) that the 
unconstrained Riemannian gradient of $\psi(x)$ is given by 
\begin{equation}
    \label{eqn:mzs5n}
    \nabla_g\psi(x)=\widetilde{G_N}(x)^{-1}\DD c(x)^{*,\mathcal{E}} c(x).
\end{equation}
Let us express $\nabla_g\psi(x)$ in pseudoinverse form. 
Due to $\ker(\DD c(x))=\Tx\Mx$ and $\mathrm{Range}(\DD
c(x)^{*,\mathcal{E}})=\Tx\Mx^{\perp,\mathcal{E}}$, we have 
\begin{equation}
\label{eqn:5miob}
\begin{aligned}
    \DD c(x)\Proj_x^{\perp} &  =\DD c(x), & \Proj_x^{*,\mathcal{E}}\DD
c(x)^{*,\mathcal{E}}
&  =0\\ 
\DD c(x)\Proj_x & = 0,  & (\Proj_x^{\perp})^{*,\mathcal{E}}\DD
      c(x)^{*,\mathcal{E}}& =\DD c(x)^{*,\mathcal{E}}.
\end{aligned}
\end{equation}
These identities imply
\[
    \widetilde{G_N}(x)=(\Proj_x^{\perp})^{*,\mathcal{E}} 
G_N(x) \Proj_x^{\perp}=(\Proj_x^{\perp})^{*,\mathcal{E}}\DD
c(x)^{*,\mathcal{E}}H(x)^{-1}\DD c(x)
\Proj_x^{\perp}=\DD c(x)^{*,\mathcal{E}}H(x)^{-1}\DD c(x).
\] 
Right multiplying by $\Proj_x^{\perp}\DD c(x)^{\dagger,\mathcal{E}}$, 
this leads to
\[
    \begin{aligned}
    \widetilde{G}_N(x)\Proj_x^{\perp} \DD c(x)^{\dagger,\mathcal{E}} 
     & =\DD c(x)^{*,\mathcal{E}}H(x)^{-1}\DD c(x)\Proj_x^{\perp}\DD c(x)^{\dagger,\mathcal{E}}
    \\
     & =\DD c(x)^{*,\mathcal{E}}H(x)^{-1}\DD c(x)\DD c(x)^{\dagger,\mathcal{E}} 
     = \DD c(x)^{*,\mathcal{E}}H(x)^{-1},
    \end{aligned}
\] 
invoking the first line of \cref{eqn:5miob}. 
Left multiplying by
$\widetilde{G_N}(x)^{-1}\,:\,\Tx\Mx^{\perp,\mathcal{E}}\to\Nx\Mx$, which is
allowed because 
$\mathrm{Range}(\DD c(x)^{*,\mathcal{E}})=\Tx\Mx^{\perp,\mathcal{E}}$, and right
multiplying by $H(x)$, we obtain 
\begin{equation}\label{eq:GN-inverse-identity}
    \widetilde{G_N}(x)^{-1}\,\DD c(x)^{*,\calE}
    =\Proj_x^{\perp}\DD c(x)^{\dagger,\mathcal{E}}H(x).
\end{equation}
Combining \eqref{eqn:mzs5n} 
and \eqref{eq:GN-inverse-identity} yields~\eqref{eq:pseudoinverse_h}.
\end{proof}

In other words, defining the normal term 
as a negative unconstrained Riemannian gradient \cref{eqn:3dfjy} 
or 
as the pseudo inverse step 
\cref{eq:rlanding-gstar-normal} 
 leads to the same family of
optimization algorithms.

\section{Links between the Riemannian landing method and 
existing algorithms}\label{sec:comparison}

In this section, we make explicit connections between the Riemannian landing 
method and several well-known constrained optimization algorithms.
Specifically, in \cref{subsec:gfwul}, we demonstrate that the
landing algorithm---with equivalent definitions in~\cref{eq:landing} 
and~\eqref{eq:pseudoinverse_landing}---can be interpreted as an
augmented Lagrangian method with the least-squares Lagrange multiplier update.

In
\cref{subsec:c4t76}, we prove that the basic version of Sequential
Quadratic Programming (SQP)---without trust-regions---
is a particular case of the landing algorithm. 

In \cref{subsec:3ubkf}, 
we deepen the connection---previously observed by multiple authors, e.g. 
\citep{miller_newton_2005,absil_all_2009,mishra2016riemannian}---between the
quadratically convergent version of SQP and the Riemannian Newton
method on a manifold. 
We show that
the landing method~\eqref{eq:pseudoinverse_landing}, with a
pseudoinverse normal update \(d_N(x) = -\DD c(x)^{\dagger,g}c(x)\) 
and a Riemannian Newton tangent step, 
achieves
quadratic convergence only for a specific choice of normal space,
which coincides with the one mandated by the SQP framework.

\subsection{Link with an augmented Lagrangian method}
\label{subsec:gfwul}
Consider the iterative
scheme 
\begin{equation}
\label{eqn:5cnov}
x_{k+1}= x_k-\alpha_k \nabla_{g}\mathcal{L}_{\beta_k}(x_k,\lambda_k),
\end{equation}
with  
the augmented Lagrangian $\mathcal{L}_{\beta_k}$ given by
\[
    \mathcal{L}_{\beta_k}(x, \lambda) =
    f(x) + \langle \lambda, \c(x)\rangle_{\mathcal{F}} +
    \frac{\beta_k}{2}\|\c(x)\|_{\mathcal{F}}^2.
\]
Here, $\beta_k$ is a penalty parameter which is not necessarily constant. 
In \cref{eqn:5cnov}, the unconstrained Riemannian gradient $\nabla_g
\mathcal{L}_{\beta_k}$ is taken with respect to the variable $x$.
\begin{proposition}
The augmented Lagrangian iteration \cref{eqn:5cnov} coincides with the Riemannian
landing iteration \cref{eq:pseudoinverse_landing} with $H(x_k):=\beta_k \DD c(x)\DD c(x)^{*,g}$ 
in the normal component, 
and with the least-squares
multipliers 
\begin{equation}
\label{eq:lambda}
\lambda_k := \arg\min_{\lambda \in \mathcal{F}} \norm{\nabla_{g} f(x_k) - \DD
\c(x_k)^{*,g}\lambda}^2_\mathcal{E} 
= (\DD c(x_k)\DD c(x_k)^{*,g})^{-1}\DD c(x_k)\nabla_g f(x_k).    
\end{equation} 
\end{proposition}
\begin{proof}
The unconstrained Riemannian gradient  of $\mathcal{L}_{\beta_k}$ with respect
to $x$ at $x=x_k$ reads 
\[
   \begin{aligned}
   \nabla_{g} \mathcal{L}_{\beta_k}(x_k,\lambda_k) & =\nabla_{g}f(x_k)+\DD
c(x_k)^{*,g}\lambda_k+\beta_k \DD c(x_k)^{*,g} c(x_k)\\
 & =(\Id_{\mathcal{E}}-\DD c(x_k)^{*,g}(\DD c(x_k)\DD
c(x_k)^{*,g})^{-1}\DD c(x_k))\nabla_g f(x_k)+\beta_k \DD c(x_k)^{*,g}c(x_k).
   \end{aligned}
\] 
Using \cref{eqn:vdyhp} and \cref{eq:rlanding-gstar-normal} with 
$H(x)=\beta_k \DD c(x)\DD c(x)^{*,g}$, we obtain 
$\nabla_g \mathcal{L}_{\beta_k}(x_k,\lambda_k)= \grad^g_{\Mx} f(x_k)+d_N(x_k)$.
\end{proof}
Thus, the landing algorithm is an augmented
Lagrangian method with   the Lagrange multipliers $\lambda_k$ being 
the least-squares multipliers, and for which the minimization
of the augmented Lagrangian subproblem consists in a single gradient step. 
\subsection{SQP is a particular case of the Riemannian landing method} 
\label{subsec:c4t76}
In its most basic form, as described in \citep{boggs_sequential_1995}, the
Sequential Quadratic Programming (SQP) method considers the iterative sequence 
\begin{equation}
\label{eqn:chtww}
x_{k+1}=x_k+\alpha_k d_k,
\end{equation}
where $d_k$ is obtained 
by solving
at each iteration the quadratic program
\begin{equation}\label{eq:SQP}
\begin{aligned}
d_k:= & \arg\underset{d\in \mathcal{E}}{\min}
& & f(x_k) + \langle d,\nabla_{\mathcal{E}} f(x_k)  \rangle_{\mathcal{E}} +
\frac12 \langle d, B_k d  \rangle_{\mathcal{E}} \\
& \text{subject to}
& & \DD \c(x_k)d + \c(x_k) =0,
\end{aligned}
\end{equation}
given a sequence of symmetric positive definite operators $B_k\,: \mathcal{E}\to
\mathcal{E}$. This method is locally convergent under mild assumptions on the selection
of symmetric positive-definite matrices $B_k$. 
Setting $B_k=\nabla^{2}_{\mathcal{E}}\mathcal{L}(x,\lambda_k)$
($\nabla^{2}_{\mathcal{E}}$ being the Euclidean Hessian in the $x$ variable),
with the Lagrangian
\[
\calL(x,\lambda):=f(x)-\langle \lambda,c(x)  \rangle_{\mathcal{F}},
\]
SQP achieves quadratic convergence around a KKT point for a unit step size when
$\lambda_k \in \mathcal{F}$ is set to the Lagrange multiplier associated to the
equality constraint of \cref{eq:SQP} of the previous 
iteration~\citep{nocedal2006numerical}. Globalization procedures also exist, based on
merit functions or filters
\citep{boggs_sequential_1995,fletcher_global_2002,wachter_line_2005,obara_sequential_2022}.
\medskip

The following proposition shows that the SQP algorithm \cref{eqn:chtww}
forms a particular instance of the Riemannian landing method \cref{eq:pseudoinverse_landing}. 
 
\begin{proposition}
    \label{prop:6nqa8}
    Suppose that  
    $B_k$ is a symmetric positive definite operator on $\mathrm{T}_{x_k}\mathcal{M}_{x_k}$. 
   Then, if $x_k\in \mathcal{D}$, 
the SQP direction \cref{eq:SQP} is given by 
    \begin{equation}
    \label{eqn:i0xfv}
        \begin{aligned}
		d_{k} &= - \grad_{\mathcal{M}_{x_{k}}}^{g} f(x_k)  - \DD c(x_k)^{\dagger,g}c(x_k).
%		d_k &= - \Proj_{x,B_k}	B_k\inv\nabla_{\mathcal{E}} f(x)  - \DD c(x)^{*,B_k}\big(\DD c(x) \DD c(x)^{*,B_k}\big)^{-1}c(x)\\
%         & =-\left(\I_n-B_k^{-1}\DD c(x_k)\T(\DD c(x_k) B_k^{-1} \DD c(x_k)\T)^{-1}\DD
%        c(x_k)\right) 
%        B^{-1}\nabla f(x_k) \\
%        & \quad -B_k^{-1}\DD c(x_k)\T(\DD c(x_k)B_k^{-1}\DD c(x_k)\T)^{-1} c(x_k)
        \end{aligned}
    \end{equation}
where the metric $g$ is defined as 
\[
    g(\xi,\zeta) := \langle \Proj_{x_k}\xi,B_k\Proj_{x_k}\zeta  \rangle_{\mathcal{E}}
    +\langle \Proj_{x_k}^{\perp}\xi,G_{N}(x_k)\Proj_{x_k}^\perp \zeta\rangle_{\mathcal{E}},
\]
where $\Proj_{x_k}$ is the oblique projection on the decomposition
$\mathcal{E}=\Txk\Mxk\oplus\Nxk\Mxk$ with
  $\mathrm{N}_{x_k}\mathcal{M}_{x_k}$
    being the $\mathcal{E}$-orthogonal space to $B_k\mathrm{T}_{x_k}\mathcal{M}_{x_k}$:
    \begin{equation}
    \label{eqn:g68tx}
        \Nxk\Mxk:=(B_k\Txk\Mxk)^{\perp,\mathcal{E}},
    \end{equation}
and where $G_{N}(x_k)$ is any symmetric
positive definite operator on $\Nxk\Mxk$.

In other words, 
   the SQP direction $d_k$ is a particular case of the Riemannian landing
    direction \cref{eq:pseudoinverse_landing},
  which corresponds to the choice $H(x_k)=\Id_{\cset}$, the normal space~\cref{eqn:g68tx},
   and the metric
   induced by $B_k$ on $\mathrm{T}_{x_k}\mathcal{M}_{x_k}$.
\end{proposition}
\begin{proof}
    The space decomposition 
    $\mathcal{E}=\Txk\Mxk\oplus\Nxk\Mxk$ holds
    because $B_k$ is symmetric positive definite on $\Tx \Mx$.
    Observe that any $d$ satisfying $\DD c(x_k) d + c(x_k)=0$ can be written as 
    \[
    d=u+d_N(x_k),
    \] 
    with $d_N(x_k)=-\DD c(x_k)^{\dagger,g}c(x_k)$ and 
    $u\in \mathrm{T}_{x_k}\mathcal{M}_{x_k}$. The minimizer of \cref{eq:SQP}
    is $d_k=u_k+d_N(x_k)$, where $u_k$ is the solution to the quadratic
    unconstrained minimization
    problem
    \begin{equation}
    \label{eqn:zce8x}
    u_k=\arg\min_{u\in \mathrm{T}_{x_k}\mathcal{M}_{x_k}} f(x_k)+
\langle u+d_N(x_k),\nabla_{\mathcal{E}}f(x_k)  \rangle_{\mathcal{E}}
+\frac{1}{2}\langle u+d_N(x_k),B_k(u+d_N(x_k))  \rangle_{\mathcal{E}}.
    \end{equation}
    Note further that for $u\in\mathrm{T}_{x_k}\mathcal{M}_{x_k}$,
    \[
        \langle u,\nabla_{\mathcal{E}}f(x_k)  \rangle_{\mathcal{E}}=\langle
        \Proj_{x_k}
        u,\nabla_{\mathcal{E}}f(x_k)  \rangle_{\mathcal{E}}
        =\langle u,\Proj_{x_k}^{*,\mathcal{E}}\nabla_{\mathcal{E}}f(x_k)
        \rangle_{\mathcal{E}},
    \] 
    and since
    $d_N(x_k)\in\mathrm{N}_{x_k}\mathcal{M}_{x_k}=(B_k\Txk\Mxk)^{\perp,\mathcal{E}}$,
    \[
        \begin{aligned}
        \langle u+d_N(x_k),B_k(u+d_N(x_k))  \rangle_{\mathcal{E}}
     & =\langle u,B_ku  \rangle_{\mathcal{E}}+2\langle u,B_k d_N(x_k)  \rangle_{\mathcal{E}}
    +\langle d_N(x_k),B_k d_N(x_k)  \rangle_{\mathcal{E}} \\
    &=\langle u,G_T(x_k)u  \rangle_{\mathcal{E}}+2\langle d_N(x_k),B_k u  \rangle_{\mathcal{E}}
    +\langle d_N(x_k),B_k d_N(x_k)  \rangle_{\mathcal{E}} \\
    &=\langle u,G_T(x_k)u  \rangle_{\mathcal{E}}
    +\langle d_N(x_k),B_k d_N(x_k)  \rangle_{\mathcal{E}},
        \end{aligned}
    \] 
    where we denote $G_T(x_k):=B_k$.
    Consequently, by eliminating constant terms, the solution $u_k$ to
    \cref{eqn:zce8x} is also the minimizer of the quadratic problem
    \begin{equation}
    u_k=\arg\min_{u\in \mathrm{T}_{x_k}\mathcal{M}_{x_k}} 
\langle u,\Proj_{x_k}^{*,\mathcal{E}}\nabla_{\mathcal{E}}f(x_k)  \rangle_{\mathcal{E}}
+\frac{1}{2}\langle u,\widetilde{G_T}(x_k)u  \rangle_{\mathcal{E}},
    \end{equation}
    where $\widetilde{G_T}(x_k)\,:\,\Txk\Mxk\to\Nxk^{\perp,\mathcal{E}}$ is the operator
    defined in \cref{eqn:4ir6e}.
    The solution to this quadratic program is 
    $u_k=-\widetilde{G_T}(x_k)^{-1}\Proj_{x_k}^{*,\mathcal{E}}\nabla_{\mathcal{E}}f(x_k)$,
    which is exactly $d_T(x_k)$ according to \cref{eqn:gd5ra}.
\end{proof}
\begin{remark}
    An alternative, more algebraic proof can be obtained if one assumes $B_k$ to be
    symmetric positive definite on the whole space $\mathcal{E}$, rather than only on 
    $\mathrm{T}_{x_k}\mathcal{M}_{x_k}$.
In this case, the  result is well-known and can be found 
in different forms in the
literature; see for instance \cite[eq. (3.8)-(3.10)]{boggs_sequential_1995} or 
more recently \cite[Proposition 1]{feppon_density-based_2024}. The
idea of the proof is recalled here to highlight the link with the landing method
\cref{eq:pseudoinverse_landing}.
    The KKT condition for \cref{eq:SQP} states that there exists a Lagrange
    multiplier $\lambda_k\in \mathcal{F}$
such that the minimizer $d_k$  reads 
    \begin{equation}
    \label{eqn:trot6}
    d_k=-B_k^{-1}\nabla_\mathcal{E} f(x_k) +B_k^{-1}\DD c(x_k)^{*,\mathcal{E}}\lambda_k.
    \end{equation}
    Inserting this expression into $\DD c(x_k)d_k=-c(x_k)$, we infer that 
    \[
    \DD c(x_k)d_k = -c(x_k)=-\DD c(x_k)B_k^{-1}\nabla_{\mathcal{E}} f(x_k)+\DD c(x_k) B_k^{-1}\DD
    c(x_k)^{*,\mathcal{E}}\lambda_k.
    \] 
    Since $\DD c(x_k)$ is full rank and $B_k^{-1}$ is assumed to be invertible, the
    operator $\DD
    c(x_k)B_k^{-1}\DD c(x_k)^{*,\mathcal{E}}$
    is symmetric positive definite and $\lambda_k$ is given by 
    \[
    \lambda_k=
    -(\DD c(x_k)B_k^{-1}\DD c(x_k)^{*,\mathcal{E}})^{-1}c(x_k) +(\DD c(x_k)B_k^{-1}\DD
    c(x_k)^{*,\mathcal{E}})^{-1}\DD
    c(x_k)B_k^{-1}\nabla f(x_k).
    \] 
    Substituting this expression into \cref{eqn:trot6} yields
%\begin{align*}
%d_k
%&= -\,B_k^{-1}\nabla_{\mathcal{E}} f(x_k)
%\;-\;
%B_k^{-1}\DD c(x_k)^\top\big(\DD c(x_k)B_k^{-1}\DD c(x_k)^\top\big)^{-1}
%\Big(c(x_k) - \DD c(x_k) B_k^{-1}\nabla_{\mathcal{E}} f(x_k)\Big),
%\end{align*}
%which can be rewritten as
\begin{multline}
    d_k   = - \left(\I_n - B_k^{-1}\DD c(x_k)^{*,\mathcal{E}}\big(\DD c(x_k)B_k^{-1}\DD
    c(x_k)^{*,\mathcal{E}}\big)^{-1}\DD c(x_k) \right)B_k\inv\nabla_{\mathcal{E}} f(x_k) \\
 - B_k^{-1}\DD c(x_k)^{*,\mathcal{E}}\big(\DD c(x_k)B_k^{-1}\DD c(x_k)^{*,\mathcal{E}}\big)^{-1}c(x_k),
\end{multline}
which is \cref{eqn:i0xfv} with the metric $g(\xi,\zeta):=\langle \xi,B_k \zeta  \rangle$.
\end{remark}

Proposition~\ref{prop:6nqa8} shows that, for any matrix $B_k$, we can 
choose a metric such that the 
Riemannian landing~\eqref{eq:pseudoinverse_landing} matches the SQP step.
In fact, SQP is fully equivalent to~\cref{eqn:i0xfv}: 
%In fact, the converse of	~\cref{prop:6nqa8} is also true. 
for any metric $g$, 
the proof of \cref{prop:6nqa8} shows that the SQP step~\cref{eq:SQP} can reproduce
step~\cref{eqn:i0xfv} by choosing
\[
B_k := \Proj_{x_k}^{*,\mathcal{E}}\, G_T(x_k)\, \Proj_{x_k}.
\] 
%Conversely, given a metric \(g\) at \(x_k\), the proof of \cref{prop:6nqa8}
%shows that the “landing’’ step \cref{eqn:i0xfv} can be obtained as the SQP step
%\(d_k\) solving \cref{eq:SQP} by choosing
%\[
%B_k := \Proj_{x_k}^{*,\mathcal{E}}\, G_T(x_k)\, \Proj_{x_k}.
%\]
Since~\cref{eqn:i0xfv} does not feature an operator $H(x)$ compared to~\cref{eq:pseudoinverse_landing}, the SQP algorithm~\cref{eq:SQP}
 constitutes a strict subclass
of the Riemannian landing method \cref{eq:pseudoinverse_landing}. In SQP, the iterate
is completely determined by the choice of the normal space and the tangent
metric, whereas the landing framework additionally allows one to tune the normal
metric for the normal component \cref{eqn:3dfjy}—or, equivalently, to choose the
operator \(H\) in \cref{eq:rlanding-gstar-normal}.

%\smallskip 
%\fg{Remove the next paragraph? Not sure that it brings much, 
%it's more important that they focus on the idea just above.}
%Lastly, it is interesting to note that, given a  matrix $B_k$ generating the
%SQP direction \cref{eq:SQP}, the normal component can be modified
%without changing the tangent component by replacing $B_k$ with
%$\Proj_{x_k}^{*,\mathcal{E}}B_k\Proj_{x_k}$, where $\Proj_{x_k}$ is any linear
%projector on $\Txk\Mxk$. After this change, the SQP direction reads 
%\cref{eqn:i0xfv} where $g$ is a metric for which $\Txk\Mxk$ and
%$\Nxk\Mxk:=\ker(\Proj_{x_k})$ 
%$g$-orthogonal 
%and satisfying $g(\xi,\zeta)=\langle \xi,B_k\zeta  \rangle_{\mathcal{E}}$ for any
%$\xi,\zeta\in\Txk\Mxk$.

\subsection{All roads lead to SQP: link with the Riemannian Newton method}

\label{subsec:3ubkf}

We recalled earlier that the SQP method \cref{eq:SQP} 
is quadratically convergent with $B_k=\nabla_{\mathcal{E}}^{2}
\mathcal{L}(x_k,\lambda_k)$ where $\lambda_{k}$ is the Lagrange multiplier
of the quadratic program at the previous iteration.
In fact, this makes SQP equivalent to Newton's method for
root finding applied to the function
$F(x,\lambda):=(\nabla_{\mathcal{E}}\mathcal{L}(x,\lambda),c(x))$
\cite[Chapter 18]{nocedal2006numerical}. 
In view of \cref{prop:6nqa8}, 
this implies that in order to  make the landing method 
\begin{equation}
\label{eqn:ot3ro}
x_{k+1}=x_k-\grad^{g}_{\Mxk}f(x_k)-\DD c(x_k)^{\dagger,g}c(x_k)
\end{equation}
locally quadratically
convergent, it is sufficient to 
choose a metric $g$ such that
\begin{enumerate}[(i)]
    \item  the normal space at $x_k$ is
        $\Nxk\Mxk:=(\nabla_{\mathcal{E}}^{2}\mathcal{L}(x_k,\lambda_k)\Txk\Mxk)^{\perp,\mathcal{E}}$;
    \item
    the tangent metric is $G_T(x_k)=\nabla_{\mathcal{E}}^{2}
        \mathcal{L}(x_k,\lambda_k)$.
\end{enumerate}\medskip

On the other hand, it is known since 
\citep{edelman1998geometry,absil_all_2009} that when setting 
$G_T(x)=\nabla^{2}_{\mathcal{E}}\mathcal{L}(x,\lambda^{*}(x))$ and 
$\Proj_x=\Proj_{x,\mathcal{E}}$ the $\mathcal{E}-$orthogonal projection,
the tangent term becomes a Riemannian Newton step:
\[
    -\grad^{g}_{\Mx}f(x)=-\Hess^{\mathcal{E}}_{\Mx}f(x)^{-1}\grad^{g^{\mathcal{E}}}_{\mathcal{M}_x}f(x).
\] 
Here, 
$\Hess^{\mathcal{E}}_{\Mx}f(x)=\Proj_{x,\mathcal{E}}
\nabla^{2}_{\mathcal{E}}\mathcal{L}(x,\lambda(x))\Proj_{x,\mathcal{E}}$
is the Riemannian Hessian with respect to the Riemannian connection associated
to the Euclidean metric;
this follows from the fact that $\xi:=\grad^{g}_{\Mx}f(x)$ satisfies 
\[
   \widetilde{G_T}(x)\xi =  \Proj_{x,\mathcal{E}}\nabla^{2}_{\mathcal{E}}\mathcal{L}(x,\lambda(x))\Proj_{x,\mathcal{E}} \xi = -\Proj_{x,\mathcal{E}}
    \nabla_{\mathcal{E}}f(x)=-\grad^{g^{\mathcal{E}}}_{\mathcal{M}_x}f(x).
\]
In \citep{absil_all_2009}, this property enabled to interpret 
the Feasibly
Projected SQP (FP-SQP) method, i.e.
\begin{equation}
\label{eqn:g6lr2}
x_{k+1}=R_{x_k}(\alpha_k d_k),
\end{equation}
with $d_k$ being the SQP direction \cref{eq:SQP} and
$B_k=\nabla^{2}_{\mathcal{E}}\mathcal{L}(x_k,\lambda_k^{*}(x_k))$, 
as 
 a Riemannian Newton method on the manifold
$\mathcal{M}$.
Here,
the multiplier is not the SQP multiplier 
$\lambda_k$ as above but the least-squares multiplier $\lambda^{*}(x_k)$ where 
\[
\lambda^{*}(x) = (\DD c(x)\DD c(x)^{*,\mathcal{E}})^{-1}\DD
c(x)\nabla_{\mathcal{E}}f(x),
\] 
however, both choices are asymptotically equivalent since $\lambda_k\rightarrow \lambda^{*}(x^{*})$
as $k\rightarrow +\infty$.

The FP-SQP method \cref{eqn:g6lr2} combines 
the SQP step $d_k$ with 
$R_{x_k}$ 
a retraction on $\mathcal{M}$ to ensure the feasibility of the iterates at every
iteration.  

\medskip

In what follows, we deepen this connection, by showing that the conditions (i)
and (ii) are in some sense necessary to make \cref{eqn:ot3ro} locally quadratically
convergent.
Consider the iterative scheme  
\begin{equation}
\label{eqn:je105}
x_{k+1}=x_k-\Hess^{g,A}_{\Mxk} f(x_k)^{-1} \grad^g_{\Mxk}f(x_k)-\DD c(x_k)^{\dagger,g}
c(x_k),
\end{equation}
which combines a Riemannian-Newton tangent step
$-\Hess^{g,A}_{\Mx} f(x_k)^{-1} \grad^g_{\Mxk}f(x_k)$ on the manifold $\Mxk$
with the `Newton-like'
pseudoinverse normal step $-\DD c(x_k)^{\dagger,g}
c(x_k)$
aiming to reduce the constraint violation. 
The scheme \cref{eqn:je105} can be seen as a particular case of the landing
method \cref{eq:pseudoinverse_landing}, up to redefining the metric in the tangent space. 

Let us recall that the Riemannian Hessian is defined
with respect to the metric $g$ and the choice of an affine connection
$\nabla^{A}$ by the formula
\[
\Hess^{g,A}_{\Mx} f(x) \xi:=\nabla^{A}_\xi \grad^{g}_{\Mx}f(x),
\] 
where an affine  connection $\nabla^{A}$ is a bilinear operator on the tangent
bundle that verifies Leibniz
rule: $\nabla^{A}_{\xi}\zeta\in\Tx\Mx$ and 
$\nabla^{A}_\xi(f \zeta)=\DD_{\xi} f \zeta + f\nabla^{A}_{\xi}\zeta$ for any
smooth real function $f$ and any $\xi,\zeta\in\Tx\Mx$~\citep[chap. 5]{absil_all_2009}.
The superscript $\cdot^{A}$ in the notation  $\nabla^{A}$ of the affine
connection is suggested by the next lemma, which states that affine connections
on a Riemannian submanifold can be parameterized by the choice of a tangent
bilinear term $A(\xi,\zeta)$.
\begin{lemma}
    \label{lem:1fkkf}
    Any affine connection on $\Mx$ can be written as 
    \begin{equation}
    \label{eqn:0w0od}
\nabla_{\xi} \eta \equiv\nabla_\xi^{A}\eta:= \Proj_x(\DD_{\xi}\eta) +A(\xi,\eta),
    \end{equation}
    where $A(\xi,\eta)\,:\, \Tx\Mx\times \Tx\Mx \rightarrow \Tx\Mx$ is a
    vectorial
    bilinear form
with values 
    on the tangent space $\Tx\Mx$,
    and $\Proj_x$ is any projection operator on $\Tx\Mx$.
    Reciprocally, \cref{eqn:0w0od} defines an
    affine connection. 
\end{lemma}
\begin{proof}
According to Gauss formula,     any affine connection can be written as
\begin{equation}
\label{eqn:7zhi6}
    \nabla_{\xi}\eta =\D_{\xi}\eta - \Gamma(\xi,\eta),
\end{equation}
where $\Gamma(\xi,\eta)$ is the Christoffel symbol of the connection. Since
$\nabla_{\xi}\eta\in\Tx\Mx$, we must have 
\[
\Proj_x^{\perp}(\Gamma(\xi,\eta))=\Proj_x^{\perp}(\DD_{\xi}\eta),
\] 
so that $\Gamma(\xi,\eta)=\Proj_x^{\perp}(\DD_{\xi}\eta) - A(\xi,\eta)$ with
$A(\xi,\eta):=-\Proj_x(\Gamma(\xi,\eta))$. Substituting this expression back
into \cref{eqn:7zhi6} yields \cref{eqn:0w0od}. Conversely, we verify that
for a given vector bilinear form $A(\xi,\eta)$, formula  \cref{eqn:0w0od}
satisfies all the axioms of an affine connection.
\end{proof}
Recall that  on the one hand, 
the scheme $x_{k+1}=x_k-\DD c(x_k)^{\dagger,g}c(x_k)$ generates iterates that
approach the feasible set at a quadratic rate, and on the other hand,
the Riemannian-Newton method \cref{eqn:g6lr2} with $d_k=-\Hess^{A,g}_{\Mxk}\grad^{g}_{\Mxk}$ 
is locally quadratically convergent to a critical point; these properties hold
for any choice of metric $g$ and any choice of affine connection $\nabla^{A}$
\citep{absil2008optimization}.
It could therefore be tempting to think that
the scheme \cref{eqn:je105} combining this two directions  should 
be automatically locally quadratically 
convergent near critical points of $f$ on the manifold $\mathcal{M}$. 

We show in the following that this is not true: the local quadratic convergence
is obtained only by selecting a metric $g$ such that the normal space at $x_k$
is
$\Nxk\Mxk=(\nabla^{2}_{\mathcal{E}}\mathcal{L}(x_k,\lambda^{*}(x_k))\Tx\Mx)^{\perp,\mathcal{E}}$
(or a
convergent perturbation of it). This means that quadratic convergent iterates
``land'' on the manifold $\mathcal{M}$ only through a preferred distinguished
direction of the space, which is exactly the one taken by SQP. \medskip 

We prove this result in two steps. 
First, we observe that the choice of affine connection 
 influences the Hessian only up to a term vanishing near critical points.
\begin{proposition}
There exists a bilinear tangent mapping 
$\widetilde{A}_g\,:\, \Tx\Mx\times\Tx\Mx\to \Tx\Mx$ continuous with respect to
$g$ and $x$ 
such that the Hessian $\Hess^{g,A}_{\Mx} f(x)\,:\,\Tx\Mx \to \Tx\Mx$ can be rewritten
as 
\begin{equation}
\label{eqn:z9tjw}
\Hess^{g,A}_{\Mx}f(x)[\xi]=\widetilde{G_T}(x)^{-1}\Hess^{\mathcal{E}}_{\Mx}f(x)[\xi]
    +\widetilde{A}_g(\xi,\Proj_{x,\mathcal{E}}\nabla_{\mathcal{E}}f),
\end{equation}
where 
\begin{enumerate}[(i)]
    \item $\Proj_{x,\mathcal{E}}\,:\, \mathcal{E}\to \Tx\Mx$ is the $\mathcal{E}$-orthogonal projection operator  
        on $\Tx\Mx$; 
    \item 
        $\widetilde{G_T}(x)^{-1}\,:\, \Tx\Mx\to\Tx\Mx$ is the 
        mapping characterized by 
        \[
            \widetilde{G_T}(x)^{-1}\xi\in\Tx\Mx \text{ and }
            g(\widetilde{G_T}(x)^{-1}\xi,\zeta)=\langle \xi,\zeta
            \rangle_{\mathcal{E}}
        \text{ for any }\xi,\zeta\in\Tx\Mx;
        \]
    \item $\Hess^{\mathcal{E}}_{\Mx}$ is the Riemannian Hessian with respect to
        the Euclidean metric $g\equiv g^{\mathcal{E}}$ on $\Mx$. It reads 
        \[
            \begin{aligned}
                \Hess^{\mathcal{E}}_{\Mx}[\xi] &
            =\Proj_{x,\mathcal{E}}\nabla^{2}_{\mathcal{E}}\mathcal{L}(x,\lambda^{*}(x))\Proj_{x,\mathcal{E}}
            \xi\\
            &=\Proj_{x,\mathcal{E}}[\D_{\xi}\nabla_{\mathcal{E}}f(x) +
            \DD_{\xi}\Proj_{x,\mathcal{E}}
            ((\Id_{\mathcal{E}}-\Proj_{x,\mathcal{E}})\nabla_{\mathcal{E}}f(x))],\qquad \text{ for all }
            \xi\in\Tx\Mx,
            \end{aligned}
        \]
        where $\lambda^{*}(x):=(\DD c(x)\DD c(x)^{*,\mathcal{E}})^{-1}\DD
        c(x)\nabla_{\mathcal{E}}f(x)$ is the least-squares multiplier. 
\end{enumerate}
\end{proposition}
\begin{proof}
    Due to \cref{eqn:0w0od}, 
    the Hessian $\Hess^{g,A}_{\Mx} f(x)\,:\,\Tx\Mx\to\Tx\Mx$
    with respect to the connection
$\nabla^{A}$ and the metric $g$
is the operator defined by 
\[
\Hess^{g,A}_{\Mx} f(x)[\xi] := \nabla_{\xi}^{A}\grad^{g}_{\Mx}
f=\Proj_x(\DD_{\xi} \grad^{g}_{\Mx}f)+A(\xi,\grad^{g}_{\Mx}f),
\qquad \text{ for all } \xi\in \Tx\Mx,
\] 
where for the moment, we don't specify the tangent projection operator $\Proj_x$.
Since
$\grad^{g}_{\mathcal{M}}f=\widetilde{G_T}(x)^{-1}\Proj_x^{*,\mathcal{E}}\nabla_{\mathcal{E}}f$,
we find 
\begin{equation}
\label{eqn:hpvei}
    \begin{aligned}
        \Hess^{g,A}_{\Mx} f(x)[\xi]  & = \Proj_x\left(\D_{\xi}
(\widetilde{G_T}(x)^{-1}\Proj_x^{*,\mathcal{E}}\nabla_{\mathcal{E}}f)\right) 
+A(\xi,\widetilde{G_T}(x)^{-1}\Proj_x^{*,\mathcal{E}}\nabla_{\mathcal{E}}f)\\
&= \Proj_x
(\widetilde{G_T}(x)^{-1}\D_{\xi}\left[\Proj_x^{*,\mathcal{E}}(\nabla_{\mathcal{E}}f)\right]) \\
& \quad +\Proj_x\left(
\DD_{\xi}(\widetilde{G_T}(x)^{-1}\Proj_x^{*,\mathcal{E}})\Proj_x^{*,\mathcal{E}}\nabla_{\mathcal{E}}f
\right)
+A(\xi,\widetilde{G_T}(x)^{-1}\Proj_x^{*,\mathcal{E}}\nabla_{\mathcal{E}}f),
    \end{aligned}
\end{equation}
where we note that 
$\widetilde{G_T}(x)^{-1}\Proj_x^{*,\mathcal{E}}=G(x)^{-1}\Proj_x^{*,\mathcal{E}}$
is differentiable as a mapping $\mathcal{E}\to \mathcal{E}$ since 
$G(x)^{-1}$ and $\Proj_x^{*,\mathcal{E}}$ are differentiable. 
Let us then recall the \emph{Weingarten identities} for the dual projector 
$\Proj_x^{*,\mathcal{E}}$ \cite[Proposition
4]{feppon_extrinsic_2019}:
\[
    \DD_{\eta}\Proj_x^{*,\mathcal{E}}(\xi)=
    (\Proj_x^{\perp})^{*,\mathcal{E}}(\DD_{\eta}\xi) \text{ for any
    }\eta\in\Tx\Mx,\xi\in \mathcal{C}^{\infty}(\Mx,\Nx\Mx^{\perp,\mathcal{E}}),
\] 
\[
    \DD_{\eta}\Proj_x^{*,\mathcal{E}}(\zeta)=-\Proj_x^{*,\mathcal{E}}(\DD_{\eta}\zeta)
    \text{ for any
    }\eta\in\Tx\Mx,\zeta\in \mathcal{C}^{\infty}(\Mx,\Tx\Mx^{\perp,\mathcal{E}}).
\]
These identities can be obtained by differentiating the identities 
$\Proj_x^{*,\mathcal{E}}\xi = \xi$ and $\Proj_x^{*,\mathcal{E}}\zeta = 0 $ with
respect to $\eta\in\Tx\Mx$ for any smooth vector field $\xi$ with values in
$\Nx\Mx^{\perp,\mathcal{E}}$ and any smooth vector field $\zeta$ with values in
$\Tx\Mx^{\perp,\mathcal{E}}$. 
Using $\mathrm{Range}(\widetilde{G_T}(x)^{-1})=\Tx\Mx$ and the first Weingarten
identity, we obtain 
\begin{multline}
\label{eqn:svjmu}
    \Proj_x(\widetilde{G_T}(x)^{-1}\D_{\xi}[\Proj_x^{*,\mathcal{E}}\nabla_{\mathcal{E}}f(x)])
    \\
    =\widetilde{G_T}(x)^{-1}\Proj_x^{*,\mathcal{E}}\DD_{\xi}
    \Proj_x^{*,\mathcal{E}}\nabla_{\mathcal{E}}f(x) 
    +\widetilde{G_T}(x)^{-1}\Proj_x^{*,\mathcal{E}}(\DD_{\xi}\nabla_{\mathcal{E}}f(x)).
\end{multline}
Introducing
$\widetilde{A}_g(\xi,\zeta):=\Proj_x(\DD_{\xi}(\widetilde{G_T}(x)^{-1}\Proj_x^{*,\mathcal{E}})\zeta)
+A(\xi,\widetilde{G_T}(x)^{-1}\zeta),$ \cref{eqn:hpvei,eqn:svjmu} lead
to 
\[
\Hess^{g,A}_{\Mx}f(x)[\xi]=\widetilde{G_T}(x)^{-1}\Proj_x^{*,\mathcal{E}}\left(\DD_\xi\nabla_{\mathcal{E}}f(x)
+
\DD_\xi\Proj_x^{*,\mathcal{E}}((\Proj_x^{\perp})^{*,\mathcal{E}}\nabla_{\mathcal{E}}f(x)\right)
    +\widetilde{A}_g(\xi,\Proj_{x,\mathcal{E}}\nabla_{\mathcal{E}}f).
\] 
The projection perator $\Proj_x$ can be chosen freely 
in the definition of the covariant derivative by changing the operator
$A$ according to \cref{lem:1fkkf} and when writing 
$\widetilde{G_T}(x)^{-1}\Proj_x^{*,\mathcal{E}}$,
according to 
\cref{rmk:xxqnr}. 
The result follows by choosing $\Proj_x=\Proj_{x,\mathcal{E}}$
 to be the $\mathcal{E}$-orthogonal projector $\Proj_{x,\mathcal{E}}$ on $\Tx\Mx$, which is 
self-adjoint. 
The point (iii) can be found in \citep{absil_all_2009}.
\end{proof}
The term in \cref{eqn:z9tjw} 
featuring $A$ vanishes at critical points $x^{*}$ satisfying 
$\Proj_{x,\mathcal{E}}\nabla_{\mathcal{E}}f(x)=0$. This implies that 
the tangent Newton step in \cref{eqn:je105}
does not significantly depend on the affine connection $\nabla^{A}$ or
on the tangent metric $G_T(x)$.
\begin{lemma}
   In the vicinity of a critical point $x^{*}$ with
   $\grad^{g}_{\mathcal{M}_{x^{*}}}f(x^{*})=0$ and
   $\Hess^{\mathcal{E}}_{\Mx}f(x^{*})$ invertible, all
   Riemannian Newton steps agree up to a quadratically vanishing error: 
   \begin{equation}
\label{eqn:lk8i3}
 - \Hess^{g,A}_{\Mx}f(x)^{-1}\grad^{g}_{\Mx}f(x) = -(\Hess^{\mathcal{E}}_{\Mx}f(x))^{-1}\Proj_{x,\mathcal{E}}\nabla_{\mathcal{E}}f(x) 
+\calO(\norm{\Proj_{x,\mathcal{E}}\nabla_{\mathcal{E}}f(x)}^{2}_{\mathcal{E}}).
   \end{equation} 
\end{lemma}
\begin{proof}
 Let    
$d_T(x)=-(\Hess^{g,A}_{\Mx}f(x))^{-1}\grad_{\Mx}^{g}f(x)$.
It holds $\Hess^{g,A}_{\Mx}f(x)
d_T(x)=-\grad_{\Mx}^{g}f(x)=-\widetilde{G_T}(x)\Proj_x^{*,\mathcal{E}}
\nabla_{\mathcal{E}}f(x)$.
Using \cref{eqn:z9tjw}, left multiplying by $\Proj_{x,\mathcal{E}}\widetilde{G_T}(x)$ and using
$\Proj_{x,\mathcal{E}}\Proj_x^{*}=\Proj_{x,\mathcal{E}}$, we find
\[
    (\Hess^{\mathcal{E}}_{\Mx}+\calO(\Proj_{x,\mathcal{E}}\nabla_{\mathcal{E}}f(x)))d_T(x)=-\Proj_{x,\mathcal{E}}\nabla_{\mathcal{E}}f(x).
\] 
The result follows by a standard perturbation analysis.
\end{proof}

We are now in position to prove that 
Riemannian-Newton landing iterations \cref{eqn:je105} converge quadratically 
when choosing
the normal space
$\Nx\Mx=(\nabla^{2}_{\mathcal{E}}\mathcal{L}(x,\lambda^{*}(x)))^{\perp,\mathcal{E}}$,
characterizing the normal step $d_N(x)=-\DD
c(x)^{\dagger,g}c(x)=-\Proj_x^{\perp}\DD c(x)^{\dagger,\mathcal{E}}c(x)$. To
make the proof more readable, 
we don't explicit here the various Lipschitz constants involved in the convergence
estimates, absorbing them in the $O(\cdot)$ notation.
\begin{proposition}
    \label{prop:94jvx}
    Assume that $c$ and $f$ are continuous twice differentiable mappings on
    $\mathcal{E}$.
    Let $x^{*}\in\M$ be a local minimizer of \cref{eq:P} with
  Riemannian Hessian  $\Hess^{\mathcal{E}}_{\mathcal{M}_{x^{*}}}f(x^{*})$ positive definite at
  $x^{*}$. If the metric $g$
  is set such that the normal space is  
  $\Nxk\Mxk=(\nabla^{2}_{\mathcal{E}}\mathcal{L}(x_k,\lambda^{*}(x_k)))^{\perp,\mathcal{E}}$,
  then, 
  for any $x_0\in \mathcal{E}$ sufficiently close to $x^{*}$, the
  Riemannian-Newton landing 
  iterates \cref{eqn:je105} are well-defined and converge quadratically to
  $x^{*}$: there exists a
  constant $\gamma>0$ independent of $k\in\N$ such that 
  \[
      \norm{x_{k+1}-x^{*}}_{\mathcal{E}}\< \gamma
      \norm{x_{k}-x^{*}}_{\mathcal{E}}^{2},\qquad \text{ for all } k\in\N.
  \] 
\end{proposition} 
\begin{proof}
    Due to \cref{eqn:je105} and \cref{eqn:lk8i3}.
    \begin{equation}
    \label{eqn:0k87v}
    \begin{aligned}
    x_{k+1}-x^{*} & =x_k-x^{*}-\Hess^{g,A}_{\Mxk} f(x_k)^{-1} \grad^g_{\Mxk}f(x_k)-\DD c(x_k)^{\dagger,g}
c(x_k)\\
&=x_k-x^{*}-\Hess^{\mathcal{E}}_{\Mxk}f(x_k)^{-1}\Proj_{x_k,\calE}\nabla_{\mathcal{E}}f(x_k) 
-\DD
c(x_k)^{\dagger,g}c(x_k)+\calO(\norm{\Proj_{x_k,\calE}\nabla_{\mathcal{E}}f(x_k)}^{2}_{\mathcal{E}}).
    \end{aligned}
    \end{equation}
    Since
    $\Proj_{x^{*},\calE}\nabla_{\mathcal{E}}f(x^{*})=\grad^{\mathcal{E}}_{\mathcal{M}_{x^{*}}}f(x^{*})=0$
    and $x\mapsto \Proj_{x,\mathcal{E}}\nabla_{\mathcal{E}}f(x)$ is a Lipschitz map, 
    it is clear that 
    \begin{equation}
    \label{eqn:ghtiq}
    \calO(\norm{\Proj_{x_k,\calE}\nabla_{\mathcal{E}}f(x_k)}^{2}_{\mathcal{E}})=
    \calO(\norm{x_k-x^{*}}^{2}_{\mathcal{E}}).
    \end{equation}
Then, by considering a Taylor expansion, we have on the one hand,
\begin{equation}
\label{eqn:ixnh1}
0  =c(x^{*})
      =c(x_k)+\DD c(x_k)[x^{*}-x_k]+\calO(\norm{x_k-x^{*}}^{2}_{\mathcal{E}}),
\end{equation}
which implies 
\[
    \Proj_{x_k}^{\perp}(x^{*}-x_k) =\DD c(x_k)^{\dagger,g}\DD
    c(x_k)[x^{*}-x_k]=-\DD c(x_k)^{\dagger,g}c(x_k) +
    \calO(\norm{x_k-x^{*}}^{2}_{\mathcal{E}}).
\] 
Inserting into \cref{eqn:0k87v}, we obtain 
\begin{equation}
\label{eqn:g2emq}
x_{k+1}-x^{*}=\Proj_{x_k,\calE}(x_k-x^{*})-\Hess^{\mathcal{E}}_{\Mxk}f(x_k)^{-1}
\Proj_{x_k,\calE}
\nabla_{\mathcal{E}}f(x_k)
+\calO(\norm{x_k-x^{*}}^{2}_{\mathcal{E}}).
\end{equation}
On the other hand, performing a Taylor expansion of $x\mapsto
\Proj_{x,\mathcal{E}}\nabla_{\mathcal{E}}f(x)$ around $x_k$ and multiplying by
 $\Proj_{x_k,\calE}$,
\begin{equation}
\label{eqn:qep94}
\begin{aligned}
0&=\Proj_{x_k,\calE}\Proj_{x^{*},\calE}\nabla_{\mathcal{E}}f(x^{*})\\
&=\Proj_{x_k,\calE}\Proj_{x_k,\calE}\nabla_{\mathcal{E}}f(x_{k})+
\Proj_{x_k,\calE}\DD
[\Proj_{\calE}\nabla_{\mathcal{E}}f](x_{k})[x^{*}-x_k]
+\calO(\norm{x_k-x^{*}}_{\mathcal{E}}^{2}).
\end{aligned}
\end{equation}
Recall that for any $\xi\in \mathcal{E}$ \citep{absil_all_2009},
\begin{equation*}
	\begin{aligned}
	\Proj_{x,\mathcal{E}}    \DD_{\xi}(\Proj_{\calE}\nabla_{\mathcal{E}}f)(x)
    &=\Proj_{x,\mathcal{E}}(\DD_{\xi}\Proj_{x,\calE}[\Proj_{x,\mathcal{E}}^{\perp}\nabla_{\mathcal{E}}f(x)]
    +\DD_{\xi}\nabla_{\mathcal{E}}f(x))\\
    &=\Proj_{x,\mathcal{E}}\nabla^{2}_{\mathcal{E}}\mathcal{L}(x,\lambda^{*}(x))[\xi].
\end{aligned}
\end{equation*}
Consequently, 
\[
    \begin{aligned}
    \Proj_{x_k,\calE}\DD(\Proj_{\calE}\nabla_{\mathcal{E}}f)(x_k)[x^*-x_k]
     & =\Proj_{x_k,\calE}\nabla^{2}_{\mathcal{E}}\mathcal{L}(x_k,\lambda^{*}(x_k))[x^{*}-x_k]
    \\
     & =\Proj_{x_k,\calE}\nabla^{2}_{\mathcal{E}}\mathcal{L}(x_k,\lambda^{*}(x_k))\Proj_{x_k}(x^{*}-x_k)\\
    &~+\Proj_{x_k,\calE}\nabla^{2}_{\mathcal{E}}\mathcal{L}(x_k,\lambda^{*}(x_k))\Proj_x^{\perp}(x^{*}-x_k)\\
    &=\Hess_{\Mxk}^{\mathcal{E}}f(x_k) \Proj_{x_k}(x^{*}-x_k)  \\
    &~+\Proj_{x_k,\calE}\nabla^{2}_{\mathcal{E}}\mathcal{L}(x_k,\lambda^{*}(x_k))\Proj_x^{\perp}(x^{*}-x_k)\\
    &\quad +\calO(\norm{x_k-x^{*}}^{2}_{\mathcal{E}}).
    \end{aligned}
\] 
Substituting into  \cref{eqn:g2emq} after using \cref{eqn:qep94}, 
we infer 
\begin{equation}
\label{eqn:edvos}
    x_{k+1}-x^{*}  =
    \Hess^{\mathcal{E}}_{\Mxk}f(x_k)^{-1}\Proj_{x_k,\calE}
    \nabla^{2}_{\mathcal{E}}\mathcal{L}(x_k,\lambda^{*}(x_k))\Proj_{x_k}^{\perp}(x^{*}-x_k)
+ \calO(\norm{x_k-x^{*}}^{2}_{\mathcal{E}}).
\end{equation}
Therefore, if
$\Nxk\Mxk=(\nabla^{2}_{\mathcal{E}}\mathcal{L}(x_k,\lambda^{*}(x_k)))^{\perp,\mathcal{E}}$, 
we have
$\nabla^{2}_{\mathcal{E}}\mathcal{L}(x_k,\lambda_k^{*})\Nxk\Mxk=\Txk\Mxk^{\perp,\mathcal{E}}$
and 
\[
    \Proj_{x_k,\calE}
    \nabla^{2}_{\mathcal{E}}\mathcal{L}(x_k,\lambda^{*}(x_k))\Proj_{x_k}^{\perp}=0.
\]
It follows that $x_{k+1}-x^{*}=\calO(\norm{x_k-x^{*}}^{2}_{\mathcal{E}})$, which
completes the proof. 
\end{proof}
If $\Nxk\Mxk\neq
(\nabla^{2}_{\mathcal{E}}\mathcal{L}(x_k,\lambda^*(x_k)))^{\perp,\mathcal{E}}$, 
\cref{eqn:edvos} shows that 
we only have a priori $x_{k+1}-x^{*}=\calO(\norm{x_k-x^{*}})$, which prevents quadratic
convergence. 

\section{Global convergence using adaptive step sizes}
\label{sec:globalization}

In this section, we design an adaptive step size 
globalization procedure for the landing method.
Namely, we propose a backtracking line search
based on an Armijo condition for the decrease of a merit function.
Our approach relies heavily on classical techniques for the globalization 
of SQP~\citep{nocedal2006numerical,curtis2024worst}, leveraging 
the fact---highlighted in
\cref{subsec:c4t76}---that SQP is a particular instance of the Riemannian landing.

We consider the landing update with the normal term written in pseudoinverse
form, given by 
 \begin{equation}
     \label{eqn:ly4nc}
  d(x) = \tangent(x) + \normal(x) = - \grad_{\Mx}^g f(x) - \DD
c(x)^{\dagger,g} H(x)
c(x)
\end{equation}
for some metric $g$ on $\calE$ and $H\colon \cset \to \cset$, as 
described in~\cref{subsec:bxujs}.

\subsection{Assumptions}
\label{sec:assumptions}
We introduce a set of generic assumptions on the manifold and Lipschitz constants
that will arise in our analysis. The assumptions are standard in the SQP literature 
(e.g.,~\cite{berahas2021sequential,curtis2024worst}). 
We first assume that 
there exists a constraint threshold
$R>0$ such that the following compactness assumption holds. 
\begin{assumption}\label{assu:boundedMandC}
    There exists a constant $R>0$ such that 
the set $\saferegion = \{ x \in \calE :
\|\c(x)\|_{\mathcal{F}} \leq R \}$ is compact, contains all the iterates
$(x_k)_{k\in\N}$ and trial points,  and $f$ is bounded from below on
$\mathcal{R}$:
\[
\inf_{x\in \mathcal{R}} f(x)>-\infty.
\] 
\end{assumption}
We then require $f$ and $c$ to be Lipschitz
continuous on $\mathcal{R}$.
\begin{assumption}\label{assu:lipschitz}
The functions $f$ and $c$ are continuously differentiable on $\mathcal{D}$ with
Lipschitz continuous derivatives. In
particular, there exist constants $\norm{\nabla_{\mathcal{E}}f}_{\infty},
\norm{\DD c}_{\infty}, L_f, L_c>0$ such that
\[
\norm{\nabla_{\mathcal{E}} f(x)}_{\mathcal{E}}\< \norm{\nabla_{\mathcal{E}}f}_{\infty}, \quad 
\sup_{v\in \mathcal{E}\backslash\{0\}}\frac{\|\DD
c(x)[v]\|_{\mathcal{F}}}{\norm{v}_{\mathcal{E}}}\<
\norm{\DD c}_{\infty}\qquad \text{ for all } x\in \mathcal{R}. 
\]
\label{lemma:Lip_fc}
\[
\|\nabla_{\mathcal{E}} f(x)-\nabla_{\mathcal{E}} f(y)\|_{\mathcal{E}}
\le L_f\|x-y\|_{\mathcal{E}}
\quad \text{ for all } x,y\in\saferegion,
\]
\[
\|\DD c(x)-\DD c(y)\|_{\mathcal{E}\to\mathcal{F}}
\le L_c\|x-y\|_{\mathcal{E}}
\quad \text{ for all } x,y\in\saferegion,
\]
where $\norm{A}_{\mathcal{E}\to \mathcal{F}}:=\sup\limits_{v\in
\mathcal{E}\backslash\{0\}} \norm{Av}_{\mathcal{F}}/\norm{v}$ denotes the operator
norm for linear operators $A\colon \mathcal{E}\to \mathcal{F}$.
\end{assumption}
Third, we assume the classical linear independence
constraint qualification (LICQ) in a neighborhood of the feasible set $\mathcal{M}$. 
\begin{assumption} \label{assu:ROI}
	There exists a positive constant $\sigmabar$ such that 
\[
\inf_{x\in \mathcal{R}}\sigma_{\min}(\DD c(x))=\inf_{x\in \mathcal{R}}\sigma_m(\DD
c(x))\>\sigmabar>0,
\] 
  where
        $\sigma_k(A)$ and $\sigmamin(A)$ denote the $k$\textrm{th} and the smallest
        (Euclidean)     singular value of a linear map $A$, respectively. 
        Equivalently, the Euclidean pseudoinverse of $\DD
        c(x)$ is bounded:
        \[
            \sup_{w\in \mathcal{F}\backslash\{0\}}  \frac{\norm{\DD
            c(x)^{\dagger,\mathcal{E}}w}_{\mathcal{E}}}{\norm{w}_{\mathcal{F}}}\< \norm{\DD
            c^{\dagger,\mathcal{E}}}_{\infty},
            \text{ for all } x\in \mathcal{R}.
        \] 
\end{assumption}
These bounds also imply the uniform boundedness of Riemannian gradients
$\nabla_g f(x)$,
adjoints $\DD c^{*,g}(x)$ and pseudoinverses $\DD c^{\dagger,g}$ 
by assuming uniform ellipticity of the metric.
\begin{assumption}\label{assu:mod-sqp}
    There exist uniform constants $\underline{g}>0$ and $\bar{g}>0$ such that 
    \[
        \underline{g}\norm{\xi}^{2}_{\mathcal{E}} \< g(\xi,\xi)=\langle G(x)\xi,\xi
        \rangle_{\mathcal{E}}\< \bar{g}\norm{\xi}^{2}_{\mathcal{E}},  \text{ for all } x\in \mathcal{R} \text{ and } \xi\in \mathcal{E}.
    \] 
\end{assumption}
This assumption implies the following bound which will be useful in the proof of
\cref{prop:moc7l}.
\begin{lemma}
    \label{lem:gyrct}
    Assuming \cref{assu:ROI},  the $g$- pseudo inverse can be bounded in terms of the Euclidean
    pseudoinverse:
    \[
        \sup_{w\in \mathcal{F}\backslash\{0\}} \frac{\norm{\DD
        c(x)^{\dagger,g}w}_{\mathcal{E}}}{\norm{w}_{\mathcal{F}}}
        \< \sqrt{\frac{\overline{g}}{\underline{g}}}
        \sup_{w\in \mathcal{F}\backslash\{0\}} \frac{\norm{\DD
        c(x)^{\dagger,\mathcal{E}}w}_{\mathcal{E}}}{\norm{w}_{\mathcal{F}}}\<
    \sqrt{\frac{\overline{g}}{\underline{g}}} \norm{\DD c^{\dagger,\mathcal{E}}}_{\infty},
        \text{ for all } x\in\mathcal{R}.
    \] 
\end{lemma}
\begin{proof}
    This follows from the fact that $\DD
    c(x)^{\dagger,g}w$ is the vector $\xi\in \mathcal{E}$ of smallest
    $g$-norm satisfying $\DD c(x) \xi=w$. Since $\DD
    c(x)^{\dagger,\mathcal{E}}w$ is another vector satisfying this property, 
    this implies that 
    \[
        \sqrt{\underline{g}}\norm{\DD c(x)^{\dagger,g}w}_{\mathcal{E}}  \<\norm{\DD c(x)^{\dagger,g}w}_g\<\norm{\DD c(x)^{\dagger,\mathcal{E}}w}_{g}\<
    \sqrt{\overline{g}} \norm{\DD c(x)^{\dagger,\mathcal{E}}w}_{\mathcal{E}}.
    \] 
\end{proof}
Finally, we also assume the uniform ellipticity of the operator field $H(x)$ involved
in the normal space direction \cref{eq:rlanding-gstar-normal}.
\begin{assumption}
    \label{assu:g387q}
    There exist uniform constants $\lambda_{\max}(H)>0$ and $\lambda_{\min}(H)>0$ such that 
    \[
        \lambda_{\min}(H)\norm{\zeta}^{2}_{\mathcal{F}}   \<  \langle\zeta,H(x)\zeta
        \rangle_{\mathcal{F}}\< \lambda_{\max}(H)\norm{\zeta}^{2}_{\mathcal{F}}, \quad \text{ for all }  x\in \mathcal{R} \text{ and }  \zeta\in \mathcal{F}.
    \] 
\end{assumption}
\begin{remark}
    Assumption \cref{assu:g387q} holds  in both cases 
    $H(x)=\Id_{\mathcal{F}}$ and $H(x)=\DD c(x)\DD c(x)^{*,g}$.
\end{remark}

\subsection{Line search procedure}

 The line search procedure is based on decreasing the $\ell_2$-merit function
\begin{align}\label{eq:phi}
	\phi_{\mu}(x) = f(x) + \mu \norm{c(x)}_\cset
\end{align}
for a large enough penalty parameter $\mu$. 
We perform a backtracking line-search on the $\ell_2$-merit
\(
\phi_\mu(x)\)
with Armijo parameters $\eta\in(0,1)$ and $\beta\in(0,1)$ as described in
\cref{algo:LSLanding}.

\begin{algorithm}
\caption{Riemannian Landing Method}
\label{algo:LSLanding}
\begin{algorithmic}[1]
\Require choose $\eta\in(0,\tfrac12)$,\; $\beta\in(0,1)$;
$\rho\in(0,\lambda_{\min}(H)/2)$; 
 initial  $x_0 \in \calE$
%\State Evaluate $f_0,\ \nabla f_0,\ c_0,\ J_0$
\State $\mu_0 = 1$
\State $k=0$
\Repeat
\State Compute $d_k=d(x_k)$, the landing direction \cref{eqn:ly4nc} at $x_k$ 
    \State Set $\mu_k=\mu_{k-1}$ if $c(x_k)=0$ and 
    \begin{equation}
    \label{eqn:3mn2m} \mu_k = \max\left\{\mu_{k-1}, \dfrac{\DD
f(x_k)\normal(x_k)}{\rho \norm{c(x_k)}_\cset}\right\}\text{ if }c(x_k)\neq 0.\end{equation} 
        \State Set $\alpha_k \gets 1$
    \While{$\phi_{\mu_k}(x_k+\alpha_k d_k) > \phi_{\mu_k}(x_k) + \eta\,\alpha_k\,\DD\phi_{\mu_k}(x_k)d_k$}
        \State $\alpha_k \gets \beta\,\alpha_k$ \Comment{backtracking}
    \EndWhile
    \State $x_{k+1} \gets x_k + \alpha_k d_k$
\Until{convergence}
\end{algorithmic}
\end{algorithm}

The following lemma gives the directional derivative of the 
merit function $\phi_{\mu}$.
\begin{lemma}\label{lemma:Dphi}
If $x\in \calE$ is such that $\norm{c(x)}\neq 0$, 
the merit function $\phi_\mu$ is differentiable at $x$ and 
\begin{align}\label{eq:Dphi_smooth}
\DD \phi_\mu (x)d &= \DD f(x)d + 
	\mu \inner{\dfrac{c(x)}{\norm{c(x)}}}{\DD c(x)d}_{\cset}. 
\end{align}
	For $c(x) =0$, the merit function is not differentiable at $x$; 
	however,  it admits the following 
	 directional derivative in any direction $d\in \calE$:
	\begin{align}\label{eq:Dphi_nonsmooth}
		\DD \phi_\mu (x)d:=\lim_{t\rightarrow 0 }  
        \frac{\phi_{\mu}(x+td)-\phi_{\mu}(x)}{t}&= \DD f(x)d  + \mu\norm{\DD	c(x)d}_\cset.
		\end{align}
\end{lemma}

%\fg{It is likely also possible to use the formula
%$\mu(x) = \dfrac{\DD f(x)[d]}{\rho \norm{c(x)}}$,
% which means that the algorithm does not need to compute
% the normal component separately. If we do this, 
% we may need to rework the descent we get. 
%Probably we will get something like
%$\DD\phi_{\mu}(x)[d]
%\le
%-\rho \mu \|c(x)\|_{\mathcal{F}}$ (which we had before).
% Then a new argument is needed 
%in order to show $\lim \norm{\grad f} = 0$.}

In what follows, for $x\in \mathcal{R}$, 
we denote by
\begin{equation}
\label{eqn:63luz}
    d(x)=d_T(x)+d_N(x)
\end{equation}
the landing step with 
$d_T(x)=-\grad^{g}_{\mathcal{M}_x}f(x)$ and
 $d_N(x)=-\DD c(x)^{\dagger,g}H(x) c(x)$.
 
 The following result shows that if the penalty parameter is large enough, 
 the landing direction is a descent direction for the merit function $\phi_\mu$.
  
\begin{proposition}[Sufficient decrease of the merit function]
 Let $x\in \mathcal{R}$ and $ \rho \leq \lambdamin(H(x))/2 $, 
 then the directional derivative of the merit function
$\phi_{\mu}(x)$ in the landing direction $d(x)$~\eqref{eq:pseudoinverse_landing} satisfies
\begin{equation}\label{eq:Dphi}
\DD\phi_{\mu}(x)[d(x)]
\le
-\|\grad_{\Mx}^{g} f(x)\|_{g}^2
-\rho \mu \|c(x)\|_{\mathcal{F}}.
\end{equation}
provided that $c(x)=0$, or that $c(x)\neq 0$ and 
\begin{equation}\label{eq:mu_value}
	\mu \geq \dfrac{\DD f(x)\normal(x)}{\rho \norm{c(x)}_\cset}.
\end{equation}	
\end{proposition}
\begin{proof}
For $\rho \leq \lambdamin(H(x))/2$, we have
\begin{equation} 
    \label{eqn:g63at}
\dfrac{\DD f(x)\normal(x) \norm{c(x)}_{\mathcal{F}}}{\inner{c(x)}{H(x)c(x)}_{\cset}-\rho
\norm{c(x)}_\mathcal{F}^2} \leq \dfrac{\DD f(x)\normal(x)
\norm{c(x)}_{\mathcal{F}}}{(\lambdamin(H(x))-\rho) \norm{c(x)}_{\mathcal{F}}^2} \leq  \dfrac{\DD
f(x)\normal(x)}{\rho \norm{c(x)}_{\mathcal{F}}} \leq \mu.
\end{equation}
For $\norm{c(x)}\neq 0$, \cref{eq:Dphi_smooth} and $\DD c(x)d(x)=-H(x)c(x)$
imply
	\begin{equation*}
	\begin{aligned}
		\DD \phi_\mu (x)d(x) &= \DD f(x)\tangent(x) + 
		\DD f(x)\normal(x) + \mu \inner{\dfrac{c(x)}{\norm{c(x)}_{\cset}}}{\DD
        c(x)d(x)}_{\cset}\\
 &= g\left(\nabla_g f(x), - \grad_{\Mx}^g f(x)\right) 
+ \DD f(x)\normal(x) - \mu \dfrac{\inner{c(x)}{H(x)c(x)}_{\cset}}{\norm{c(x)}_{\cset}}\\
&\leq - \norm{\grad^g f(x)}_g^2 +
\dfrac{\mu}{\norm{c(x)}_{\cset}}(\inner{c(x)}{H(x)c(x)}_{\cset} - \rho
\norm{c(x)}_{\cset}^2) - \mu \dfrac{\inner{c(x)}{H(x)c(x)}_{\cset}}{\norm{c(x)}_{\cset}} \\
		&= - \norm{\grad^g f(x)}_g^2 - \mu\rho \norm{c(x)}_{\cset}, 
	\end{aligned}
		\end{equation*}
        where we have used \cref{eqn:g63at} in the third line. 
For $x\in \M$, where $c(x) =0$, the landing direction $d$ only
has a tangent component ($\normal(x) =0$) and
the directional derivative~\eqref{eq:Dphi_nonsmooth} becomes
\[
    \DD \phi_\mu (x)d(x) = \DD f(x)\tangent(x)
    = - \norm{\grad^g f(x)}_g^2.
\]
 Thus~\cref{eq:Dphi} holds for all $x\in \saferegion$.
\end{proof}

The following result shows that the penalty parameter $\mu$ remains bounded in the region
of interest $\saferegion$.

\begin{proposition}[Upper bound on the penalty parameter]
    \label{prop:moc7l}
Under the above assumptions, 
 the increasing merit parameter sequence $(\mu_k)_{k\>0}$ defined by 
\cref{eqn:3mn2m} in \cref{algo:LSLanding} 
remains bounded:
	\begin{equation}\label{eq:mu_bar}
\mu(x):= \dfrac{\DD f(x)\normal(x)}{\rho \norm{c(x)}_\cset} 
\leq \bar \mu: =
\frac{1}{\rho}\sqrt{\frac{\bar{g}}{\underline{g}}}
  \norm{\nabla_{\mathcal{E}} f}_\infty\norm{\DD c^{\dagger,\mathcal{E}}}_\infty
\lambdamax(H),\qquad \text{ for all } x \in \mathcal{R}.
	\end{equation}
In particular, the sequence $(\mu_k)_{k\in\N}$ is convergent. 
\end{proposition}
\begin{proof}
Using the Cauchy-Schwartz inequality and \cref{lem:gyrct}, 
we find
\begin{equation*}
	\begin{aligned}
    |\DD f(x)\normal(x)| &= \langle \nabla_{\mathcal{E}}f(x), - \DD
    c(x)^{\dagger,g}    H(x) c(x)  \rangle_{\mathcal{E}}
    \leq \sqrt{\frac{\bar{g}}{\underline{g}}}
    \norm{\nabla_{\mathcal{E}}f(x)}_{\mathcal{E}}\norm{\DD c^{\dagger,\mathcal{E}}}_\infty
\norm{H(x)c(x)}_{\mathcal{E}}\\
&\leq   \sqrt{\frac{\bar{g}}{\underline{g}}}
\norm{\nabla_{\mathcal{E}} f}_\infty\norm{\DD c^{\dagger,\mathcal{E}}}_\infty
\lambdamax(H)
\norm{c(x)}.
\end{aligned}
\end{equation*}
\end{proof}

In what follows, we use the following inequality, actually valid for any Euclidean norm. 
\begin{lemma}
Let $v, w\in\mathcal F$ 
with $v\neq 0$.
The following inequality holds:
\begin{equation}
\label{eq:norm_second_order}
\|v+w\|_{\mathcal F}
\le
\|v\|_{\mathcal F}
+\inner{ \frac{v}{\|v\|_{\mathcal F}}}{w}_{\mathcal F}
+\frac{\|w\|_{\mathcal F}^2}{2\|v\|_{\mathcal F}}.
\end{equation}	
\end{lemma}
\begin{proof}
	We have 
	\begin{equation}
			\Fnorm{v+w}^2 = \Fnorm{v}^2 + 2 \inner{v}{w}_\cset + \Fnorm{w}^2
			= \Fnorm{v}^2 (1+t). 
	\end{equation}
    with $t:=             (2\inner{v}{w} + \Fnorm{w}^2)/\Fnorm{v}^2$.
	Since $1+t = {\Fnorm{v+w}^2}/{\Fnorm{v}^2}\geq 0$,
	 it is clear that $t \geq -1$.
     Recalling the
	 elementary inequality $\sqrt{1+t} \leq 1 + {t}/{2}$, valid for $t\geq -1$,
     this implies 
	 \begin{equation}
	 		\Fnorm{v+w} = \Fnorm{v} \sqrt{1+t}
	 		\leq \Fnorm{v} \left(1+ \frac{t}{2}\right)
	 		= \Fnorm{v} \left( 1 + \dfrac{\inner{v}{w}}{\Fnorm{v}^2} 
	 		+ \dfrac{\Fnorm{w}^2}{2\Fnorm{v}^2} \right).
	 \end{equation}
	 
\end{proof}

The following lemma gives an upper bound on the merit function along the landing direction, 
which we use below to show finite termination of the line search. 
\begin{lemma}[Quadratic upper bound for the merit function]
For any $x\in\saferegion$  and $\alpha\ge 0$
such that $x+\alpha d(x)\in\saferegion$ with $d(x)$ the landing step
 \cref{eqn:63luz},
 it holds that
\begin{equation}
\label{eq:quadratic_upper_bound_phi}
\phi_\mu(x+\alpha d(x))
\le
\phi_\mu(x)+\alpha \DD\phi_\mu(x)d(x)
+\frac{\alpha^2}{2}\big(L_f+\mu L_c\big)\|d(x)\|_{\mathcal E}^2
+\frac{\alpha^2}{2}\mu\lambdamax(H)^2\,\|c(x)\|_{\mathcal F},
\end{equation}
where $\DD\phi_\mu(x)d$ denotes the one-sided
directional derivative of $\phi_\mu$ at $x$ along $d$, 
given by Lemma~\ref{lemma:Dphi}.
\end{lemma}
\begin{proof}
By the Lipschitz continuity of $\nabla_{\mathcal{E}} f$ (assumption
\cref{assu:lipschitz}), we can write 
\begin{equation}\label{eq:bound_f}
f(x+\alpha d(x))
\le
f(x) + \alpha \DD f(x)d(x)
+ \tfrac12 L_f \alpha^2\|d(x)\|_{\mathcal{E}}^2.
\end{equation}
%Indeed, 
%\[
%f(x+\alpha d)-f(x)-\alpha \DD f(x) d=\int_0^{\alpha}(\DD f(x+t d)-\DD f(x))d\, \D t
%=\int_0^{\alpha}\langle \nabla_{\mathcal{E}}f(x+td)-\nabla_{\mathcal{E}}f(x),d
%\rangle_{\mathcal{E}} \D t
%\] 
%So its norm is bounded by 	$\alpha^{2}||d||_{\mathcal{E}}^{2}/2 L_f$.
By the Lipschitz continuity of $\DD c$ and $\DD c(x)d(x) = \DD c(x)\normal(x) =
-H(x)c(x)$,  we also have\label{eqn:um57p}
\begin{equation}
\begin{aligned}
c(x+\alpha d(x)) & =c(x)+\alpha \DD c(x)d(x)+r(\alpha), \\
&= (\Id_{\mathcal{F}}-\alpha H(x))c(x)+r(\alpha)
\end{aligned}
\end{equation}
with $\|r(\alpha)\|_{\mathcal F}\le \tfrac12 L_c\,\alpha^2\|d(x)\|_{\mathcal
E}^2$.
%Indeed, 
%\[
%c(x+\alpha d)-c(x)-\alpha \DD c(x) d=\int_0^{\alpha}(\DD c(x+t d)-\DD c(x))d\, \D t
%\] 
%so its norm is bounded by $|\alpha|^{2} ||d||_{\mathcal{E}}^{2} L_c/2$.
%Moreover, the triangle inequality with $\DD c(x)d = \DD c(x)\normal(x) =
%-H(x)c(x)$ yields
%\begin{equation}\label{eq:ineq_c}
%\|c(x+\alpha d(x))\|_{\mathcal F}
%\leq \norm{c(x)+\alpha \DD c(x)d(x)} + \norm{r(\alpha)}
%= \norm{c(x)-\alpha H(x) c(x)} + \norm{r(\alpha)}.
%\end{equation}

\noindent
If $c(x)=0$, it follows that 
$\|c(x+\alpha d)\|_{\mathcal F}\le \|r(\alpha)\|_{\mathcal F}$, which imply
\begin{equation}
	\begin{aligned}
		\phi_\mu (x+\alpha d(x)) &= f(x+\alpha d(x)) 
		+ \mu \|c(x+\alpha d(x))\|_{\mathcal F}\\
		&\leq f(x) + \alpha \DD f(x)d(x)
+ \tfrac12 L_f \alpha^2\|d(x)\|_{\mathcal{E}}^2 
+ \mu \tfrac12 L_c\,\alpha^2\|d(x)\|_{\mathcal E}^2.
	\end{aligned}
\end{equation}
Moreover, \cref{eq:Dphi_nonsmooth} together with 
$\DD c(x)d(x)=0$ entail $\DD f(x) d(x)= \DD \phi_\mu (x)d(x)$.
Therefore~\eqref{eq:quadratic_upper_bound_phi} holds if $c(x) =0$. 

\noindent
We now consider the case $c(x) \neq 0$. 
Majoring the first line of~\cref{eqn:um57p} using 
\eqref{eq:norm_second_order} with $v=c(x)$ and $w=\alpha \DD c(x)d(x)$ gives
\begin{equation}
\begin{aligned}
\label{eq:Ctrue}
\|c(x+& \alpha d(x))\|_{\mathcal F}
\le
\|c(x)\|_{\mathcal F}
+\alpha\inner{ \frac{c(x)}{\|c(x)\|_{\mathcal F}}}{\DD c(x)d(x)}_{\mathcal F}
+\frac{\alpha^2}{2}\frac{\|\DD c(x)d(x)\|_{\mathcal F}^2}{\|c(x)\|_{\mathcal F}}
+\|r(\alpha)\|_{\mathcal F}\\
&\le
\|c(x)\|_{\mathcal F}
+\alpha\inner{ \frac{c(x)}{\|c(x)\|_{\mathcal F}}}{\DD c(x)d(x)}_{\mathcal F}
+\frac{\alpha^2}{2} \lambdamax(H)^2\,\|c(x)\|_{\mathcal F}
+\frac12 L_c\,\alpha^2\|d\|_{\mathcal E}^2.
\end{aligned}
\end{equation}
Combining~\eqref{eq:Ctrue}, \eqref{eq:bound_f} and \cref{eq:Dphi_smooth}
yields~\eqref{eq:quadratic_upper_bound_phi}.
\end{proof}

The following result shows that the line search terminates in a finite number of 
steps, and gives an explicit lower bound on the step size. 

\begin{proposition}[Lower bound on the step size]
   % Consider the sequence of iterates $(x_k)_{k\in\N}$  with descent directions $(d_k)_{k\in\N}$,
   % penalty parameters $(\mu_k)_{k\in\N}$
   % and
    The sequence of step sizes $(\alpha_k)_{k\in\N}$ generated by
    \cref{algo:LSLanding} are bounded from below:
%    At $x_k \in \saferegion$ the Armijo line search satisfies
% \begin{equation}\label{eq:armijo_condition}
%\phi_{\mu_k}(x_k) - \phi_{\mu_k}(\xkp) \geq - \eta \alpha_k \DD
%\phi_{\mu_k}(x_k)d_k,
%\end{equation}
    \begin{equation}
    \label{eqn:40tfv}
\alpha_k\>\underline{\alpha}:= 
\underline{g}\dfrac{2\beta (1-\eta)}{(L_f+\bar \mu L_c)} \min\!\left\{\,1,\;
\frac{\rho}{c_1}\right\} >0 \qquad \text{ for all } k\in\N,
    \end{equation}
    where $c_1$ is the constant 
\begin{equation}\label{eq:c1}
	c_1 := \left(\bar g \norm{\DD c^{\dagger,\mathcal{E}}}_{\infty}^2 R
    + \dfrac{\underline{g}\bar\mu }{(L_f+ L_c)} \right)\lambdamax(H)^2.
\end{equation}
\end{proposition}
\begin{proof}
Owing to~\eqref{eq:quadratic_upper_bound_phi}, the Armijo condition
\begin{equation}
\phi_\mu(x+\alpha d(x))\le \phi_\mu(x)+\eta\,\alpha\,D\phi_\mu(x)d(x)
\end{equation}
is satisfied whenever 
\begin{equation}
\alpha D\phi_\mu(x)d(x)
+\frac{\alpha^2}{2}\big(L_f+\mu L_c\big)\|d\|_{\mathcal E}^2
+\frac{\alpha^2}{2}\mu\lambdamax(H)^2\,\|c(x)\|_{\mathcal F} 
\leq \eta\,\alpha\,D\phi_\mu(x)d(x),
\end{equation}
or, equivalently, 
\begin{equation}
	\alpha \leq \alpha_{\min} \text{ with }\alpha_{\min}:=\dfrac{2(1-\eta)|\DD \phi_\mu(x)d(x)|}{(L_f+\mu
        L_c)\norm{d(x)}_{\mathcal{E}}^2
	 + \mu \lambdamax(H)^2 \norm{c(x)}_\cset}.
\end{equation}
This means that, at iteration $k$, the Armijo line search terminates with 
$\alpha_k>0$ satisfying either $\alpha_k=1$ or $\alpha_k/\beta\>\alpha_{\min}$.
The sufficient decrease condition \cref{eq:Dphi} implies then
\begin{equation}
    \label{eqn:i054k}
\alpha_k  \geq \frac{	2\beta(1-\eta)}{L_f+\mu_k L_c}\times
	\dfrac{\|\grad_{\Mxk}^{g} f(x_k)\|_{g}^2
+ \rho \mu_k \|c(x_k)\|_{\mathcal{F}}}
{\norm{d_k}_{\mathcal{E}}^2 + \mu_k \lambdamax(H)^2
\norm{c(x_k)}_\cset/(L_f+\mu_k L_c)}.
\end{equation}
By using \cref{assu:mod-sqp} and the $g$-orthogonality of the tangent and normal
step, the descent direction $d_k$ can be estimated as 
\[
    \begin{aligned}
    \norm{d_k}^{2}_{\mathcal{E}} & \< \frac{1}{\underline{g}} g(d_k,d_k)
      =\frac{1}{\underline{g}}(g(d_T(x_k),d_T(x_k))+g(d_N(x_k),d_N(x_k)))
    \\
     & \< \frac{1}{\underline{g}}
     \norm{\grad^{g}_{\Mxk}f(x_k)}_{g}^{2} +\frac{1}{\underline{
     g}}\norm{\DD
    c(x_k)^{\dagger,\mathcal{E}}H(x_k) c(x_k)}^{2}_{g}) \\
    &\< \frac{1}{\underline{g}}\left(\norm{\grad^{g}_{\Mxk}f(x_k)}_{g}^{2} 
    +\bar g\norm{\DD c^{\dagger, \mathcal{E}}}_{\infty}^{2}
    \lambda_{\max}(H)^{2}\norm{c(x_k)}_{\mathcal{F}}^{2}\right),
    \end{aligned}
\] 
where we use the property that $d_N(x_k)$ is the minimum $g$-norm vector $\xi\in
\mathcal{E}$
satisfying $\DD c(x_k)\xi=-H(x_k)c(x_k)$ and that $-\DD
c^{\dagger,\mathcal{E}}(x_k)H(x_k)c(x_k)$ is another such vector. 
%\begin{equation}
%\geq \dfrac{2\beta(1-\eta)}{(L_f+\mu L_c)}	\dfrac{\|\grad_{\Mx}^{g} f(x)\|_{g}^2
%+ \rho \mu_k \|c(x)\|_{\mathcal{F}}}
%{\norm{\grad^g_{\Mxk} f(x_k)}_{\mathcal{E}}^2 + \norm{\DD
%    c(x_k)^{\dagger,g}H(x_k)c(x_k)}_{\mathcal{E}}^2 
%+ \dfrac{\mu \lambdamax(H)^2}{(L_f+\mu L_c)} \norm{c(x_k)}_\cset}
%\end{equation}
%Recalling that $(\mu_k)_{k\in\N}$ is an increasing sequence bounded by
%$\bar{\mu}$ (\cref{prop:moc7l}), 
%we have thus 
%\[
%\alpha_k\> \frac{2\beta (1-\eta)}{L_f+\bar\mu L_c}\times 
%	\dfrac{\|\grad_{\Mxk}^{g} f(x_k)\|_{g}^2
%+ \rho \mu_k \|c(x_k)\|_{\mathcal{F}}}
%{\norm{d_k}_{\mathcal{E}}^2 + \mu_k \lambdamax(H)^2
%\norm{c(x_k)}_\cset/(L_f+\mu_k L_c)}.
%\] 
Therefore, 
for $\|c(x)\|_{\mathcal F}\le R$, the denominator of \cref{eqn:i054k} 
is upper bounded as follows
\begin{equation}
	\begin{aligned}
        \norm{d_k}_{\mathcal{E}}^2 & + \mu_k \lambdamax(H)^2
        \norm{c(x_k)}_\cset/(L_f+\mu_k L_c) \\
        &\< 
		  \frac{1}{\underline{g}}
          \norm{\grad^g_{\Mxk} f(x_k)}_{g}^2 +
		  \frac{\bar{g}}{\underline{g}}
          \norm{\DD
          c^{\dagger,\mathcal{E}}}_{\infty}^2
		 		 \lambdamax(H)^2 \norm{c(x_k)}_{\mathcal{F}}^2 
		 + \dfrac{\mu_k \lambdamax(H)^2}{L_f+\mu_k L_c} \norm{c(x_k)}_\cset\\
         &\leq \frac{1}{\underline{g}}\norm{\grad^g_{\Mxk}
             f(x_k)}_{g}^2 
             + \left(\frac{\bar g}{\underline{g}} \norm{\DD
                 c^{\dagger,\mathcal{E}}}_{\infty}^2  R
		 + \dfrac{\mu_k }{L_f+\mu_k L_c} \right) \lambdamax(H)^2\norm{c(x_k)}_\cset\\
		 		 &\leq \frac{1}{\underline{g}} \left(
                     \norm{\grad^g_{\Mxk} f(x_k)}_{g}^2 + c_1
                 \norm{c(x_k)}_\cset\right),
	\end{aligned}
\end{equation}
where $c_1$ is defined in~\eqref{eq:c1} and
 the last inequality follows from $1\leq \mu_k \leq \bar \mu$ for all $k$.
Thus we have 
\begin{equation}
        \alpha_k \geq \underline{g} \dfrac{2\beta(1-\eta)}{(L_f+\bar \mu L_c)} 
        F(\|\grad_{\Mxk}^{g} f(x_k)\|_{g}^2,\norm{c(x_k)}_{\mathcal{F}})
        \text{ with }F(t_1,t_2):=\frac{t_1+\rho t_2}{t_1+c_1t_2}.
\end{equation}
Finally, for $t_1,t_2\>0$ with $t_1+c_1 t_2>0$, it holds $F(t_1,t_2)\> \min(1,\rho/c_1)$. 
This implies the bound.
\end{proof}

We are now in a position to prove our main convergence result for the Riemannian 
landing method. The convergence rate is consistent with an existing result 
for SQP~\citep[Thm. 1]{curtis2024worst}.
\begin{theorem}[Global convergence of the landing iterates]
Under assumptions \aref{assu:boundedMandC}
to
\aref{assu:g387q},
the iterates $(x_k)_{k\in\N}$ generated by 
\cref{algo:LSLanding} satisfy
\begin{equation}
\label{eqn:pyjru}\norm{c(x_k)}_\cset \rightarrow 0 \text{ and }
    \norm{\grad^{g}_{\Mx}f(x_k)}_{g}\rightarrow 0 \text{ as }k\rightarrow
+\infty.
\end{equation}
    Additionally, the rate of convergences of these sequences is given by
   \begin{equation}\label{eq:complexity_grad}
	\min_{k=0,\dots K-1} \|\grad_{\Mx}^{g}
    f(x_k)\|_{g} \leq \frac{C_1}{\sqrt{{K}} }
\end{equation}
and 
\begin{equation}\label{eq:complexity_feasibility}
\min_{k=0,\dots K-1} \|c(x_k)\|_{\mathcal{F}} 
	\leq \frac{C_2}{K}.
\end{equation}
for constants $C_1$ and $C_2$ given in \eqref{eq:C}.
\end{theorem}
\begin{proof}
    For $K\>0$, the Armijo condition reads
\[
\phi_{\mu_k}(x_{k+1})\< \phi_{\mu_k}(x_k)+\eta \alpha_k \DD \phi_{\mu_k}(x_k)
d_k,
\] 
which, thanks to \cref{eq:Dphi}, implies
\begin{equation}\label{eq:armijo-proof}
\phi_{\mu_k}(x_{k+1})
\leq
\phi_{\mu_k}(x_k) -
\eta\alpha_k
\Big(
\|\grad_{\mathcal M_{x_k}}^g f(x_k)\|_g^2
+
\rho \mu_k \|c(x_k)\|_{\calF}
\Big).
\end{equation}
Next, using
\[
\phi_{\mu_{k+1}}(x_{k+1})
=
\phi_{\mu_k}(x_{k+1})
+
(\mu_{k+1}-\mu_k)\|c(x_{k+1})\|_\calF,
\]
we infer from~\eqref{eq:armijo-proof} that
\begin{equation}\label{eq:phi-mukplus1}
\phi_{\mu_{k+1}}(x_{k+1})
\le
\phi_{\mu_k}(x_k)
-\eta\alpha_k
\Big(
\|\grad_{\mathcal M_{x_k}}^g f(x_k)\|_g^2
+
\rho \mu_k \|c(x_k)\|_\calF
\Big)
+
(\mu_{k+1}-\mu_k)\|c(x_{k+1})\|_\calF,
\end{equation}
that is, using $\alpha_k \geq \underline{\alpha}$~\cref{eqn:40tfv} and $\mu_k \geq \mu_0$,
\begin{equation}\label{eq:alkj}
\phi_{\mu_k}(x_k) - \phi_{\mu_{k+1}}(x_{k+1}) +
(\mu_{k+1}-\mu_k)\|c(x_{k+1})\|_\calF
\geq \eta   \underline{\alpha}
\Big(
\|\grad_{\mathcal M_{x_k}}^g f(x_k)\|_g^2
+
\rho \mu_0 \|c(x_k)\|_\calF
\Big),
\end{equation}
Let $K\in \mathbb{N}$, using~\aref{assu:boundedMandC} 
and $\mu_K \leq \bar{\mu}$~\eqref{eq:mu_bar}, we have
\begin{equation}
\begin{aligned}
\sum_{k=0}^{K-1}  \phi_{\mu_k}(x_k) &- \phi_{\mu_{k+1}}(x_{k+1}) +
(\mu_{k+1}-\mu_k)\|c(x_{k+1})\| \\
&\leq \phi_{\mu_0}(x_0) - \phi_{\mu_{K}}(x_{K}) +R (\mu_K - \mu_0) \\
%&\leq f(x_0) + \mu_0 \norm{c(x_0)}_\calF - \flow - \mu_K \norm{c(x_K)}_\calF + R (\mu_K - \mu_0)\\
&\leq f(x_0) - \flow + \mu_0 \norm{c(x_0)}_\calF + R(\bar{\mu}- \mu_0).
    \end{aligned}
\end{equation}
By summing~\eqref{eq:alkj}, we obtain
\begin{equation}
    \begin{aligned}
        f(x_0) - \flow &+ \mu_0 \norm{c(x_0)}_\calF + R(\bar{\mu}- \mu_0) \geq \sum_{k=0}^{K-1} \eta   \underline{\alpha}
\Big(
\|\grad_{\mathcal M_{x_k}}^g f(x_k)\|_g^2
+
\rho \mu_0 \|c(x_k)\|
\Big).
    \end{aligned}
\end{equation}
This implies \cref{eqn:pyjru,eq:complexity_grad,eq:complexity_feasibility}, where 
\begin{equation}\label{eq:C}
C_1 = \left(f(x_0) - \flow + \mu_0 \norm{c(x_0)}_\calF + R(\bar{\mu}- \mu_0)\right)/(\eta \underline{\alpha}),
\end{equation}
and $C_2 = C_1/(\rho \mu_0)$.
\end{proof}

%By Theorem~\ref{thm:mu-bounded-PG}, $\mu_k$ is uniformly bounded and eventually constant; denote this stabilized value by $\bar\mu$.
%We now analyze the backtracking Armijo line-search on $\phi_{\bar\mu}$ with parameters $\eta\in(0,1)$, $\beta\in(0,1)$.
%
%Let $L_\phi$ denote a Lipschitz constant of $\nabla \phi_{\bar\mu}$ on the level set visited by the iterates with $c\neq 0$ (exists by (A1) and boundedness of the level set; see below). Standard backtracking then accepts some $\alpha_k=\beta^{j_k}$ satisfying the Armijo condition
%\begin{equation}\label{eq:armijo-accepted}
%\phi_{\bar\mu}(x_{k+1}) \;\le\; \phi_{\bar\mu}(x_k) + \eta\,\alpha_k\,\DD\phi_{\bar\mu}(x_k)[d_k].
%\end{equation}
%Moreover, descent-direction smoothness yields the well-known lower bound
%\begin{equation}\label{eq:alpha-lower}
%\alpha_k \;\ge\; \min\Bigl\{\,1,\ \frac{2(1-\eta)\,|\DD\phi_{\bar\mu}(x_k)[d_k]|}{L_\phi\,\|d_k\|^2}\Bigr\}.
%\end{equation}
%Since $\|d_k\|\le \|u_k\|+\|v_k\| \le M\,\|\grad f(x_k)\| + C\,\|c(x_k)\|$ and both $\|\grad f(x_k)\|$ and $\|c(x_k)\|$ remain bounded on the level set, $\|d_k\|$ is uniformly bounded.
%
%We also use that the iterates remain in a compact level set: by \eqref{eq:armijo-accepted} and \eqref{eq:Dphi-final}, $\{\phi_{\bar\mu}(x_k)\}$ is monotonically decreasing and bounded below, hence $\{x_k\}$ stays in the initial level set of $\phi_{\bar\mu}$, which is compact by (A1).

%\begin{corollary}\label{lem:tangent-stationarity}
%Let $x_\star$ be any accumulation point. Then $\grad f(x_\star)=0$ and
%$c(x_{\star})=0$.
%\end{corollary}
Due to the boundedness of the sequence $(x_k)_{k\>0}$, there must be an
accumulation point, which is necessarily a feasible KKT point according to
\cref{eqn:pyjru}.
%\begin{proof}
%By Lemma~\ref{lem:feasible}, $c(x_k)\to 0$ along the full sequence, so $v_k=-\DD c(x_k)^\top c(x_k)\to 0$. Suppose, by contradiction, that along a subsequence $K$ converging to $x_\star$ we have $\|\grad f(x_k)\|\ge \varepsilon>0$ for all $k\in K$.
%Using \eqref{eq:Dphi-final} and continuity, there exists $k_0$ such that for all $k\in K$, $k\ge k_0$,
%\[
%\DD\phi_{\bar\mu}(x_k)[d_k] \;\le\; -\frac{m}{2}\,\varepsilon^2.
%\]
%As before, the smoothness lower bound \eqref{eq:alpha-lower} implies a uniform $\underline{\alpha}>0$ for $k\in K$, large. Then \eqref{eq:armijo-accepted} gives a fixed decrease
%\[
%\phi_{\bar\mu}(x_{k+1}) \;\le\; \phi_{\bar\mu}(x_k) - \eta\,\underline{\alpha}\,\frac{m}{2}\,\varepsilon^2
%\quad\text{for all }k\in K,\ k\ge k_0,
%\]
%contradicting boundedness below. Hence $\|\grad f(x_k)\|\to 0$ along $K$, and by continuity $\grad f(x_\star)=0$.
%\end{proof}

\section{Matrix optimization with orthogonality constraints}
\label{sec:ym3u0}
In this section, we focus on the geometric design of tangent and normal terms for matrix optimization with
orthogonality constraints---a natural application for the landing algorithm. Consider the problem 
\begin{align}
\label{eq:Pstiefel}
        \underset{X\in \Rnp}{\minimize} & ~f(X) \textrm{ subject to } X\T X =\Ip, 
    \end{align}
    where $p\leq n$. The feasible set---called the Stiefel manifold---consists of rectangular matrices with orthonormal columns
\[
\mathcal{M}\equiv\St(n,p):=\{ X\in\R^{n\times p }\,|\, X\T X= \Ip \}.
\]   
A function that defines the constraints is
\begin{align}\label{eq:h}
\c\colon \Rnp \to \Sym(p)\colon \c(X) = \frac{1}{2}(X\transpose X - \Ip).
\end{align}
Thus, we have an instance of~\eqref{eq:P} with $\calE= \Rnp$ and $\cset = \Sym(p)$,
the set of symmetric matrices of size $p$. The spaces $\mathcal{E}$ and
$\mathcal{F}$ are equiped with the standard Frobenius inner product $\langle
\cdot,\cdot  \rangle_{\mathcal{E}}$ and $\langle \cdot,\cdot  \rangle_{\mathcal{F}}$.
In all that follows, we drop the subscript notation $\cdot_{\mathcal{E}}$ and
$\cdot_{\mathcal{F}}$ when writing matrix Frobenius inner products or norms, in particular,
for two matrices of same size $X$ and $Y$, we have $\langle X,Y  \rangle=\mathrm{trace}(X\T
Y)$ and $\norm{X}^{2}=\langle X,X  \rangle$.
Let $\Rnpstar$ denote the set of full-rank matrices of size $n\times p$. We write $\Sym(p)
= \{S\in \Rpp|S\T = S\}$ and $\Skew(p) = \{\Omega \in \Rpp | \Omega\T = - \Omega\}$,
and $\sym(A) = (A + A\T)/2$ and $\skew(A) = (A - A\T)/2$. 

\medskip 
\noindent
The squared
infeasibility for the problem \cref{eq:Pstiefel}
reads
\[\psi(X) = \frac{1}{2}\norm{c(X)}^{2}=\frac{1}{4}\norm{X\T X - \I_p}^2,\]
where $\I_p$ is the identity matrix of $\R^{p\times p}$.
The set $\mathcal{D}$ of \cref{eq:D} is the set of full-rank matrices
$\R^{n\times p}_*$ and the layer manifold $\mathcal{M}_X$ is 
	\[
	\M_X\equiv\StX:= \left\{ Y\in \Rnp : Y\T Y = X\T X\right\}.
	\]

In this section, we demonstrate 
how the family  of normal spaces $X\to \mathrm{N}_X\StX$
is crucial to 
 lead to 
efficient computations of the tangent and normal terms of the landing algorithm
\cref{eq:landing}. To illustrate this claim, we first consider the family of normal spaces
orthogonal to the tangent spaces with respect to  the Euclidean Frobenius inner
product. Although this choice is natural, 
computing  the corresponding orthogonal projections is expensive
because it requires solving Sylvester equations. We then take the reverse
approach: 
we 
first introduce an alternative 
family of
normal spaces with associated projection operators that can be computed explicitly, 
then we design tangent and normal metrics that  
lead to explicit---and thus more exploitable---formulas for the
tangent and normal steps for matrix optimization 
with orthogonality constraints. 

\medskip 

The starting point is to recall the characterization of the tangent space to 
  $\StX$ at $X\in \Rnpstar$. We emphasize that this set does not depend on the metric.
\begin{proposition}[\cite{goyens_geometric_2026}]
    The tangent space of $\StX$ at $X$ is the set    
        \begin{align} \label{eq:tgtspace1}
        \Trm_X \StX  & = \{ \xi\in\R^{n\times p}\,|\, \xi\T X+X\T \xi = 0  \} \\   
        & = \{ X(X\T X)^{-1}\Omega + \Delta\,|\, \Omega\in \Skew(p),\, \Delta
        \in \R^{n\times p} \text{ with } \Delta\T X=0\}  \label{eq:tgtspace2}\\  
    &= \{W X\,|\, W\in\Skew(n)\}, \label{eq:3form_tangent}
        \end{align}
        with dimension $np - p(p+1)/2$.
\end{proposition}
\subsection{Landing algorithm for the Euclidean metric}
\label{subsec:r3f8x}

We first outline the derivation of the tangent and normal steps
of the Riemannian landing 
when $g\equiv
g^{\mathcal{E}}$ is the Euclidean or   Frobenius inner product metric: 
\[
g^{\mathcal{E}}(\xi,\zeta)=\langle \xi,\zeta  \rangle \qquad \text{ for all } \xi,\zeta \in \R^{n\times p}.
\] 
%A choice of metric in the ambient space induces a projection operator, equivalently,
%it induces a notion of normal space (the set of vectors orthogonal to the tangent
%space). 
We choose formulation~\cref{eq:pseudoinverse_landing} of the Riemannian landing algorithm  based on the
pseudoinverse formula for the normal step. 
The tangent and normal vector fields
with respect to the Euclidean metric, are given by 
    \begin{align}
        \label{eqn:115ya}
    d_T(X) & =-\Proj_{X,\mathcal{E}}[\nabla_{\mathcal{E}}f(X)], \\
    \label{eqn:ccq3v}
    d_N(X) &= -\DD c(X)^{\dagger,\mathcal{E}}[H(X)[c(X)]]
    \end{align}
where brackets denote the action of linear operators on matrix spaces.  

\medskip 

In what follows, we compute these two directions explicitly.
The orthogonal projector
$\Proj_{X,\mathcal{E}}$ can be obtained through the characterization
\cref{eq:projection}, as is commonly done in the literature on Riemannian
optimization \cite{absil2008optimization}. However, since the formula for $d_N(X)$ 
also requires the computation of  the pseudoinverse $\DD c(X)^{\dagger, \mathcal{E}}$, 
we instead propose to compute it using the explicit relation
$\Proj_{X,\mathcal{E}}=\Id_{\mathcal{E}}-\DD c(X)^{\dagger,\mathcal{E}}\DD c(X).$ As
it becomes clear below, the main difficulty in computing this projection lies in inverting 
 the symmetric operator
$(\DD c(X)\DD c(X)^{*,\mathcal{E}})^{-1}$ that appears in the  pseudoinverse formula 
$\DD c(X)^{\dagger,\mathcal{E}}=\DD c(X)^{*,\mathcal{E}}(\DD c(X)\DD
c(X)^{*,\mathcal{E}})^{-1}$. Let us start by evaluating $\DD c(X)$ and $\DD
c(X)^{*,\mathcal{E}}$.
\begin{proposition}
    \label{prop:f9x2y}
    The operator $\DD c(X)\,:\, \R^{n\times p}\to \Sym(p)$ and 
    its adjoint $\DD c(X)^{*,\mathcal{E}}\,:\, \Sym(p)\to \R^{n\times p}$ read:
    \begin{enumerate}[(i)]
        \item $\DD c(X)[\xi]=\frac{1}{2}(X\T \xi+\xi\T X)=\sym(X\T \xi)$ for all
            $\xi\in \R^{n\times p}$,
        \item $\DD c(X)^{*,\mathcal{E}}[S] =XS$ for all $S\in \Sym(p)$.
    \end{enumerate}
\end{proposition}
\begin{proof}
    The point (i) follows directly from \cref{eq:h}. For (ii), we write, for any
    $S\in \Sym(p)$ and $\xi\in \R^{n\times p}$,
    \[
        \langle \xi,\DD c(X)^{*,\mathcal{E}}[S]  \rangle  = \langle \DD
        c(X)[\xi],S  \rangle 
   =\langle \sym(X\T \xi),S  \rangle=\langle X\T \xi,S
   \rangle
   =\langle \xi,XS  \rangle.
    \] 
\end{proof}
\noindent We infer 
the following characterization of the normal space 
$\normalE_X \StX=\mathrm{Range}(\DD c(X)^{*,\mathcal{E}})$ and of 
the Euclidean pseudoinverse $\DD c(X)^{\dagger,\mathcal{E}}$.
\begin{corollary}
    The normal space of $\St_{X\T X}$ at $X\in\Rnpstar$ with respect to the Euclidean
    metric $\gE$ is    
    \begin{equation}
    \label{eqn:wmtku}
        \normalE_{X} \St_{X\transpose X}	=\{ XS\,|\, S\in \Sym(p)\}.
    \end{equation}
\end{corollary}
\begin{proposition}
    \label{prop:5dcoz}
    The pseudoinverse $\DD c(X)^{\dagger,\mathcal{E}}\,:\, \Sym(p)\to \R^{n\times p}$ is
    the operator defined by
    \[
        \DD c(X)^{\dagger,\mathcal{E}}[T]=XS, \quad \text{ for all } T\in \Sym(p),
    \] 
    where $S$ is the unique solution in $\Sym(p)$ to the Sylvester equation 
    \begin{equation}
    \label{eqn:e6bhg}
  \frac{1}{2} (X\T X S +S X\T X)=T.
    \end{equation}
\end{proposition}
\begin{proof}
    From \cref{prop:f9x2y}, it is readily seen that 
    \[
        \DD c(X)\DD c(X)^{*,\mathcal{E}}[S]=\frac{1}{2}(X\T X S+SX\T X),
    \] 
    from which these results follow.
\end{proof}
Consequently, the computation of the tangent and the normal terms of the landing
algorithm require, in full generality, solving Sylvester equations.
\begin{corollary}
    \label{cor:wfjg8}
    \begin{enumerate}[(i)]
        \item The tangent term \cref{eqn:115ya} of the landing algorithm
                  with respect to the Euclidean metric
    given by 
    \begin{equation}
    \label{eqn:w46u0}
       d_T(X)=- \Proj_{X,\mathcal{E}}[\nabla_{\mathcal{E}}(f(X))] =
      -\nabla_{\mathcal{E}}f(X)+XS, \qquad \text{ for all } Z\in\R^{n\times p},
    \end{equation}
    where $S$ is the unique solution in $\Sym(p)$ to the equation
    \cref{eqn:e6bhg} with $T=\sym(X\T \nabla_{\mathcal{E}}f(X))$.
\item 
    The normal vector field  \cref{eqn:ccq3v} is given by 
    $d_N(X)=- X S$, where $S$ is the solution to the Sylvester 
    equation \cref{eqn:e6bhg} with
    $T=\frac{1}{2}H(X)[X\T X-\I_p]$.
\end{enumerate}
\end{corollary}
\begin{proof}
    These results are obtained by using \cref{prop:5dcoz} and formula
    \cref{eq:Projg} for the projection operator $\Proj_{X,\mathcal{E}}$.
\end{proof}
In the general case,  we mention that the solution to \cref{eqn:e6bhg} 
could be calculated
from the singular value decomposition of~$X$. The following result is classical
 and may be
found e.g. in \citep{li2018spectral}. 
\begin{lemma}
    \label{lem:27ja6}
    Let $T\in
    \R^{p\times p}$ be an arbitrary matrix. 
    Let
    $X=\sum_{i=1}^{p}\sigma_i u_i v_i\T$ be the singular value decomposition of
    a full rank matrix 
    $X\in\R^{n\times p}$ 
    with singular values $\sigma_1\>\sigma_2\>\dots\>\sigma_p>0$, 
    left and right  singular vectors $(u_i)_{1\<i\<p}$ and $(v_i)_{1\<i\<p}$. 
   The unique solution $S\in\R^{p\times p}$ to the
   Sylvester equation  \cref{eqn:e6bhg}
    is given by     
\[
S=
\frac{1}{2}\sum_{1\<i\<j\<p}\frac{v_i\T Tv_j}{\sigma_i^{2}+\sigma_j^{2}}(v_iv_j\T+v_jv_i\T).
\] 
\end{lemma}
When  
    $T=Q(X\T X)$ with $Q$ an analytic function, the solution to the Sylvester
    equation \cref{eqn:e6bhg} 
    is explicitly given by
    \[
        S=Q(X\T X)(X\T X)^{-1}.
    \]
This observation leads to explicit formulas for the normal field $d_N(X)$ 
for two particular values of the operator $H(X)\,:\,\Sym(p)\to\Sym(p)$.
\begin{proposition}
    \label{prop:n356m}
\begin{enumerate}[(i)]
    \item \label{eqn:53mb8} If $H(X)=\Id_{\mathcal{F}}$, then $
        d_N(X)=-\frac{1}{2}X(\I_p-(X\T X)^{-1})$.
    \item \label{eqn:oqvzn}
        if $H(X)=\DD c(X)\DD c(X)^{*,\mathcal{E}}$, then $
        d_N(X)=-\nabla_{\mathcal{E}}\psi(X)=-X(X\T X - \I_p)$.
\end{enumerate}
\end{proposition}
\begin{proof}
    For $H(X)=\Id_{\mathcal{F}}$,
    $d_N(X)=-XS$ with $S$ being the solution to \cref{eqn:e6bhg} 
    with $T=Q(X\T X)$ with $Q(x)=\frac{1}{2}(x-1)$. 
    The solution is $S=(\I_p-(X\T X)^{-1})/2$, which yields
    (i). 	
     For $H(X)=\DD c(X)\DD c(X)^{*,\mathcal{E}}$, (ii) can
    be found directly by using the identity $d_N(X)=-\nabla_{\mathcal{E}} \psi(X)$ with
    $\psi(X)=\frac{1}{4}\norm{X\T X-\I_p}_{\mathcal{F}}^{2}$. It is instructive,
    however, to retrieve
    this result from the result of \cref{cor:wfjg8}. We have in this case that $S$ is
    the solution to \cref{eqn:e6bhg} the right-hand side $T$ given by
    \[
        \begin{aligned}
        T & =\frac{1}{2}\DD c(X)\DD c(X)^{*,\mathcal{E}}[X\T X-\I_p]\\
        &= \frac{1}{2}[X\T (X X\T X -X)+(X\T X X\T -X\T)X] \\
        &= (X\T X)^{2}-X\T X \\ 
        & =Q(X\T X) \text{ with }Q(x)=x^{2}-x.
        \end{aligned}
    \] 
    Hence, the solution to the Sylvester equation \cref{eqn:e6bhg} is 
    $S=(X\T X-\I_p)$, which yields (ii).
\end{proof}
To summarize, the execution of the landing
algorithm \cref{eq:pseudoinverse_landing} for minimizing a function  
on the Stiefel manifold using the  Euclidean metric $g^{\mathcal{E}}$ to define the
tangent and the normal terms $d_T(X)$ and $d_N(X)$ requires 
solving a Sylvester equation for computing the tangent term  at every iteration,
which is potentially costly for large matrices $X$. 

\subsection{Metric design based on an explicit choice of projection operators}
In this section, we adopt the reverse approach:  instead of first defining the metric
\( g \) and then deducing the associated tangent and normal terms \( d_T(X) \) and \(
d_N(X) \) via  \cref{eqn:uaqaq,eq:rlanding-gstar-normal}, we begin by choosing
\emph{first} an explicit family of projection operators \(\Proj_X\) onto the tangent
space  \(\mathrm{T}_{X}\StX\), or equivalently, a family of normal spaces \(X \mapsto
\mathrm{N}_X\StX\). This step is outlined in details in \cref{subsec:ckoiv}.
In the second step, we construct explicit operators  
\begin{equation}
\label{eqn:oasjx}
\widetilde{G_T}(X): \mathrm{T}_X\StX \to \mathrm{N}_X\StX^{\perp,\mathcal{E}},  
\qquad  
\widetilde{G_N}(X): \mathrm{N}_X\StX \to \mathrm{T}_X\StX,
\end{equation}
each admitting explicit inverses. This construction defines the metric \emph{a posteriori} through  
\cref{eqn:oitse}:  
\begin{equation}
\label{eqn:glmty}
    g(\xi,\zeta) = \langle \widetilde{G_T}(X)\Proj_X[\xi], \zeta \rangle  
    + \langle \widetilde{G_N}(X)\Proj_X^{\perp}[\xi], \zeta \rangle.
\end{equation}
Moreover, these explicit inverses provide closed-form expressions for the tangent and
normal terms of the landing algorithm via  \cref{eqn:gd5ra,eqn:cuylb}. 

We examine two possible choices for $\widetilde{G_T}$ and $\widetilde{G_N}$. 
In \cref{subsub:gi2g3}, we consider a `canonical' choice for these mappings,
 which leads to the definition of a new metric \( g \) on
\(\R^{n\times p}_*\) and new formulas for the landing algorithm for the optimization
problem  \cref{eq:Pstiefel}. In \cref{subsub:eawup}, we consider instead the family of
\(\beta\)-metrics  \( g^{\beta} \) introduced in \citep{goyens_geometric_2026}, which
includes for $\beta=\frac{1}{2}$ 
the pulled-back canonical metric for the layered Stiefel manifold studied in
\citep{gao_optimization_2022}. 
We show that this family of metric is associated to the proposed family of normal
spaces $X\mapsto \mathrm{N}_X\StX$ 
considered in the first step, with a specific
choice of \(\widetilde{G_T}\) and  \(\widetilde{G_N}\).
These operators turn to admit explicit
inverses, thereby highlighting why closed-form formulas for \(d_T(X)\) and \(d_N(X)\)
are also available in this case.

\subsubsection{Choice of oblique projection operators and associated normal spaces}
\label{subsec:ckoiv}

Recalling that any tangent vector $\xi\in \mathrm{T}_{X}\StX$ can be written as 
\[
\xi = X(X\T X)^{-1}\Omega + \Delta \text{ with }\Omega\in \Skew(p) \text{ and } \Delta\T X =0,
\] 
any matrix $Z\in\R^{n,p}$ can be naturally decomposed as 
\[
    \begin{aligned}
    Z & =X(X\T X)^{-1}X\T Z + (\I_n-X (X\T X)^{-1}X\T)Z \\
    &= X(X\T X)^{-1} \skew(X\T Z) + (\I_n-X(X\T X)^{-1}X\T
    )Z + X(X\T X)^{-1}\sym(X\T Z) \\
    &=  \Proj_X[Z]+\Proj_X^{\perp}[Z],
    \end{aligned}
\] 
where $\Proj_X$ and $\Proj_X^{\perp}$ are the operators defined by 
\begin{align}
    \label{eqn:2znwx}
    \Proj_X[Z]& :=X(X\T X)^{-1} \skew(X\T Z)+(\I_n-X(X\T X)^{-1}X\T)Z, \\
    \label{eqn:0jkp9}
    \Proj_X^{\perp}[Z]& := \Id_{\mathcal{E}}-\Proj_X(Z)=X(X\T X)^{-1} \sym(X\T Z).
\end{align}
It is straightforward to verify that $\Proj_X$ and $\Proj_X^{\perp}$ are linear
(oblique) projectors.
In view of \cref{eqn:0jkp9},
it is clear that this choice of projectors corresponds to attaching to
$\mathrm{T}_X\StX$ the normal space 
\begin{equation}
\label{eqn:azgxt}
    \mathrm{N}_X\StX:=\mathrm{Range}(\Proj_X^{\perp})=\{ X(X\T X)^{-1}S\,|\, S\in\Sym(p) \},
\end{equation}
to be compared with \cref{eqn:wmtku}. Given any symmetric operators $G_T(X)\,:\,
\mathcal{E}\to \mathcal{E}$ and $G_N(X)\,:\, \mathcal{E}\to \mathcal{E}$, 
positive definite on respectively $\mathrm{T}_{X}\StX$ and 
$\mathrm{N}_X \StX$, $\Proj_X$ and $\Proj_X^{\perp}$ are orthogonal
projectors for the metric \cref{eqn:glmty} with 
\[
    \widetilde{G_T}(X)=\Proj_X^{*,\mathcal{E}}G_T(X)\Proj_X, \quad 
    \widetilde{G_N}(X)=(\Proj_X^{\perp})^{*,\mathcal{E}}G_N(X)\Proj_X^{\perp}.
\] 
The following
proposition provides expressions for  
the Euclidean adjoints of the tangent projection $\Proj_X$ and 
$\Proj_X^{\perp}$.
\begin{proposition}
    \label{prop:nhmwd}
    The adjoint operators $\Proj_X^{*,\mathcal{E}}$ and
    $(\Proj_X^{\perp})^{*,\mathcal{E}}$ of the projection operators
    $\Proj_X$ and $\Proj_X^{\perp}$---with respect to the Euclidean inner
    product---are given by 
    \begin{equation}
    \label{eqn:bgc18}
    \Proj_X^{*,\mathcal{E}}[Z]=X\skew((X\T X)^{-1}X\T Z)+(\I_n-X(X\T
    X)^{-1}X\T)Z,\qquad Z\in \R^{n\times p},
    \end{equation}
    \begin{equation}
        \label{eqn:8mag5}
        (\Proj_X^\perp)^{*,\mathcal{E}}[Z]=X\sym((X\T X)^{-1}X\T Z),\qquad
        Z\in\R^{n\times p}.
    \end{equation}
    In particular, the Euclidean orthogonal spaces of the tangent and normal spaces are
    \begin{equation}
    \label{eqn:4qfl4}
        (\mathrm{T}_X\StX)^{\perp,\mathcal{E}}=\mathrm{Range}((\Proj_X^{\perp})^{*,\mathcal{E}})=\{
        XS\,|\, S\in \Sym(p) \},
    \end{equation}
    \begin{equation}
    \label{eqn:230nz}
        (\mathrm{N}_X\StX)^{\perp,\mathcal{E}}=\mathrm{Range}(\Proj_X^{*,\mathcal{E}})=
        \{ X\Omega+\Delta\,|\, \Omega\in\Skew(p) \text{ and }\Delta\in\R^{n\times
        p}\text{ with }X\T \Delta= 0 \}.
    \end{equation}
\end{proposition}
\begin{proof}
    By using the self-adjointness of the operators $X(X\T X)^{-1}X\T$ and $\skew$, 
    we have for any
    $\xi,\zeta\in\R^{n\times p}$,
\[
    \begin{aligned}
    \langle \Proj_X[\xi],\zeta  \rangle
      & =\langle X(X\T X)^{-1}\skew(X\T \xi),\zeta  \rangle 
     +\langle (\I_n-X(X\T X)^{-1}X\T)\xi,\zeta  \rangle \\
     &=\langle \skew(X\T \xi),(X\T X)^{-1}X\T \zeta  \rangle 
     +\langle \xi, (\I_n-X(X\T X)^{-1}X\T)\zeta  \rangle \\
     &= \langle X\T \xi,\skew((X\T X)^{-1}X\T \zeta)  \rangle 
     +\langle \xi, (\I_n-X(X\T X)^{-1}X\T)\zeta  \rangle \\
     &=\langle \xi,X\skew((X\T X)^{-1}X\T \zeta)+(\I_n-X(X\T X)^{-1}X\T)\zeta
     \rangle.
    \end{aligned}
\]
Formula \cref{eqn:8mag5} can be obtained identically, or from
$(\Proj_X^\perp)^{*,\mathcal{E}}=\Id_{\mathcal{E}}-\Proj_X^{*,\mathcal{E}}$.
\end{proof}
We mention the following characterization of the adjoint of the operator $\DD c(X)$
with respect to the metric $g$ defined in \cref{eqn:glmty}.
\begin{proposition}
    \label{prop:lld7u}
    For any $S\in\Sym(p)$, 
    \[
        \DD c(X)^{*,g}[S] = \widetilde{G_N}(X)^{-1}[XS],
    \] 
    where $\widetilde{G_N}(X)^{-1}\,:\, (\mathrm{T}_X\StX)^{\perp,\mathcal{E}}\to
    \mathrm{N}_X\StX$ is the inverse of the operator $\widetilde{G_N}$
    defined in \cref{eqn:4ir6e}. 
\end{proposition}
\begin{proof}
    Due to \cref{eq:link_adjoints} and \cref{prop:f9x2y}:
    \begin{equation}
    \label{eqn:bj4q3}
        \DD c(X)^{*,g}[S]=G(X)^{-1}\DD c(X)^{*,\mathcal{E}}[S]=G(X)^{-1}[XS].
    \end{equation}
    Then, in view of \cref{prop:nhmwd}, 
    $\Proj_X^{*,\mathcal{E}}[XS]=0$ 
    and $(\Proj_X^{\perp})^{*,\mathcal{E}}[XS]=XS$.
    The result follows by applying \cref{eq:block-diag-metricinverse} to
    \cref{eqn:bj4q3}.
\end{proof}
From \cref{sec:euclidean_landing}, the  tangent and
normal steps of the landing algorithm \cref{eq:pseudoinverse_landing} read
\begin{align}
    \label{eqn:fkr7g}
    d_T(X)& =-\widetilde{G_T}(X)^{-1}\Proj_X^{*,\mathcal{E}}[\nabla_{\mathcal{E}}f(X)],\\
    \label{eqn:zgiyo}
    d_N(X)& =-\Proj_X^{\perp}\DD c(X)^{\dagger,\mathcal{E}}H(X)[c(X)].
\end{align}
The 
operators $\widetilde{G_T}(X)^{-1}$
and $H(X)$ can be freely chosen  by the user which gives some latitude 
for the computation of $d_T(X)$ and $d_N(X)$, and which have
rather clear physical interpretations.
We obtain the following formulas for
$d_T(X)$ and $d_N(X)$.
\begin{proposition}
    \label{prop:nw9q9}
    The tangent and normal terms (eq. \cref{eqn:fkr7g} and \cref{eqn:zgiyo}) 
    of the landing algorithm \cref{eq:pseudoinverse_landing}
    relatively to the metric \cref{eqn:glmty} are given by 
    \begin{align}
        \label{eqn:6ow5x}
        d_T(X)& =  -\widetilde{G_T}(X)^{-1}[X\skew((X\T X)^{-1}X\T
        \nabla_{\mathcal{E}}f(X))+(\I_n-X(X\T X)^{-1}X\T)\nabla_{\mathcal{E}}f(X)],\\
        \label{eqn:d9l5x}
        d_N(X) &= - X (X\T X)^{-1}\sym(X\T XS),
    \end{align}
    where $S$ is the solution to the Sylvester equation \cref{eqn:e6bhg} with
    $T=\frac{1}{2}H(X)[X\T X-\I_p]$.
    The expression of $d_N(X)$ becomes explicit in the
    following cases: 
    \begin{enumerate}[(i)]
        \item if $H(X)=\Id_{\mathcal{E}}$, it reads $d_N(X)=-\frac{1}{2}X(\I_p-(X\T X)^{-1});$ 
        \item if $H(X)=\DD c(X)\DD c(X)^{*,g}$, it reads $d_N(X)=-\frac{1}{2}\widetilde{G_N}(X)^{-1}[X(X\T X-\I_p)];$ 
        \item if $H(X)=\DD c(X)\DD c(X)^{*,\mathcal{E}}$, it reads $d_N(X)=-X(X\T
            X-\I_p)$. 
    \end{enumerate}
\end{proposition}
\begin{proof}
Formula \cref{eqn:6ow5x} follows immediately from \cref{eqn:fkr7g} and
\cref{eqn:bgc18}. Formula \cref{eqn:d9l5x} is obtained by combining \cref{eqn:zgiyo}
and the result of \cref{eqn:115ya}.  For the cases (i) and (iii), we observe
that due to \cref{eqn:zgiyo},
\[
    d_N(X)=\Proj_X^{\perp}[d_N(X)^{\mathcal{E}}]
\] 
where $d_N(X)^{\mathcal{E}}$ is the corresponding normal direction 
obtained in \cref{prop:n356m}. Since 
\[
X(\I_p-(X\T X)^{-1})=X(X\T X)^{-1}(X\T X-\I_p)\in \mathrm{N}_X\StX,\] 
\[
X(X\T X-\I_p)=X(X\T X)^{-1}((X\T X)^{2}-X\T X)\in \mathrm{N}_X\StX,
\] 
we have in both cases $d_N(X)=d_N(X)^{\mathcal{E}}$, which proves (i) and (iii). 
Finally, $H(X)=\DD c(X)\DD c(X)^{*,g}$ implies 
that $d_N(X)=-\DD c(X)^{*,g} c(X)$,  which yields (ii) after using 
\cref{prop:lld7u}.
\end{proof}
It remains to choose the operator $\widetilde{G_T}(X)^{-1}$ in \cref{eqn:6ow5x}, and
the operator $\widetilde{G_N}(X)^{-1}$ if one chooses (ii) for the normal direction. 
\subsubsection{Construction of the tangent and normal steps based on canonical 
tangent and normal metric representers}
\label{subsub:gi2g3}

There is a natural choice of operators $\widetilde G_T(X)\,:\,
\mathrm{T}_X\StX\to \mathrm{N}_X\StX^{\perp,\mathcal{E}}$ 
and $\widetilde{G_N}(X)\,:\, \mathrm{N}_X\StX\to
\mathrm{T}_X\StX^{\perp,\mathcal{E}}$ with explicit inverses
$\widetilde{G_T}(X)^{-1}$ and $\widetilde{G_N}(X)^{-1}$, leading to explicit
expressions for the tangent and normal steps \cref{eqn:6ow5x,eqn:d9l5x}.
Recalling \cref{eq:tgtspace2,eqn:230nz}, and \cref{eqn:azgxt,eqn:4qfl4},
these `canonical' operators read
\begin{equation}
\label{eqn:8yr7b}
    \begin{array}{rrcl}
        \widetilde{G_T}(X)\,:\,&  \mathrm{T}_X\StX& \longrightarrow&
        \mathrm{N}_X\StX^{\perp,\mathcal{E}}\\
        & X(X\T X)^{-1}\Omega + \Delta & \longmapsto & X\Omega +\Delta,
    \end{array}
\end{equation}
\begin{equation}
\label{eqn:53242}
    \begin{array}{rrcl}
        \widetilde{G_N}(X)\,:\,&  \mathrm{N}_X\StX& \longrightarrow&
        \mathrm{T}_X\StX^{\perp,\mathcal{E}}\\
        & X(X\T X)^{-1}S & \longmapsto & XS,
    \end{array}
\end{equation}
where $\Omega\in\Skew(p),$ $\Delta\in\R^{n\times p}$ with $\Delta^{T}X=0$ and
$S\in\Sym(p)$. These mappings have the following 
natural extensions to the whole $\mathcal{E}=\R^{n\times
p}$:
\begin{equation}
\label{eqn:96g0s}
G_T(X)[Z] := XX\T Z + (\I_n-X(X\T X)^{-1}X\T)Z,\qquad Z\in \R^{n\times p},
\end{equation}
\begin{equation}
\label{eqn:fyrfd}
G_N(X)[Z]:= XX\T Z, \qquad Z\in\R^{n\times p}.
\end{equation}
With these definitions, we have indeed $G_T(X)\Proj_X=\widetilde{G_T}(X)\Proj_X$ and 
$G_N(X)\Proj_X^{\perp}=\widetilde{G_N}(X)\Proj_X^{\perp}$.
\begin{proposition}
    The operators $G_T(X)$ and $G_N(X)$ of \cref{eqn:96g0s,eqn:fyrfd} are symmetric
    positive operators on $\mathcal{E}$, definite respectively on $\mathcal{E}$
    and $\mathrm{N}_X\StX$. Formula \cref{eqn:glmty} defines thus the associated
    metric 
    \begin{equation}
    \label{eqn:wen2o}
        \begin{aligned}
        g(\xi,\zeta) & =\langle \xi,\Proj_X^{*,\mathcal{E}}G_T(X)\Proj_X\zeta  \rangle+
    \langle \xi,(\Proj_X^{\perp})^{*,\mathcal{E}}G_N(X)\Proj_X^{\perp}\zeta
    \rangle\\
    &=\langle \xi,(XX\T + \I_n-X(X\T X)^{-1}X\T )\zeta  \rangle.
        \end{aligned}
    \end{equation}
\end{proposition}
\begin{proof}
   The symmetry, the positivity of $G_T$ and $G_N$ is clear, as well as the
   definiteness of $G_T$ on $\mathcal{E}$. 
   The definiteness of $G_N$ on $\mathrm{N}_X\StX$ follows from the inequality 
   \[
   \langle X(X\T X)^{-1}S,G_N(X)X(X\T X)^{-1}S  \rangle 
   =\langle X(X\T X)^{-1}S,XS  \rangle =\langle S,S  \rangle, \qquad \text{ for all }
   S\in\Sym(p).
   \]
   We then find that
   \[
       \begin{aligned}
       \Proj_X^{*,\mathcal{E}}G_T(X)\Proj_X\zeta
       &       =\Proj_X^{*,\mathcal{E}}[X\skew(X\T \zeta) + (\I_n-X(X\T X)^{-1}X\T
       )\zeta]\\
       &=X \skew(X\T \zeta) + (\I_n-X(X\T X)^{-1}X\T
       )\zeta,
       \end{aligned}
   \] 
   \[
       \begin{aligned}
           (\Proj_X^\perp)^{*,\mathcal{E}}G_N(X)\Proj_X^{\perp}\zeta
       &       =(\Proj_X^{\perp})^{*,\mathcal{E}}[X\sym(X\T \zeta)]
       =X \sym(X\T \zeta).
       \end{aligned}
   \] 
   With these formulas, we obtain thus 
   \[
       \begin{aligned}
       g(\xi,\zeta) & =\langle \xi,X\skew(X\T \zeta)+(\I_n-X(X\T X)^{-1}X\T)\zeta  \rangle
    +\langle \xi,X\sym(X\T \zeta)  \rangle\\
    &=\langle \xi,(XX\T +\I_n- X(X\T X)^{-1}X\T)\zeta  \rangle.
       \end{aligned}
   \] 
\end{proof}
Since $\widetilde{G_T}^{-1}$ and $\widetilde{G_N}^{-1}$ are explicit, we can 
provide explicit formulas for the tangent step
$d_T(X)=-\grad^g_{\StX} f(X)$
and the normal step $d_N(X)=-\nabla_g \psi(X)$	(corresponding to
$H(X)=\DD c(X)\DD c(X)^{*,g}$) in the metric $g$ of \cref{eqn:wen2o}.
\begin{proposition}
    With the metric $g$ defined in \cref{eqn:wen2o}, the tangent and normal
    directions of the landing algorithm \cref{eq:landing} read 
    \begin{align}
        \label{eqn:qof6v}
        d_T(X)& =-X(X\T X)^{-1}\skew((X\T X)^{-1}X\T \nabla_{\mathcal{E}}f(X)) - (\I_n-X(X\T
X)^{-1}X\T)\nabla_{\mathcal{E}}f(X), \\
\label{eqn:qvdbx}
d_N(X)&=-\frac{1}{2}X(\I_p-(X\T X)^{-1}).
\end{align}
\end{proposition}
\begin{proof}
    The inverse of the operators $\widetilde{G_T}(X)$ and $\widetilde{G_N}(X)$ read
\begin{equation}
    \begin{array}{rrcl}
        \widetilde{G_T}(X)^{-1}\,:\,& \mathrm{N}_X\StX^{\perp,\mathcal{E}} & \longrightarrow& \mathrm{T}_X\StX
        \\
        & X\Omega + \Delta & \longmapsto & X(X\T X)^{-1}\Omega +\Delta,
    \end{array}
\end{equation}
\begin{equation}
    \label{eqn:b4lel}
    \begin{array}{rrcl}
        \widetilde{G_N}(X)^{-1}\,:\,& \mathrm{T}_X\StX^{\perp,\mathcal{E}} & \longrightarrow&
\mathrm{N}_X\StX
        \\
        & XS & \longmapsto & X(X\T X)^{-1}S.
    \end{array}
\end{equation}
Formula \cref{eqn:qof6v} follow from \cref{eqn:6ow5x}, and combining \cref{eqn:b4lel}
with the item (ii) of \cref{prop:nw9q9} yields
\[
d_N(X)=-\frac{1}{2} X(X\T X)^{-1}(X\T X-\I_p)=-\frac{1}{2}X(\I_p-(X\T X)^{-1}).
\] 
\end{proof}
\begin{remark}
Here, $d_N(X)$ is both the Riemannian gradient
of the penalty function $\psi(X)$ and 
the
pseudoinverse direction \cref{eqn:zgiyo} with $H(X)=\Id_{\mathcal{E}}$.
This happens, according to
\cref{prop:pseudoinverse_E}, when $\widetilde{G_N}(X)=\DD c(X)^{*,\mathcal{E}}\DD
c(X)$ as operators from $\mathrm{N}_X\StX$ to $\mathrm{T}_X\StX^{\perp,\mathcal{E}}$.
This is incidentally indeed the case here: for any $S\in\Sym(p)$,
\[
    \begin{aligned}
    \DD c(X)^{*,\mathcal{E}}\DD c(X)[X(X\T X)^{-1}S] & =\DD
    c(X)^{*,\mathcal{E}}[\sym(X\T X(X\T X)^{-1}S)]
      =\DD
    c(X)^{*,\mathcal{E}}[S]\\
    & =XS
      =\widetilde{G_N}(X)[X(X\T X)^{-1}S].
    \end{aligned}
\] 
\end{remark}
\subsubsection{Construction of the tangent and normal steps based on the $\beta$-metric}
\label{subsub:eawup}

It turns out that the projectors $\Proj_X$ and $\Proj_X^{\perp}$ of
\cref{eqn:2znwx,eqn:0jkp9}
 are exactly the orthogonal projectors  for the family of $\beta$-metric $g^{\beta}$ 
 considered in 
\citep{goyens_geometric_2026}:
\begin{equation}
\label{eq:betametric}
g^{\beta}(\xi,\zeta) = \inner{\left(\I-(1-\beta)X(X\transpose X)\inv
X\transpose\right)\zeta (X\transpose X)\inv}{\xi},
\end{equation}
see \citep[Prop. 4]{goyens_geometric_2026}. The $\beta$-metric $g^{\beta}$ 
is the pull-back of a generalization of the
canonical metric \citep{edelman1998geometry} on the Stiefel manifold $\St$ to the
layered manifold $\StX$.
We observe, however, that  $g^{\beta}(\xi,\zeta)$  
does not coincide with the metric $g$ found in  in
\cref{eqn:wen2o}, even though they both generate the same family of normal spaces $\mathrm{N}_X\StX$.

\medskip 

In the next proposition, we show 
that the $\beta$-metric \cref{eq:betametric} is a special case of 
\cref{eqn:glmty} for 
particular choices of $G_T(X)$ and $G_N(X)$. Then, we verify that 
the operators $\widetilde{G_T}$ and $\widetilde{G_N}$ have explicit inverses, which
is the reason why this family of metrics leads to explicit formulas for the tangent and normal
terms~$d_T(X)$~and~$d_N(X)$. 

\begin{proposition}
    Let $\beta>0$. The $\beta$-metric \cref{eq:betametric} is the metric
    \cref{eqn:glmty} with
    $G_T(X)\,:\, \mathcal{E}\to \mathcal{E}$ and $G_N(X)\,:\, \mathcal{E}\to
    \mathcal{E}$ being the symmetric operators (with respect to the Euclidean inner
    product) defined by 
    \begin{align}
        \label{eqn:fd7io}
        G_T(X)[\xi]&:=((\I_n-X(X\T X)^{-1}X\T)\xi +\beta X(X\T X)^{-1}X\T \xi)(X\T
        X)^{-1},\\
        \label{eqn:xmg5y}
        G_N(X)[\xi]&:=\beta X(X\T X)^{-1}X\T\xi(X\T X)^{-1}.
    \end{align} 
    The operator $G_T(X)$ is positive definite on $\R^{n\times p}$ and 
    the operator $G_N(X)$ is positive definite  on $\mathrm{N}_X \StX$.
\end{proposition}
\begin{proof}
   Let us denote by $\Pi_X:=X(X\T X)^{-1}X\T$ the projection matrix onto the image
   space of $X$. We observe that for $\xi,\zeta \in \R^{n\times p}$,
   \[
   g^{\beta}(\xi,\zeta)=\beta\langle \Pi_X \xi (X\T X)^{-1},\Pi_X\zeta  \rangle 
   +\langle (\I_n-\Pi_X)\xi (X\T X)^{-1},(\I_n-\Pi_X)\zeta  \rangle.
   \] 
   Consequently, 
   \[
       \begin{aligned}
       g^{\beta}(\Proj_X[\xi],\Proj_X^{\perp}[\zeta])  & = \langle X(X\T X)^{-1}\skew(X\T
  \xi)(X\T X)^{-1},X(X\T X)^{-1}\sym(X\T \zeta)  \rangle \\
  & =\langle \skew(X\T \xi),(X\T X)^{-1}\sym(X\T \zeta)(X\T X)^{-1}  \rangle
  = 0,
       \end{aligned}
   \] 
   due to the Frobenius orthogonality between skew and symmetric matrices. 
   We have thus retrieved the fact that the tangent space $\mathrm{T}_X\StX$ 
   and the normal space $\mathrm{N}_X \StX$ of \cref{eqn:azgxt} are 
   orthogonal for the metric
   $g^{\beta}$. This implies in particular that 
   \[
   g^{\beta}(\xi,\zeta)=g^{\beta}(\Proj_X[\xi],\Proj_X[\zeta]) +
   g^{\beta}(\Proj_X^{\perp}[\xi],\Proj_X^{\perp}[\zeta]),
   \] 
with 
\begin{align*}
    &    \begin{aligned}
        g^{\beta}(\Proj_X[\xi]& ,\Proj_X[\zeta])\\
 & =\beta \langle \Pi_X \Proj_X[\xi] (X\T X)^{-1}, \Proj_X[\zeta]  \rangle
+\langle (\I_n-\Pi_X)\Proj_X[\xi](X\T X)^{-1},\Proj_X[\zeta]  \rangle\\
&= \langle G_T(X)\Proj_X[\xi],\Proj_X[\zeta]  \rangle,
    \end{aligned} \\
    &  \begin{aligned}
    g^{\beta}(\Proj_X^{\perp}[\xi],\Proj_X^{\perp}[\zeta])
 & =\beta \langle \Pi_X \Proj_X[\xi] (X\T X)^{-1}, \Proj_X[\zeta]
 \rangle_{\mathcal{E}}\\
&= \langle G_N(X)\Proj_X[\xi],\Proj_X[\zeta]  \rangle.
    \end{aligned}
\end{align*} 
Thus, $g^{\beta}$ is the metric $g$ of \cref{eqn:glmty} for $G_T$ and $G_N$ defined
by \cref{eqn:fd7io,eqn:xmg5y}. 
The positive definiteness of $G_T(X)$ and $G_N(X)$ are visible from 
\[
    \langle G_T(X)[\xi],\xi  \rangle
=\beta \norm{\Pi_X \xi (X\T X)^{-\frac{1}{2}}}^{2}
+\norm{(\I_n-\Pi_X)\xi (X\T X)^{-\frac{1}{2}}}^{2},
\] 
and, for all $S\in\Sym(p)$,
\[
\begin{aligned}
\langle G_N(X)& [X(X\T X)^{-1}S],X(X\T X)^{-1}S
\rangle \\
& =\beta \langle X(X\T X)^{-1}S(X\T X)^{-1},X(X\T X)^{-1}S  \rangle\\
 &=\beta \norm{(X\T X)^{-\frac{1}{2}} S (X\T X)^{-\frac{1}{2}}}^{2}.
\end{aligned}
\] 
\end{proof}
\begin{proposition}
 The inverse
    of the operators 
    \[
        \widetilde{G_T}(X)=\Proj_X^{*,\mathcal{E}}G_T(X)\Proj_X \text{ and }
\widetilde{G_N}(X)=(\Proj_X^{\perp})^{*,\mathcal{E}}G_N(X)\Proj_X^{\perp},
    \]
    associated to the $\beta$-metric $g^{\beta}$ 
    are explicit
    and given by 
    \begin{equation}
    \label{eqn:j7yzm}
        \begin{array}{rrcl}
        \widetilde{G_T}(X)^{-1}\colon& \mathrm{N}_X\StX^{\perp,\mathcal{E}} & \longrightarrow& \mathrm{T}_X\StX
        \\
        & X\Omega + \Delta & \longmapsto & \frac{1}{\beta}X\Omega X\T X +\Delta X\T X,
        \end{array}
    \end{equation}
    \begin{equation}
    \label{eqn:zn5gh}
    \begin{array}{rrcl}
        \widetilde{G_N}(X)^{-1}\colon & \mathrm{T}_X\StX^{\perp,\mathcal{E}} & \longrightarrow&
\mathrm{N}_X\StX
        \\
        & XS & \longmapsto & \frac{1}{\beta} XS(X\T X),
    \end{array}
    \end{equation}
    for any $\Omega\in\Skew(p)$, $\Delta\in\R^{n\times p}$ with $\Delta\T X=0$ and
    $S\in\Sym(p)$. Consequently, the tangent and normal directions of the landing
    algorithm \cref{eq:pseudoinverse_landing} with $H(X)=\DD c(X)\DD c(X)^{*,g^{\beta}}$ read 
\begin{align}
d_T(X) &= -\grad_{\StX}^{g^{\beta}} f(X) \notag\\
     &= -\frac{1}{\beta} X\,\skew\!\bigl((X^\top
     X)^{-1}X^\top\nabla_{\mathcal{E}}f(X)\bigr)\,X^\top X  \notag\\
     & \quad
        -(\I_n-X(X^\top X)^{-1}X^\top)\nabla_{\mathcal{E}}f(X)\,X^\top X,
         \label{eq:u} \\
         &= - \nabla_\calE f(X) X\T X - \dfrac{1}{2\beta}X \nabla_\calE f(X)\T X \notag\\
         & \quad + \left(\frac{1}{2\beta}-1\right) X(X\T X)\inv X\T \nabla_\calE f(X) X\T X,\notag \\[3ex]
d_N(X) &= -\nabla_{g^{\beta}}\psi(X) 
     = -\frac{1}{2\beta} X (X^\top X - \I_p) X^\top X.
        \label{eq:v}
\end{align}
\end{proposition}
\begin{proof}
    Using \cref{eqn:fd7io,eqn:xmg5y}, we see that $\widetilde{G_T}(X)$ and 
    $\widetilde{G_N}(X)$ are the operators 
\[
    \begin{array}{rrcl}
        \widetilde{G_T}(X)\,:\,&  \mathrm{T}_X\StX& \longrightarrow&
        \mathrm{N}_X\StX^{\perp,\mathcal{E}}\\
        & X(X\T X)^{-1}\Omega + \Delta & \longmapsto & \beta X(X\T X)^{-1} \Omega
        (X\T X)^{-1} +\Delta (X\T X)^{-1},
    \end{array}
\]
\[
    \begin{array}{rrcl}
        \widetilde{G_N}(X)\,:\,&  \mathrm{N}_X\StX& \longrightarrow&
        \mathrm{T}_X\StX^{\perp,\mathcal{E}}\\
        & X(X\T X)^{-1}S & \longmapsto & \beta X (X\T X)^{-1}S(X\T X)^{-1},
    \end{array}
\]
whose inverses are given by \cref{eqn:j7yzm,eqn:zn5gh}. Inserting these
formulas into \cref{eqn:fkr7g} and item (ii) of \cref{prop:nw9q9} yields the result. 
\end{proof}
The reader may verify that \cref{eq:u} coincides 
with \citep[Eq. 35]{goyens_geometric_2026}. In this earlier work, 
the normal term is taken as
the Euclidean gradient $d_N(X)=-\nabla_{\mathcal{E}}\psi(X)$, which corresponds to
\cref{eqn:zgiyo} with $H(X)=\DD c(X)\DD c(X)^{*,\mathcal{E}}$ (item (iii) of
\cref{prop:nw9q9}).

\appendix

\bibliographystyle{apalike} 
\bibliography{zotero}   

\begin{thebibliography}{}

\bibitem[Ablin and Peyr{\'e}, 2022]{ablin2022fast}
Ablin, P. and Peyr{\'e}, G. (2022).
\newblock Fast and accurate optimization on the orthogonal manifold without
  retraction.
\newblock In {\em International Conference on Artificial Intelligence and
  Statistics}, pages 5636--5657. PMLR.

\bibitem[Absil et~al., 2008]{absil2008optimization}
Absil, P.-A., Mahony, R., and Sepulchre, R. (2008).
\newblock {\em Optimization Algorithms on Matrix Manifolds}.
\newblock Princeton University Press.

\bibitem[Absil and Malick, 2012]{absil_projection-like_2012}
Absil, P.-A. and Malick, J. (2012).
\newblock Projection-like {{Retractions}} on {{Matrix Manifolds}}.
\newblock {\em SIAM Journal on Optimization}, 22(1):135--158.

\bibitem[Absil et~al., 2009]{absil_all_2009}
Absil, P.-A., Trumpf, J., Mahony, R., and Andrews, B. (2009).
\newblock All roads lead to {{Newton}}: {{Feasible}} second-order methods for
  equality-constrained optimization.
\newblock {\em https://sites.uclouvain.be/absil/2009.024}.

\bibitem[Allaire et~al., 2021]{allaire_chapter_2021}
Allaire, G., Dapogny, C., and Jouve, F. (2021).
\newblock Chapter 1 - {{Shape}} and topology optimization.
\newblock In Bonito, A. and Nochetto, R.~H., editors, {\em Handbook of
  {{Numerical Analysis}}}, volume~22 of {\em Geometric {{Partial Differential
  Equations}} - {{Part II}}}, pages 1--132. Elsevier.

\bibitem[Bai and Mei, 2018]{bai2018analysis}
Bai, Y. and Mei, S. (2018).
\newblock Analysis of sequential quadratic programming through the lens of
  riemannian optimization.
\newblock {\em arXiv preprint arXiv:1805.08756}.

\bibitem[Barbarosie et~al., 2020]{barbarosie_gradient-type_2020}
Barbarosie, C., Toader, A.-M., and Lopes, S. (2020).
\newblock A gradient-type algorithm for constrained optimization with
  application to microstructure optimization.
\newblock {\em Discrete and Continuous Dynamical Systems - B},
  25(5):1729--1755.

\bibitem[Berahas et~al., 2024]{berahas_stochastic_2024}
Berahas, A.~S., Curtis, F.~E., O'Neill, M.~J., and Robinson, D.~P. (2024).
\newblock A {{Stochastic Sequential Quadratic Optimization Algorithm}} for
  {{Nonlinear-Equality-Constrained Optimization}} with {{Rank-Deficient
  Jacobians}}.
\newblock {\em Mathematics of Operations Research}, 49(4):2212--2248.

\bibitem[Berahas et~al., 2021]{berahas2021sequential}
Berahas, A.~S., Curtis, F.~E., Robinson, D., and Zhou, B. (2021).
\newblock Sequential {{Quadratic Optimization}} for {{Nonlinear Equality
  Constrained Stochastic Optimization}}.
\newblock {\em SIAM Journal on Optimization}, 31(2):1352--1379.

\bibitem[Biros and Ghattas, 2005]{biros_parallel_2005}
Biros, G. and Ghattas, O. (2005).
\newblock Parallel {{Lagrange--Newton--Krylov--Schur Methods}} for
  {{PDE-Constrained Optimization}}. {{Part I}}: {{The Krylov--Schur Solver}}.
\newblock {\em SIAM Journal on Scientific Computing}, 27(2):687--713.

\bibitem[Boggs and Tolle, 1995]{boggs_sequential_1995}
Boggs, P.~T. and Tolle, J.~W. (1995).
\newblock Sequential {{Quadratic Programming}}.
\newblock {\em Acta Numerica}, 4:1--51.

\bibitem[Booker and Ong, 1971]{booker_multiple-constraint_1971}
Booker, A. and Ong, C.-y. (1971).
\newblock Multiple-constraint adaptive filtering.
\newblock {\em Geophysics}, 36(3):498--509.

\bibitem[Conn et~al., 1991]{conn_globally_1991}
Conn, A.~R., Gould, N. I.~M., and Toint, P. (1991).
\newblock A {{Globally Convergent Augmented Lagrangian Algorithm}} for
  {{Optimization}} with {{General Constraints}} and {{Simple Bounds}}.
\newblock {\em SIAM Journal on Numerical Analysis}, 28(2):545--572.

\bibitem[Curtis et~al., 2024]{curtis2024worst}
Curtis, F.~E., O'Neill, M.~J., and Robinson, D.~P. (2024).
\newblock Worst-case complexity of an sqp method for nonlinear equality
  constrained stochastic optimization.
\newblock {\em Mathematical Programming}, 205(1):431--483.

\bibitem[Edelman et~al., 1998]{edelman1998geometry}
Edelman, A., Arias, T.~A., and Smith, S.~T. (1998).
\newblock The geometry of algorithms with orthogonality constraints.
\newblock {\em SIAM journal on Matrix Analysis and Applications},
  20(2):303--353.

\bibitem[Evtushenko and Zhadan, 1994]{evtushenko_stable_1994}
Evtushenko, Y.~G. and Zhadan, V.~G. (1994).
\newblock Stable barrier-projection and barrier-{{Newton}} methods in linear
  programming.
\newblock {\em Computational Optimization and Applications}, 3(4):289--303.

\bibitem[Fang et~al., 2024]{fang_fully_2024}
Fang, Y., Na, S., Mahoney, M.~W., and Kolar, M. (2024).
\newblock Fully {{Stochastic Trust-Region Sequential Quadratic Programming}}
  for {{Equality-Constrained Optimization Problems}}.
\newblock {\em SIAM Journal on Optimization}, 34(2):2007--2037.

\bibitem[Feppon, 2024]{feppon_density-based_2024}
Feppon, F. (2024).
\newblock Density-based topology optimization with the {{Null Space
  Optimizer}}: A tutorial and a comparison.
\newblock {\em Structural and Multidisciplinary Optimization}, 67(1).

\bibitem[Feppon et~al., 2020a]{feppon_null_2020}
Feppon, F., Allaire, G., and Dapogny, C. (2020a).
\newblock Null space gradient flows for constrained optimization with
  applications to shape optimization.
\newblock {\em ESAIM: Control, Optimisation and Calculus of Variations},
  page~90.

\bibitem[Feppon et~al., 2020b]{feppon_topology_2020-1}
Feppon, F., Allaire, G., Dapogny, C., and Jolivet, P. (2020b).
\newblock Topology optimization of thermal fluid--structure systems using
  body-fitted meshes and parallel computing.
\newblock {\em Journal of Computational Physics}, 417:109574.

\bibitem[Feppon and Lermusiaux, 2019]{feppon_extrinsic_2019}
Feppon, F. and Lermusiaux, P. F.~J. (2019).
\newblock The {{Extrinsic Geometry}} of {{Dynamical Systems Tracking Nonlinear
  Matrix Projections}}.
\newblock {\em SIAM Journal on Matrix Analysis and Applications},
  40(2):814--844.

\bibitem[Fletcher et~al., 2002]{fletcher_global_2002}
Fletcher, R., Gould, N. I.~M., Leyffer, S., Toint, P.~L., and W{\"a}chter, A.
  (2002).
\newblock Global {{Convergence}} of a {{Trust-Region SQP-Filter Algorithm}} for
  {{General Nonlinear Programming}}.
\newblock {\em SIAM Journal on Optimization}, 13(3):635--659.

\bibitem[Frost, 1972]{frost_algorithm_1972}
Frost, O. (1972).
\newblock An algorithm for linearly constrained adaptive array processing.
\newblock {\em Proceedings of the IEEE}, 60(8):926--935.

\bibitem[Gangi and Byun, 1976]{gangi_corrective_1976}
Gangi, A.~F. and Byun, B.~S. (1976).
\newblock The corrective gradient projection method and related algorithms
  applied to seismic array processing.
\newblock {\em Geophysics}, 41(5):970--984.

\bibitem[Gao et~al., 2022]{gao_optimization_2022}
Gao, B., Vary, S., Ablin, P., and Absil, P.~A. (2022).
\newblock Optimization flows landing on the {{Stiefel}} manifold.
\newblock {\em IFAC-PapersOnLine}, 55(30):25--30.

\bibitem[Gharaei et~al., 2023]{gharaei_integrated_2023}
Gharaei, A., Amjadian, A., Amjadian, A., Shavandi, A., Hashemi, A., Taher, M.,
  and Mohamadi, N. (2023).
\newblock An integrated lot-sizing policy for the inventory management of
  constrained multi-level supply chains: Null-space method.
\newblock {\em International Journal of Systems Science: Operations \&
  Logistics}, 10(1):2083254.

\bibitem[Gill et~al., 2015]{gill_users_2015}
Gill, P., Wong, E., Murray, W., and Saunders, M. (2015).
\newblock User's {{Guide}} for {{SNOPT Version}} 7.4: {{Software}} for
  {{Large-Scale Nonlinear Programming}}.

\bibitem[Gill et~al., 2002]{gill_snopt_2002}
Gill, P.~E., Murray, W., and Saunders, M.~A. (2002).
\newblock {{SNOPT}}: {{An SQP Algorithm}} for {{Large-Scale Constrained
  Optimization}}.
\newblock {\em SIAM Journal on Optimization}, 12(4):979--1006.

\bibitem[Gill and Wong, 2012]{gill_sequential_2012}
Gill, P.~E. and Wong, E. (2012).
\newblock Sequential {{Quadratic Programming Methods}}.
\newblock In Lee, J. and Leyffer, S., editors, {\em Mixed {{Integer Nonlinear
  Programming}}}, pages 147--224, New York, NY. Springer.

\bibitem[Goyens et~al., 2026]{goyens_geometric_2026}
Goyens, F., Absil, P.-A., and Feppon, F. (2026).
\newblock Geometric {{Design}} of~the~{{Tangent Term}} in~{{Landing
  Algorithms}} for~{{Orthogonality Constraints}}.
\newblock In Nielsen, F. and Barbaresco, F., editors, {\em Geometric
  {{Science}} of {{Information}}}, pages 133--141, Cham. Springer Nature
  Switzerland.

\bibitem[Goyens et~al., 2024]{goyens2024computing}
Goyens, F., Eftekhari, A., and Boumal, N. (2024).
\newblock Computing second-order points under equality constraints: revisiting
  fletcher's augmented lagrangian.
\newblock {\em Journal of Optimization Theory and Applications},
  201(3):1198--1228.

\bibitem[Henrot and Pierre, 2018]{henrot_shape_2018}
Henrot, A. and Pierre, M. (2018).
\newblock {\em Shape Variation and Optimization}, volume~28 of {\em {{EMS
  Tracts}} in {{Mathematics}}}.
\newblock European Mathematical Society (EMS), Z\"urich.

\bibitem[Hu et~al., 2025]{hu2025ostquant}
Hu, X., Cheng, Y., Yang, D., Xu, Z., Yuan, Z., Yu, J., Xu, C., Jiang, Z., and
  Zhou, S. (2025).
\newblock Ostquant: {{Refining}} large language model quantization with
  orthogonal and scaling transformations for better distribution fitting.
\newblock {\em arXiv preprint arXiv:2501.13987}.

\bibitem[H{\"u}per et~al., 2021]{huper_lagrangian_2021}
H{\"u}per, K., Markina, I., and Leite, F.~S. (2021).
\newblock A {{Lagrangian}} approach to extremal curves on {{Stiefel}}
  manifolds.
\newblock {\em Journal of Geometric Mechanics (JGM)}.

\bibitem[Jongen and Stein, 2003]{jongen_complexity_2003}
Jongen, H.~T. and Stein, O. (2003).
\newblock On the {{Complexity}} of {{Equalizing Inequalities}}.
\newblock {\em Journal of Global Optimization}, 27(4):367--374.

\bibitem[Jongen and Stein, 2004]{jongen_constrained_2004}
Jongen, H.~{\relax Th}. and Stein, O. (2004).
\newblock Constrained {{Global Optimization}}: {{Adaptive Gradient Flows}}.
\newblock In Floudas, C.~A. and Pardalos, P., editors, {\em Frontiers in
  {{Global Optimization}}}, pages 223--236, Boston, MA. Springer US.

\bibitem[Li and Zhou, 2018]{li2018spectral}
Li, Z.-Y. and Zhou, B. (2018).
\newblock Spectral decomposition based solutions to the matrix equation.
\newblock {\em IET Control Theory \& Applications}, 12(1):119--128.

\bibitem[Liu and Yuan, 2010]{liu_null-space_2010}
Liu, X. and Yuan, Y. (2010).
\newblock A null-space primal-dual interior-point algorithm for nonlinear
  optimization with nice convergence properties.
\newblock {\em Mathematical programming}, 125(1):163--193.

\bibitem[Luenberger, 1972]{luenberger_gradient_1972}
Luenberger, D.~G. (1972).
\newblock The {{Gradient Projection Method Along Geodesics}}.
\newblock {\em Management Science}, 18(11):620--631.

\bibitem[Mataigne et~al., 2025]{mataigne_efficient_2025}
Mataigne, S., Zimmermann, R., and Miolane, N. (2025).
\newblock An {{Efficient Algorithm}} for the {{Riemannian Logarithm}} on the
  {{Stiefel Manifold}} for a {{Family}} of {{Riemannian Metrics}}.
\newblock {\em SIAM Journal on Matrix Analysis and Applications},
  46(2):879--905.

\bibitem[Miller and Malick, 2005]{miller_newton_2005}
Miller, S.~A. and Malick, J. (2005).
\newblock Newton methods for nonsmooth convex minimization: Connections among
  -{{Lagrangian}}, {{Riemannian Newton}} and {{SQP}} methods.
\newblock {\em Mathematical Programming}, 104(2):609--633.

\bibitem[Mishra and Sepulchre, 2016]{mishra2016riemannian}
Mishra, B. and Sepulchre, R. (2016).
\newblock Riemannian preconditioning.
\newblock {\em SIAM Journal on Optimization}, 26(1):635--660.

\bibitem[Nie, 2004]{nie_null_2004}
Nie, P.-y. (2004).
\newblock A null space method for solving system of equations.
\newblock {\em Applied Mathematics and computation}, 149(1):215--226.

\bibitem[Nocedal and Overton, 1985]{nocedal_projected_1985}
Nocedal, J. and Overton, M.~L. (1985).
\newblock Projected {{Hessian Updating Algorithms}} for {{Nonlinearly
  Constrained Optimization}}.
\newblock {\em SIAM Journal on Numerical Analysis}, 22(5):821--850.

\bibitem[Nocedal et~al., 2014]{nocedal_interior_2014}
Nocedal, J., {\"O}ztoprak, F., and Waltz, R.~A. (2014).
\newblock An interior point method for nonlinear programming with infeasibility
  detection capabilities.
\newblock {\em Optimization Methods and Software}, 29(4):837--854.

\bibitem[Nocedal and Wright, 2006]{nocedal2006numerical}
Nocedal, J. and Wright, S.~J. (2006).
\newblock {\em Numerical optimization}.
\newblock Springer.

\bibitem[Obara et~al., 2022]{obara_sequential_2022}
Obara, M., Okuno, T., and Takeda, A. (2022).
\newblock Sequential {{Quadratic Optimization}} for {{Nonlinear Optimization
  Problems}} on {{Riemannian Manifolds}}.
\newblock {\em SIAM Journal on Optimization}, 32(2):822--853.

\bibitem[Rosen, 1960]{rosen_gradient_1960}
Rosen, J.~B. (1960).
\newblock The {{Gradient Projection Method}} for {{Nonlinear Programming}}.
  {{Part I}}. {{Linear Constraints}}.
\newblock {\em Journal of the Society for Industrial and Applied Mathematics},
  8(1):181--217.

\bibitem[Rosen, 1961]{rosen_gradient_1961}
Rosen, J.~B. (1961).
\newblock The {{Gradient Projection Method}} for {{Nonlinear Programming}}.
  {{Part II}}. {{Nonlinear Constraints}}.
\newblock {\em Journal of the Society for Industrial and Applied Mathematics},
  9(4):514--532.

\bibitem[Schechtman et~al., 2023]{schechtman_orthogonal_2023}
Schechtman, S., Tiapkin, D., Muehlebach, M., and Moulines, {\'E}. (2023).
\newblock Orthogonal {{Directions Constrained Gradient Method}}: From
  non-linear equality constraints to {{Stiefel}} manifold.
\newblock In {\em Proceedings of {{Thirty Sixth Conference}} on {{Learning
  Theory}}}, pages 1228--1258. PMLR.

\bibitem[Schropp and Singer, 2000]{schropp_dynamical_2000}
Schropp, J. and Singer, I. (2000).
\newblock A dynamical systems approach to constrained minimization.
\newblock {\em Numerical Functional Analysis and Optimization},
  21(3-4):537--551.

\bibitem[Shi and Wang, 2025]{shi_adaptive_2025}
Shi, Q. and Wang, X. (2025).
\newblock Adaptive directional decomposition methods for nonconvex constrained
  optimization.
\newblock {\em arXiv preprint arXiv:2511.03210}.

\bibitem[Shikhman and Stein, 2009]{shikhman_constrained_2009}
Shikhman, V. and Stein, O. (2009).
\newblock Constrained {{Optimization}}: {{Projected Gradient Flows}}.
\newblock {\em Journal of Optimization Theory and Applications},
  140(1):117--130.

\bibitem[Song et~al., 2025]{song2025distributed}
Song, Y., Li, P., Gao, B., and Yuan, K. (2025).
\newblock Distributed retraction-free and communication-efficient optimization
  on the stiefel manifold.
\newblock {\em arXiv preprint arXiv:2506.02879}.

\bibitem[Tanabe, 1974]{tanabe1974algorithm}
Tanabe, K. (1974).
\newblock An algorithm for constrained maximization in nonlinear programming.
\newblock {\em Journal of the Operations Research Society of Japan},
  17:184--201.

\bibitem[Tanabe, 1980]{tanabe_geometric_1980}
Tanabe, K. (1980).
\newblock A geometric method in nonlinear programming.
\newblock {\em Journal of Optimization Theory and Applications},
  30(2):181--210.

\bibitem[Tapia, 1977]{tapia_diagonalized_1977}
Tapia, R.~A. (1977).
\newblock Diagonalized multiplier methods and quasi-{{Newton}} methods for
  constrained optimization.
\newblock {\em Journal of Optimization Theory and Applications},
  22(2):135--194.

\bibitem[Vary et~al., 2024]{vary_optimization_2024}
Vary, S., Ablin, P., Gao, B., and Absil, P.-A. (2024).
\newblock Optimization without {{Retraction}} on the {{Random Generalized
  Stiefel Manifold}}.
\newblock In {\em Proceedings of the 41st {{International Conference}} on
  {{Machine Learning}}}, pages 49226--49248. PMLR.

\bibitem[W{\"a}chter and Biegler, 2005]{wachter_line_2005}
W{\"a}chter, A. and Biegler, L.~T. (2005).
\newblock Line {{Search Filter Methods}} for {{Nonlinear Programming}}:
  {{Motivation}} and {{Global Convergence}}.
\newblock {\em SIAM Journal on Optimization}, 16(1):1--31.

\bibitem[Wegert et~al., 2023]{wegert_hilbertian_2023}
Wegert, Z.~J., Roberts, A.~P., and Challis, V.~J. (2023).
\newblock A {{Hilbertian}} projection method for constrained level set-based
  topology optimisation.
\newblock {\em Structural and Multidisciplinary Optimization}, 66(9):204.

\bibitem[Yamashita, 1980]{yamashita_differential_1980}
Yamashita, H. (1980).
\newblock A differential equation approach to nonlinear programming.
\newblock {\em Mathematical Programming}, 18(1):155--168.

\bibitem[Yuan, 2001]{yuan2001null}
Yuan, Y.-x. (2001).
\newblock A null space algorithm for constrained optimization.
\newblock {\em Advances in Scientific Computing, Science Press, Beijing}, pages
  210--218.

\bibitem[Yulin and Xiaoming, 2004]{yulin_level_2004}
Yulin, M. and Xiaoming, W. (2004).
\newblock A level set method for structural topology optimization with
  multi-constraints and multi-materials.
\newblock {\em Acta Mechanica Sinica}, 20(5):507--518.

\bibitem[Zhang et~al., 2024]{zhang2024retractionfree}
Zhang, Y., Hu, J., Cui, J., Lin, L., Wen, Z., and Li, Q. (2024).
\newblock Retraction-free optimization over the {{Stiefel}} manifold with
  application to the {{LoRA}} fine-tuning.
\newblock {\em \url{https://openreview.net/forum?id=GP30inajOt}}.

\end{thebibliography}

\end{document}